\PassOptionsToPackage{prologue, dvipsnames}{xcolor}
\documentclass[final,hidelinks,onefignum,onetabnum]{siamart250211}

\usepackage{lipsum}
\usepackage{amsfonts}
\usepackage{graphicx}
\usepackage{epstopdf}
\usepackage{algorithmic}
\ifpdf
  \DeclareGraphicsExtensions{.eps,.pdf,.png,.jpg}
\else
  \DeclareGraphicsExtensions{.eps}
\fi

\usepackage[caption=false]{subfig}
\usepackage{booktabs}
\usepackage{cases}
\usepackage{color}
\definecolor{myblue}{rgb}{0,0,0.5}
\definecolor{mygreen}{rgb}{0,0.5,0}
\definecolor{myred}{rgb}{0.5,0,0}

\crefalias{framework}{algorithm}
\crefname{framework}{algorithmic framework}{algorithmic frameworks}
\Crefname{framework}{Algorithmic Framework}{Algorithmic Frameworks}
\creflabelformat{framework}{#2#1#3}

\usepackage[normalem]{ulem}
\usepackage{enumerate}

\usepackage{multirow}
\usepackage{hyperref}
\usepackage{url}
\usepackage{listings}
\usepackage{float}
\usepackage{ulem}

\usepackage{cleveref}

\usepackage{threeparttable}
\usepackage{lscape}

\newcommand{\RNum}[1]{\uppercase\expandafter{\romannumeral #1\relax}}

\def \[{\begin{equation}}
\def \]{\end{equation}}
\def \]{\end{equation}}

\usepackage{tikz,pgfplots,pgf}
\pgfplotsset{compat=1.17}
\usetikzlibrary{matrix,shapes,arrows,positioning,
	arrows.meta,
	calc,
	fit,
	quotes
}
\usepackage{etoolbox} 
\usepackage{adjustbox}
\usepackage{anyfontsize}

\tikzset{box/.style ={
		rectangle,
		rounded corners =5pt,
		minimum width =50pt,
		minimum height =20pt,
		inner sep=5pt,
		draw=blue}
}

\tikzset{zbox/.style ={
		rectangle,
		minimum width =50pt,
		minimum height =20pt,
		inner sep=5pt,
		draw=black}
}

\tikzset{ball/.style ={
		circle,
		minimum width =20pt,
		minimum height =20pt,
		inner sep=0.1pt,
		draw=blue}
}

\tikzset{global scale/.style={
		scale=#1,
		every node/.append style={scale=#1}
	}
}

\makeatletter
\tikzset{
	meta box/.style={
		draw,
		black,
		very thick,
		text centered
	},
	inputbox/.style={
		meta box,
		rectangle,
		rounded corners,
		inner sep=5pt,
		minimum height=2em,
		minimum width=10em,
		align=center,
		text width=10em
	},
	inputbox2/.style={
		meta box,
		rectangle,
		rounded corners,
		inner sep=5pt,
		minimum height=2em,
		minimum width=12em,
		align=center,
		text width=12em
	},
	box1/.style={
		meta box,
		rectangle,
		rounded corners,
		inner sep=5pt,
		minimum height=2em,
		minimum width=3em,
		align=center,
		text width=3em
	},
	hcbox/.style={
		meta box,
		rectangle,
		inner sep=5pt,
		minimum height=2em,
		minimum width=6em,
		align=center,
		text width=6em
	},
	round box/.style={
		meta box,
		circle
	},
	every fit/.style={
		draw,
		thick,
		dashed,
		gray,
		inner sep=10pt
	}
}

\newcommand\tikz@expand@dimen[2]{\tikzset{minimum #2=#1}}
\tikzset{
	add dimen/.code 2 args={%
		\pgfkeysgetvalue{/pgf/minimum #1}\tikz@dimen@min
		\expandafter\tikz@expand@dimen\expandafter{\tikz@dimen@min + #2 * 2em}{#1}%
	},
	wider/.style={add dimen={width}{#1}},
	higher/.style={add dimen={height}{#1}},
}
\makeatother

\newcommand*{\num}{pi}
\tikzset{elegant/.style={smooth,thick,samples=50,cyan}}
\tikzset{eaxis/.style={->,>=stealth}}

\newcommand{\cS}{\mathcal{S}}

\newcommand{\cL}{\mathcal{L}}
\newcommand{\cP}{\mathcal{P}}
\newcommand{\cQ}{\mathcal{Q}}
\newcommand{\cF}{\mathcal{F}}
\newcommand{\cA}{\mathcal{A}}

\usepackage{amsopn}

\DeclareMathOperator{\prox}{prox}
\DeclareMathOperator{\argmin}{argmin}

\newsiamremark{remark}{Remark}
\newsiamremark{hypothesis}{Hypothesis}
\crefname{hypothesis}{Hypothesis}{Hypotheses}
\newsiamthm{claim}{Claim}

\headers{Operator Learning for Optimal Control of Nonlinear PDEs}{Y. Song, X. Yuan, H. Yue, and T. Zeng}

\title{An Operator Learning Approach to Nonsmooth Optimal Control of Nonlinear PDEs\thanks{April 23, 2025.
\funding{
The work of X. Yuan was supported by the GRF RGC (No. 17309824). The work of H. Yue was supported by the National Natural Science Foundation of China (No. 12301399).
The work of T. Zeng was supported by Hong Kong PhD Fellowship Scheme.}}}

\author{Yongcun Song\thanks{Department of Mathematics, City University of Hong Kong, Hong Kong, China
  (\email{ysong307@gmail.com}).}
\and Xiaoming Yuan\thanks{Department of Mathematics, The University of Hong Kong, Hong Kong, China
  (\email{xmyuan@hku.hk}).}
\and Hangrui Yue\thanks{School of Mathematical Sciences, Nankai University, Tianjin, China
  (\email{yuehangrui@gmail.com}).}
\and Tianyou Zeng\thanks{Department of Mathematics, The University of Hong Kong, Hong Kong, China
  (\email{logic@connect.hku.hk})}
}

\usepackage{amsopn}

\ifpdf
\hypersetup{
  pdftitle={An Operator Learning Approach to Nonsmooth Optimal Control of Nonlinear PDEs},
  pdfauthor={Y. Song, X. Yuan, H. Yue, and T. Zeng}
}
\fi

\begin{document}

\maketitle

\begin{abstract}
Optimal control problems with nonsmooth objectives and nonlinear partial differential equation (PDE) constraints are challenging, mainly because of the underlying nonsmooth and nonconvex structures and the demanding computational cost for solving multiple high-dimensional and ill-conditioned systems after mesh-based discretization.
To mitigate these challenges numerically, we propose an operator learning approach in combination with an effective primal-dual optimization idea which can decouple the treatment of the control and state variables so that each of the resulting iterations only requires solving two PDEs. Our main purpose is to construct neural surrogate models for the involved PDEs by operator learning, allowing the solution of a PDE to be obtained with only a forward pass of the neural network. The resulting algorithmic framework offers a hybrid approach that combines the efficiency and generalization of operator learning with the model-based nature and structure-friendly efficiency of primal-dual-based algorithms. The primal-dual-based operator learning approach offers numerical methods that are mesh-free, easy to implement, and adaptable to various optimal control problems with nonlinear PDEs. It is notable that the neural surrogate models can be reused across iterations and parameter settings, hence retraining of neural networks can be avoided and computational cost can be substantially alleviated. We affirmatively validate the efficiency of the primal-dual-based operator learning approach across a range of typical optimal control problems with nonlinear PDEs.
\end{abstract}



\section{Introduction}

We consider a general optimal control model with a nonlinear partial differential equation (PDE) constraint
\begin{equation}\label{eq:basic-problem}
	\begin{aligned}
		& \min_{u\in {U}, y\in {Y}} \quad   J(y, u) := \frac{1}{2}\| y-y_d \|^2_Y+\frac{\alpha}{2}\|u\|_U^2+ \theta(u) \\
		& \quad {\hbox{s.t.}}~ \qquad  y=S(u),
	\end{aligned}
\end{equation}
where $U$ and $Y$ are Hilbert spaces with corresponding norms $\lVert \cdot \rVert _U$ and $\lVert \cdot \rVert _Y$, respectively; $u \in U$ and $y \in Y$ are the control and state variables, respectively; $y_d \in Y$ represents a given desired state; $y = S(u)$ represents a state equation with $S: U \to Y$ being the corresponding solution operator, and $\alpha > 0$ is a regularization parameter.
In (\ref{eq:basic-problem}), the regularization functional $\theta: U \to \mathbb{R} \cup \{+\infty\}$ is usually imposed for additional constraints on the control variable $u$, such as boundedness~\cite{lions1971optimal,troltzsch2010optimal}, sparsity~\cite{stadler2009elliptic}, and discontinuity~\cite{elvetun2016split}. The model (\ref{eq:basic-problem}) captures specific optimal control problems in various domains such as physics, chemistry, biology, as well as modern engineering disciplines like digital twins and physics engines, see e.g., \cite{antil2024mathematical,de2015numerical,glowinski1994exact,glowinski1995exact,hinze2008optimization,lions1971optimal,troltzsch2010optimal}.

We focus on the difficult case where $S$ is \emph{nonlinear} and $\theta$ is \emph{nonsmooth}, hence problem \eqref{eq:basic-problem} is generally nonconvex and nonsmooth. For instance, the abstract state equation $y = S(u)$ can be specified as different PDEs, including Burgers equations \cite{de2004comparison,volkwein2000application}, semilinear and quasilinear parabolic equations \cite{hoppe2021convergence,langer2020unstructured,troltzsch2010optimal}, the Allen-Cahn equation \cite{colli2015optimal}, the Navier-Stokes equation \cite{hintermuller2006sqp}, etc.
The function $\theta(u)$ can be chosen as the indicator function of an admissible set \cite{song2023accelerated,song2024admm}, an $L^1$ regularization term \cite{stadler2009elliptic}, or a total variation regularization term \cite{elvetun2016split}. The control variable $u$ can be a distributed control or a boundary control \cite{kroner2011semismooth}. The concurrent presence of nonsmoothness in $\theta(u)$ and nonlinearity in $y=S(u)$ poses significant numerical challenges. Specifically, the nonsmoothness of $\theta(u)$ basically hinders direct applications of most of the canonical algorithms such as gradient descent methods, conjugate gradient methods, and quasi-Newton methods. Furthermore, the nonlinearity of $y=S(u)$, combined with the high-dimensionality and ill-conditioned nature of the algebraic system resulting from numerical discretization, also leads to severe numerical difficulties.

To solve such a challenging and versatile problem, we aim to propose a general \emph{operator learning approach} for (\ref{eq:basic-problem}) which can be specified as concrete and efficient algorithms for (\ref{eq:basic-problem}) with specific PDE constraints.
The operator learning approach employs the primal-dual decomposition idea to decouple \eqref{eq:basic-problem} into parametric PDE subproblems, namely the state and adjoint equations, and leverages pretrained neural networks to approximate their solution operators directly. With this feature, algorithms specified by the operator learning approach do not require solving any algebraic system from numerical discretization or retraining any neural network, hence avoiding heavy computation load.
Additionally, the resulting algorithms can be reused without retraining of any neural network even when the parameters $y_d$, $\alpha$, and $\theta$ defining \eqref{eq:basic-problem} are changed.
This capability is particularly suitable for several important scenarios requiring repeatedly solving problems with varying parameters, such as the real-time online control tasks in \cite{behrens2014real,biegler2007real} and the numerical study of optimal control problems in \cite{berggren1996computational,glowinski1994exact}.

\subsection{Traditional numerical methods}\label{sec:trad-num-method}

In the literature, some numerical methods have been developed for solving various specific cases of \eqref{eq:basic-problem} or more general optimal control models, such as semismooth Newton (SSN) methods \cite{casas2024convergence,de2004comparison,de2005semismooth,schindele2017proximal}, sequential quadratic programming (SQP) methods \cite{griesse2010local,heinkenschloss1999analysis,hintermuller2006sqp,hoppe2021convergence,troltzsch1999optimal}, and interior point methods (IPMs) \cite{grote2014inexact,mittelmann2000solving,weiser2004function}.
In particular, SSN methods have been studied in \cite{casas2024convergence} for sparse optimal control of semilinear elliptic equations, in \cite{de2005semismooth} for control constrained boundary optimal control of the Navier-Stokes equations, in \cite{schindele2017proximal} for sparse bilinear control of parabolic equations, in \cite{de2004comparison} for control constrained optimal control of Burgers equations,  just to name a few.  Various SQP methods have also been extensively studied in the past years for solving the nonsmooth optimal control of a wide range of  nonlinear PDEs, see e.g., \cite{hoppe2021convergence} for quasilinear parabolic equations, \cite{troltzsch1999optimal} for semilinear parabolic equations, \cite{hintermuller2006sqp} for time-dependent Navier-Stokes equations, \cite{griesse2010local} for semilinear elliptic equations, and \cite{heinkenschloss1999analysis} for a phase field equation.  Moreover, IPMs have been designed for the optimal control of semilinear elliptic equations in \cite{mittelmann2000solving} and for the optimal control of a general class of nonlinear PDEs in \cite{grote2014inexact,weiser2004function}, respectively.
Notably, all the above algorithms result in subproblems that usually require to be solved iteratively by deliberately designed algorithms, leading to nested (multi-layer) iterations and consequently demanding computation loads.

\subsection{Deep learning methods for PDEs}\label{sec:dl-pde}

Recently, deep learning methods have found widespread applications in various scientific and engineering fields, owing to the powerful representation capabilities~\cite{cybenko1989approximation,hornik1989multilayer,kidger2020universal} and generalization abilities~\cite{kawaguchi2017generalization,neyshabur2017exploring} of deep neural networks (DNNs).
In particular, numerous deep learning methods for solving PDEs have been developed in recent years, as evidenced by~\cite{e2018deep,li2021fourier,lu2019deeponet,raissi2019physics,sirignano2018dgm} and the references therein.
Some deep learning methods approximate the solution of a given PDE by training DNNs, such as the deep Ritz method \cite{e2018deep}, the deep Galerkin method \cite{sirignano2018dgm}, and physics-informed neural networks (PINNs) \cite{lu2021physics,raissi2019physics}. Compared with traditional numerical methods for PDEs such as finite element methods (FEMs) or finite difference methods (FDMs), the above deep learning methods are typically mesh-free, easy to implement, and flexible in solving various PDEs, especially for high-dimensional problems or those with complex geometries. Despite their applications in diverse fields, these methods require training a new DNN at each iteration for the solution of a PDE but with different parameters (e.g., initial/boundary condition, source term, or coefficients), which can be computationally expensive and time-costing.

Alternatively, operator learning methods (see~\cite{li2021fourier,lu2019deeponet}) apply DNNs to approximate the solution operator of a given family of parameterized PDEs, mapping the PDE parameters to the solution. Once the neural solution operator is learned, it serves as a surrogate model to the shared solution operator, and only one forward pass of the DNN is required to solve any PDE within the given family.
Representative operator learning methods for solving PDEs include the Deep Operator Networks (DeepONets) \cite{lu2019deeponet}, the MIONet \cite{jin2022mionet}, the physic-informed DeepONets \cite{wang2021learning}, the Fourier Neural Operator (FNO) \cite{li2021fourier}, the Graph Neural Operator (GNO) \cite{li2020neural}, the random feature model \cite{nelsen2021random}, the PCA-Net \cite{bhattacharya2021model}, and the Laplace Neural Operator (LNO) \cite{cao2023lno}.
These methods provide DNN structures that are universal for approximating operators between Banach spaces \cite{chen1995universal,kovachki2023neural}, and demonstrate impressive capabilities in capturing the nonlinear relation between parameters and solutions of PDEs. With satisfactory accuracy, numerical efficiency, and generalization ability, operator learning methods provide competitive alternatives to traditional methods, especially for scenarios that require repetitive yet expensive solves; see e.g., \cite{hwang2022solving,song2023accelerated,wang2021fast}.

\subsection{Deep learning methods for optimal control}\label{sec:dl-oc}

Deep learning techniques have also been applied to various optimal control problems, such as the PINNs specialized for smooth PDE-constrained optimization problems in \cite{barry2022physics,dai2023solving,lu2021physics,mowlavi2023optimal} and the operator learning methods for smooth optimal control problems in \cite{hwang2022solving,wang2021fast} .
These methods approximates the optimal control variable by a DNN and, in particular, require training different neural networks for different problems.
Our recent works \cite{lai2023hard,song2023accelerated,song2024admm} have investigated deep learning techniques for nonsmooth optimal control problems, either by directly solving the corresponding optimality systems \cite{lai2023hard} or by combining deep learning with traditional optimization algorithms \cite{song2023accelerated,song2024admm}. More specifically, in \cite{lai2023hard}, we studied the control constrained interface optimal control problem, and adopted the hard-constraint PINNs and a nonsmooth training loss function to solve its optimality system, which is a system of coupled PDEs. In \cite{song2023accelerated}, we focused on optimal control problems with \emph{linear} PDE constraints, and combined the well-studied primal-dual method in \cite{chambolle2011first, he2012convergence} (see also, e.g., \cite{clason2019acceleration,he2017algorithmic, valkonen2014primal} for various extensions in different settings) with operator learning techniques.
Notably, neither \cite{lai2023hard} nor \cite{song2024admm} can be directly applied to the more general model \eqref{eq:basic-problem}.

\subsection{ADMM for \texorpdfstring{\eqref{eq:basic-problem}}{(1.1)}}\label{sec:admm}

We can easily extend the ADMM in \cite{glowinski2022application} and obtain the following iterative scheme for solving \eqref{eq:basic-problem}:

\begin{subnumcases}{\label{eq:admm-pinn}}
		u^{k+1} = \underset{u \in U}{\argmin} \left\{ \frac{1}{2}\| y - y_d \|^2_Y + \frac{\alpha}{2}\|u\|_U^2 + \frac{\gamma}{2} \left\| u - z^k-\frac{\lambda^k}{\gamma} \right\|_U^2 ~\Bigm\vert~  y = S(u) \right\}, \hspace{-4.5em} \label{eq:admm-pinn-u} \\
		z^{k+1} = \underset{z \in U}{\argmin} \left\{ \theta(z) + \frac{\gamma}{2} \left\| z - \left(u^{k+1} - \frac{\lambda^k}{\gamma} \right) \right\|_U^2 \right\}, \label{eq:admm-pinn-z} \\
		\lambda^{k+1} = \lambda^k - \gamma (u^{k+1} - z^{k+1}) \label{eq:admm-pinn-lambda},
\end{subnumcases}
where $\gamma>0$ is a penalty parameter, $z \in U$ is an auxiliary variable, and $\lambda \in U$ is the dual variable. As an ADMM type algorithm originated from \cite{glowinski1975sur}, the PDE constraint $y = S(u)$ and the nonsmooth regularization $\theta$ are treated separately in the \emph{smooth} optimal control problem \eqref{eq:admm-pinn-u} and the simple optimization problem \eqref{eq:admm-pinn-z}, respectively.

To implement the ADMM (\ref{eq:admm-pinn}), it is necessary to discuss how to solve the subproblem \eqref{eq:admm-pinn-u}. For instance, it can be discretized by FEMs and then solved iteratively by gradient type algorithms; the resulting algorithm is labeled as ADMM-FEM hereafter. Recall that we combined the classic ADMM \cite{glowinski1975sur} with PINNs in \cite{song2024admm}, and proposed the ADMM-PINNs algorithmic framework for a general class of nonsmooth PDE-constrained optimization problems that cover \eqref{eq:basic-problem} as a special case. Thus, we can also apply PINNs to solve the subproblem \eqref{eq:admm-pinn-u} and readily obtain the ADMM-PINNs algorithm for \eqref{eq:basic-problem}. As mentioned in Section \ref{sec:dl-oc}, the ADMM-PINNs algorithm for \eqref{eq:basic-problem} also requires retraining neural networks iteratively for solving the subproblem \eqref{eq:admm-pinn-u} with different $z^k$ and $\lambda^k$.

\subsection{Primal-dual methods for \texorpdfstring{\eqref{eq:basic-problem}}{(1.1)}}\label{sec:pd-decoup}

The primal-dual methods are initiated in \cite{chambolle2011first, he2012convergence} for a class of optimization problems with linear constraints and extended to model \eqref{eq:basic-problem} with nonlinear constraint in some works such as \cite{clason2017primal,valkonen2014primal}. To illustrate, we first introduce functionals $F: Y \to \mathbb{R}$ and $G: U \to \mathbb{R} \cup \{+\infty\}$ defined respectively by
\begin{equation*}
	F(y) = \frac{1}{2}\| y-y_d \|^2_Y, \quad \text{and} \quad G(u) = \frac{\alpha}{2}\|u\|_U^2+ \theta(u).
\end{equation*}
By introducing a dual variable $p \in Y$, problem~\eqref{eq:basic-problem} can be reformulated as the saddle point problem~\cite{rockafellar1970convex}
\begin{equation}\label{eq:basic-saddle}
	\begin{aligned}
		\min_{u\in {U}} \max_{p\in {Y}} \quad G(u) + \langle p, S(u) \rangle_Y - F^*(p),
	\end{aligned}
\end{equation}
where $\langle \cdot, \cdot \rangle_Y$ denotes the inner product in $Y$, and $F^*(p):={\sup}_{z\in{Y}}\{\langle z,p \rangle_{Y}-F(z)\}$ is the convex conjugate of $F(y)$, which can be specified as
$F^*(p)=\frac{1}{2}\|p\|_Y^2+\langle p, y_d\rangle_Y$.
Implementing the primal-dual idea in \cite{chambolle2011first,clason2017primal} to \eqref{eq:basic-saddle}, we readily obtain the following iterative scheme:
\begin{equation}\label{eq:pdhg-nonlinear}
	\left\{~
	\begin{aligned}
		& u^{k+1} = \prox_{\tau G} (u^k - \tau (S'(u^k))^* p^k), \\
		& p^{k+1} = \prox_{\sigma F^*} (p^k + \sigma S (2u^{k+1}-u^k)),
	\end{aligned}
	\right.
\end{equation}
where $\tau > 0$ and $\sigma > 0$ are parameters that essentially determine the stepsizes of the primal and dual subproblems, respectively; $\prox_{\tau G}$ and $\prox_{\sigma F^*}$ denote the proximal operators of the functionals $\tau G$ and $\sigma F^*$, respectively; $S'(u^k)$ denotes the Fr\'echet derivative of $S$ at $u^k$; and $(S'(u^k))^*$ denotes the adjoint operator\footnote{Here, we slightly abuse the notation $^*$ to denote both the adjoint of an operator and the convex conjugate of a functional, and the meaning of $^*$ can be clear in the specific context under discussion.} of $S'(u^k)$.

It is clear that $\prox_{\sigma F^*}$ in \eqref{eq:pdhg-nonlinear} admits a closed-form solution.
Regarding $\prox_{\tau G}$, as discussed in \cite{song2024admm}, it may have a closed-form solution for some specific cases of $\theta$, e.g., when $\theta$ is an $L^1$-regularization term or the indicator function of some simple closed convex set. It may also be solved efficiently either by certain optimization algorithms or by pre-trained neural networks. Therefore, the primary workload for implementing \eqref{eq:pdhg-nonlinear} is computing $S(2u^{k+1} - u^k)$ and $(S'(u^k))^*p^k$, namely, the state equation $y=S(u)$ with $u = 2u^{k+1} - u^k$ and a PDE $z = (S'(u))^*p$ termed the \emph{adjoint equation} with $u = u^k$ and $p = p^k$~\footnote{For brevity, we use the term \emph{adjoint equation} to refer to $z = (S'(u))^* p$ for any $p \in Y$, which covers its traditional definition (where $p = S(u) - y_d$) as a special case. The explicit expression of the adjoint equation (in form of PDE) can be formally derived for each specific PDE constraint in \eqref{eq:basic-problem}, as demonstrated in \Cref{sec:burgers,,sec:bp-opt-ctrl,,sec:sp-opt-ctrl}.}.
For the case where $\prox_{\tau G}$ admits a closed-form solution, the primal-dual method \eqref{eq:pdhg-nonlinear} does not need to solve any optimal control subproblem, and this is advantageous than the ADMM (\ref{eq:admm-pinn}).

Also, as mentioned in Section \ref{sec:admm}, it is necessary to discuss how to solve the subproblems of the primal-dual method \eqref{eq:pdhg-nonlinear}. If both the state and adjoint equations in \eqref{eq:basic-problem} are solved by FEMs, the resulting algorithm is labeled as the PD-FEM. Also, PINNs can be embedded into \eqref{eq:pdhg-nonlinear} for solving the underlying PDEs and the resulting algorithm is called the PD-PINNs hereafter. Again, the neural networks should also be retrained for the PD-PINNs iteratively with different $u^k$, $u^{k+1}$, and $p^k$.

\subsection{Motivation}

All the aforementioned algorithms are either inapplicable to \eqref{eq:basic-problem} or suffer from the computational bottleneck of repeatedly solving algebraic systems or retraining neural networks.
To tackle this problem, we notice that each of the state and adjoint equations in the primal-dual method \eqref{eq:pdhg-nonlinear} can be viewed as a parametric PDE that shares the same solution operator across iterations. Hence, we could construct two pre-trained neural surrogate models to approximate these solution operators, and propose a primal-dual-based operator learning approach for (\ref{eq:basic-problem}).
For a specific application of \eqref{eq:basic-problem}, a mesh-free algorithm can be specified and each iteration of the specified algorithm entails only two efficient forward passes of the pre-trained neural networks, along with a few simple algebraic operations. The main feature of this new approach is that it avoids repeatedly solving algebraic systems and retraining neural networks across algorithm iterations. This distinguishes from traditional numerical methods and other deep learning methods for \eqref{eq:basic-problem} in the literature, including the mentioned ADMM-PINNs in Section \ref{sec:admm} as an extension work of \cite{song2024admm}.

\subsection{Organization}

The rest of this paper is organized as follows.
In~\Cref{sec:preliminaries}, we present the primal-dual-based operator learning approach to (\ref{eq:basic-problem}) conceptually.
Then, we specify different algorithms for the optimal control of stationary Burgers equations in \Cref{sec:burgers}, sparse bilinear optimal control of parabolic equations in \Cref{sec:bp-opt-ctrl}, and optimal control of semilinear parabolic equations in \Cref{sec:sp-opt-ctrl}.
The effectiveness and efficiency of the resulting algorithms
are demonstrated in each section by some preliminary numerical results. In particular, we include some numerical comparisons with the reference ones obtained by FEM-based high-fidelity traditional numerical methods and other deep learning methods. Finally, some conclusions are summarized in \Cref{sec:conclusion}.

\section{A primal-dual-based operator learning approach to (\ref{eq:basic-problem})}\label{sec:preliminaries}

In this section, we delineate the primal-dual-based operator learning approach to problem \eqref{eq:basic-problem}.
The conceptual algorithmic framework will be specified as concrete algorithms for various optimal control problems in \Cref{sec:burgers,,sec:bp-opt-ctrl,,sec:sp-opt-ctrl}.

Recall that the primal-dual method \eqref{eq:pdhg-nonlinear} for \eqref{eq:basic-problem} involves solving two PDEs: the state equation $y=S(u)$ with $u = 2u^{k+1}-u^k$ and the adjoint equation $z = (S'(u))^* p$ with $u = u^k$ and $p = p^k$.
Consider first the state equation $y=S(u)$, which is a PDE parameterized by $u$. The operator learning approach approximates its solution operator $S$ by a neural network $\mathcal{S}_{\theta_s}$ parameterized by $\theta_s$. This neural network is pre-trained over a training set consisting of sampled controls $u$ and their corresponding states $S(u)$. Denoting $\theta_s^*$ as the optimized parameter, we thus obtain a trained neural surrogate model $\mathcal{S}_{\theta_s^*}$ for $S$, and evaluating $\mathcal{S}_{\theta_s^*} (u)$ for any input function $u$ requires only a forward pass of the neural network $\mathcal{S}_{\theta_s^*}$.
Similarly, we construct and train a neural network $\mathcal{A}_{\theta_a^*}(u, y, p)$ by operator learning to compute $(S'(u))^*p$ \footnote{At first glance, the term $(S'(u^k))^*p^k$ might seem to depend only on $u^k$ and $p^k$. However, as demonstrated later, it is sometimes more convenient to consider $(S'(u^k))^*p^k$ as being parameterized by $y^k= S(u^k)$ and $p^k$. Therefore, in the general form, we define $\cA_{\theta_a}$ as an operator of $u, y$, and $p$.}.
By replacing $S(2u^{k+1} - u^k)$ and $(S'(u^k))^*p^k$ in \eqref{eq:pdhg-nonlinear} with $\cS_{\theta_s^*} (2u^{k+1} - u^k)$ and $\cA_{\theta_a^*}(u^k, y^k, p^k)$, respectively, the resulting primal-dual-based operator learning approach to (\ref{eq:basic-problem}) in the abstract setting is presented as \Cref{alg:pdhg-opl}.

\begin{algorithm}
\floatname{algorithm}{Algorithmic Framework}
	\caption{A primal-dual-based operator learning approach to (\ref{eq:basic-problem}).}
\begin{algorithmic}[1]
	\REQUIRE    Pre-trained neural solution operators $\mathcal{S}_{\theta_s^*}$ and $\mathcal{A}_{\theta_a^*}$; initialization $u^0, p^0$; stepsizes $\tau, \sigma > 0$.
	
	\FOR {$k = 0, 1, \ldots $}
	\STATE $y^k = \mathcal{S}_{\theta_s^*}(u^k)$. \\
	\STATE $u^{k+1} = \prox_{\tau G} \left(u^k - \tau \mathcal{A}_{\theta_a^*}(u^k, y^k, p^k)\right)$. \\
	\STATE $p^{k+1} = \prox_{\sigma F^*} \left(p^k + \sigma \mathcal{S}_{\theta_s^*}(2u^{k+1} - u^k)\right)$. \\
	\ENDFOR
	\ENSURE Numerical solution $(\hat{u}(x) ,\hat{y}(x)) = (u^{k+1}, \cS_{\theta_s^*}(\hat{u}))$ of problem (\ref{eq:basic-problem}).
\end{algorithmic}
\label[framework]{alg:pdhg-opl}
\end{algorithm}

Different operator learning methods can be specified from \Cref{alg:pdhg-opl} by applying neural networks with different structures, such as the DeepONets \cite{lu2019deeponet,wang2021fast}, the FNO \cite{li2021fourier},  and the LNO \cite{cao2023lno}. Each iteration of all these primal-dual-based operator learning methods requires only  implementing two forward passes of the pre-trained neural networks $\mathcal{S}_{\theta_s^*}(u)$ and $\mathcal{A}_{\theta_a^*}(u, y, p)$, and some simple algebraic operations, hence avoids retraining across iterations.
Recall that the pre-trained models $\cS_{\theta_s^*}(u)$ and $\cA_{\theta_a^*}(u, y, p)$ are used to approximate $S(u)$ and $(S'(u))^*p$, respectively.
Therefore, they can be reused by \Cref{alg:pdhg-opl} for solving \eqref{eq:basic-problem} given different parameters $y_d$, $\alpha$ and $\theta$.
\Cref{alg:pdhg-opl} capitalizes on the efficiency and generalization offered by operator learning, while also inheriting the model-based nature and the structure-friendly feature of primal-dual methods by respecting the separable structure of \eqref{eq:basic-problem}.

In the following sections, we shall demonstrate the implementation of \Cref{alg:pdhg-opl} with different operator learning models to concrete examples of problem \eqref{eq:basic-problem} with various PDE constraints and nonsmooth regularization terms. We consider three classic optimal control problems including optimal control of stationary Burgers equations, sparse bilinear control of parabolic equations, and optimal control of semilinear parabolic equations. Numerical results are presented to validate the effectiveness and efficiency of the specified primal-dual-based operator learning methods. The code is implemented in Python with \texttt{PyTorch}~\cite{paszke2019pytorch}, and numerical experiments were conducted on a server with the Linux operating system, an AMD ENYC 7542 CPU (32 cores, 64 threads), a 512GB memory and an NVIDIA GeForce RTX 4090 GPU. All the codes and generated data are available at \url{https://github.com/tianyouzeng/PDOL-optimal-control}.

\section{Optimal control of stationary Burgers equations}\label{sec:burgers}

In this section, we showcase the application of \Cref{alg:pdhg-opl} to the optimal control of stationary Burgers equations \cite{de2004comparison,volkwein2000application} and derive a specific algorithm for this problem.
Let $\Omega = (0, 1)$ and consider
\begin{equation}\label{eq:burgers-optctrl}
    \begin{aligned}
        \min_{u\in L^2(\Omega),y\in L^2(\Omega)} \quad &\frac{1}{2}\|y - y_d\|^2_{L^2(\Omega)} + \frac{\alpha}{2}\|u\|^2_{L^2(\Omega)}+I_{U_{ad}}(u)
    \end{aligned}
\end{equation}
subject to the stationary Burgers equation
\begin{equation}\label{eq:burgers-equation}
         -\nu y''+yy'=u ~ \text{in}~(0,1), \quad
         y(0)=y(1)=0.
\end{equation}
Above, $I_{U_{ad}}(u)$ is the indicator function of the admissible set $U_{ad} := \{u \in L^2(\Omega): a \leq u(x) \leq b \text{~a.e.~in~} \Omega\}$ with $a$ and $b$ being given constants. The target $y_d\in H^1(\Omega)$ is prescribed, the constant $\nu > 0$ represents the viscosity coefficient, and $\alpha > 0$ is the regularization parameter.
The existence of a solution to problem~\eqref{eq:burgers-optctrl}--\eqref{eq:burgers-equation} has been well-studied, e.g., in \cite{de2004comparison}.

\subsection{Primal-dual decoupling for~\texorpdfstring{\eqref{eq:burgers-optctrl}--\eqref{eq:burgers-equation}}{(3.1)-(3.2)}}\label{sec:burgers-optctrl-pd}

We first delineate the primal-dual decoupling \eqref{eq:pdhg-nonlinear} for \eqref{eq:burgers-optctrl}--\eqref{eq:burgers-equation}, by elaborating on the computations of the proximal operators $\prox_{\tau G}$, $\prox_{\sigma F^*}$ and specifying the terms $S(2u^{k+1} - u^k)$, $(S'(u^k))^*p^k$ as concrete state and adjoint equations, respectively.
For \eqref{eq:burgers-optctrl}--\eqref{eq:burgers-equation}, the functionals $F: L^2(\Omega) \to \mathbb{R}$ and $G: L^2(\Omega) \to \mathbb{R} \cup \{+\infty\}$ defined in Section \ref{sec:pd-decoup} are specified by
\begin{equation}\label{eq:fun1}
    F(y) = \frac{1}{2}\|y - y_d\|^2_{L^2(\Omega)} \quad \text{and} \quad G(u) = \frac{\alpha}{2}\|u\|^2_{L^2(\Omega)} + I_{U_{ad}}(u).
\end{equation}
By the Moreau identity $y = \prox_{\sigma F}(y) + \sigma \prox_{\frac{1}{\sigma} F^*} \left(\frac{y}{\sigma}\right)$ and some simple manipulations, we have the following result for $\prox_{\tau G}$ and $\prox_{\sigma F^*}$.
\begin{proposition}\label{prop:burgers-prox}
    Given the functionals $F$ and $G$ defined in \eqref{eq:fun1} and any $\tau, \sigma > 0$, the proximal operators $\prox_{\tau G}$ and $\prox_{\sigma F^*}$ are given, respectively, by
    \begin{equation*}\label{eq:burgers-prox-G}
            \prox_{\tau G}(u) = \mathcal{P}_{U_{ad}} \left(\frac{1}{\alpha \tau + 1} \, u\right), \quad
            \prox_{\sigma F^*}(y) = \frac{1}{1 + \sigma} \left(y - \sigma y_d\right) \quad  \forall u, y\in L^2(\Omega).
    \end{equation*}
    Here, $\mathcal{P}_{U_{ad}}: L^2(\Omega) \to L^2(\Omega)$ denotes the projection operator onto $U_{ad}$ defined by
    \begin{equation*}
        \mathcal{P}_{U_{ad}}(v)(x) = \min\{\max\{v(x), a\}, b\}, \quad \forall v\in L^2(\Omega), ~ x \in \Omega.
    \end{equation*}
\end{proposition}

\Cref{prop:burgers-prox} shows that both $\prox_{\tau G}$ and $\prox_{\sigma F^*}$ admit closed-form solutions, and their evaluation involves only pointwise operations. Thus, they can be readily implemented and computed efficiently.

Regarding the evaluation of $S(2u^{k+1} - u^k)$ and $(S'(u^k))^*p^k$ in \eqref{eq:pdhg-nonlinear}, the operator $S: L^2(\Omega)\rightarrow L^2(\Omega)$ reduces to the solution operator of the state equation~\eqref{eq:burgers-equation}.
In particular, $S(2u^{k+1} - u^k)$ is the solution of~\eqref{eq:burgers-equation} with $u = 2u^{k+1} - u^k$.
On the other hand, $(S'(u^k))^*p^k$ can be specified as the solution of the adjoint equation given below.
\begin{proposition}\label{prop:burgers-adjoint}
Consider problem \eqref{eq:burgers-optctrl}--\eqref{eq:burgers-equation}. For any $u \in L^2(\Omega)$ and $p \in L^2(\Omega)$, if we let $y = S(u)$, then the solution of the following adjoint equation:
    \begin{equation}\label{eq:burgers-equation-adj}
             -\nu z'' - y z' = p  ~\text{in~} (0, 1), \quad
             z(0) = z(1) = 0,
    \end{equation}
    satisfies $z = (S'(u))^*p$.
\end{proposition}

The proof of \Cref{prop:burgers-adjoint} relies on the following lemma, which can be proved by the standard arguments in \cite{troltzsch2010optimal}.

\begin{lemma}\label{lem:burgers-sfd}
    The solution operator $S$ associated to \eqref{eq:burgers-equation} is Fr\'echet differentiable.
    Moreover, for any $u, v \in L^2(\Omega)$, $w := S'(u) v$ is given by the solution of the equation
    \begin{equation}\label{eq:burgers-equation-derivative}
             -\nu w'' + y w' + y' w = v ~ \text{in~} (0, 1), \quad
             w(0) = w(1) = 0,
    \end{equation}
    where $y = S(u)$ is the solution of \eqref{eq:burgers-equation} corresponding to $u$.
\end{lemma}

\begin{proof}[Proof of \Cref{prop:burgers-adjoint}]
    For any $u, v, p \in L^2(\Omega)$, we let $y = S(u)$ and $w = S'(u)v$, and denote by $z \in H^1_0(\Omega)$ the solution of \eqref{eq:burgers-equation-adj} with the selected $p$.
    It follows from \Cref{lem:burgers-sfd} that $w \in H^1_0(\Omega)$ is the solution of \eqref{eq:burgers-equation-derivative}, and in particular, $w$ satisfies the variational formulation
   \begin{equation}\label{eq:lem-sfd-veq}
       \int_0^1 \left(\nu w'\phi' + yw' \phi + y'w\phi\right) dx = \int_0^1 v\phi ~ dx, \quad \forall \phi \in H^1_0(\Omega),
   \end{equation}
    Taking $\phi = z$ in \eqref{eq:lem-sfd-veq}, we have
    \begin{equation}\label{eq:lem-burgers-adjoint-rhs}
        \int_0^1 \left(\nu w' z' + y w' z + y' wz \right) dx = \langle v, z \rangle_{L^2(\Omega)}.
    \end{equation}
    On the other hand, the variational form of the adjoint equation \eqref{eq:burgers-equation-adj} is
    $$ \int_0^1 \left( \nu z' \phi' - y z' \phi \right) dx = \int_0^1 p \phi ~ dx, \quad \forall \phi \in H^1_0(\Omega). $$
    Taking $\phi = w$ yields
    \begin{equation}\label{eq:lem-burgers-adjoint-lhs}
        \int_0^1 \left( \nu w' z' - y w z' \right) dx = \langle w, p \rangle_{L^2(\Omega)}.
    \end{equation}
    Using the identity
    $ \int_0^1 \left( yw'z + y'wz \right) dx = (ywz) \Big\vert_0^1 - \int_0^1 ywz' dx = - \int_0^1 ywz' dx, $
    together with \eqref{eq:lem-burgers-adjoint-rhs} and \eqref{eq:lem-burgers-adjoint-lhs}, we conclude that $ \langle S'(u) v, p \rangle_{L^2(\Omega)} = \langle v, z \rangle_{L^2(\Omega)} $ and thus complete the proof.
\end{proof}

We define $A:L^2(\Omega)\times L^2(\Omega)\rightarrow L^2(\Omega), (y,p)\mapsto z $ as the solution operator of the adjoint equation (\ref{eq:burgers-equation-adj}), i.e. $z=A(y,p)$.
Note that the state equation~\eqref{eq:burgers-equation} is a PDE with a single parameter $u$, while the adjoint equation~\eqref{eq:burgers-equation-adj} is a PDE with two parameters $y$ and $p$.
This allows the application of operator learning methods to approximate the solution operators $S$ and $A$ by neural networks and efficiently solve the state and adjoint equations.

\subsection{DeepONet and MIONet}
We now apply a DeepONet $\mathcal{S}_{\theta_s}$ and an MIONet $\mathcal{A}_{\theta_a}$, parameterized by $\theta_s$ and $\theta_a$, to approximate the solution operators $S$ and $A$, respectively.
The trained networks $\mathcal{S}_{\theta_s^*}$ and $\mathcal{A}_{\theta_a^*}$ provide two neural surrogate models for solving the state and adjoint equations in the primal-dual-based operator learning method for \eqref{eq:burgers-optctrl}--\eqref{eq:burgers-equation}.

Two types of DeepONets are proposed in \cite{lu2019deeponet}: the stacked DeepONet and the unstacked DeepONet, as illustrated in~\Cref{fig:deeponet}. Both the stacked and the unstacked variants do not restrict the branch and trunk nets to any specific network architecture. As $x\in \Omega$ is usually low-dimensional, a standard fully-connected neural network (FNN) structure suffices for the trunk net. The choice of the branch net depends on the structure of the input function $u$, and it can be chosen as a FNN,  a residual neural network, a convolutional neural network, or a graph neural network.

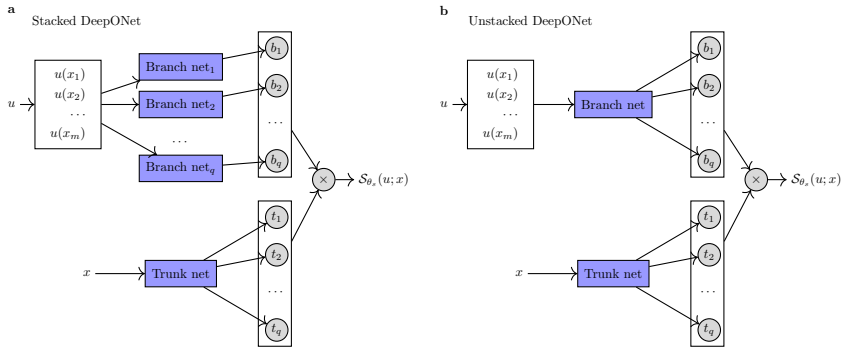
\begin{figure}[!ht]
    \centering

\begin{tikzpicture}[global scale =0.5]
\centering

\node at(0,11.5)     {\textbf{a}};
\node at(2,11.2)     {Stacked DeepONet};
\node at(0,9)     (7) {$u$};
\node at(1.5,9)     (8) [zbox]{$\begin{aligned}
    u(x_1)\\u(x_2)\\ \cdots\\u(x_m)
    \end{aligned}$};
\node at(4.5,10)    (9)[zbox,fill=blue!40]{Branch net$_1$};
\node at(4.5,9)    (10)[zbox,fill=blue!40]{Branch net$_2$};
\node at(4.5,8)        {$\cdots$};
\node at(4.5,7.3)  (11)[zbox,fill=blue!40]{Branch net$_q$};
\node at(4.5,4.5)  (12)[zbox,fill=blue!40]{Trunk net};
\node at(7,9)(13)[rectangle, minimum width =25pt, minimum height =110pt, inner sep=5pt,draw=black]  {} ;
\node at(7,4.5)(14)[rectangle, minimum width =25pt, minimum height =110pt, inner sep=5pt,draw=black]  {} ;
\node at(7.05,10.5) (15)[circle, minimum width =17pt, minimum height =17pt, inner sep=0.1pt, fill=gray!30,draw=black]{$b_1$};
\node at(7.05,9.5) (16)[circle, minimum width =17pt, minimum height =17pt, inner sep=0.1pt, fill=gray!30,draw=black]{$b_2$};
\node at(7.05,8.5)          {$\cdots$};
\node at(7.05,7.5) (17)[circle, minimum width =17pt, minimum height =17pt, inner sep=0.1pt, fill=gray!30,draw=black]{$b_q$};
\node at(7.05,6)   (18)[circle, minimum width =17pt, minimum height =17pt, inner sep=0.1pt, fill=gray!30,draw=black]{$t_1$};
\node at(7.05,5)   (19)[circle, minimum width =17pt, minimum height =17pt, inner sep=0.1pt, fill=gray!30,draw=black]{$t_2$};
\node at(7.05,4)          {$\cdots$};
\node at(7.05,3)   (20)[circle, minimum width =17pt, minimum height =17pt, inner sep=0.1pt, fill=gray!30,draw=black]{$t_q$};
\node at(8.3,7)    (21)[circle, minimum width =17pt, minimum height =17pt, inner sep=0.1pt, fill=gray!30,draw=black]{$\times$};
\node at(9.9,7)   (22){$\cS_{\theta_s}(u;x)$};
\node at(2,4.5)   (23){$x$};

\draw[->] (7) --(8);
\draw[->] (8) --(9);
\draw[->] (8) --(10);
\draw[->] (8) --(11);
\draw[->] (23) --(12);
\draw[->] (9) --(15);
\draw[->] (10) --(16);
\draw[->] (11) --(17);
\draw[->] (12) --(18);
\draw[->] (12) --(19);
\draw[->] (12) --(20);
\draw[->] (13) --(21);
\draw[->] (14) --(21);
\draw[->] (21) --(22);

\node at(11.5,11.5)     {\textbf{b}};
\node at(13.8,11.2)     {Unstacked DeepONet};
\node at(11.5,9)    (24) {$u$};
\node at(13,9)      (25) [zbox]{$\begin{aligned}
    u(x_1)\\u(x_2)\\ \cdots\\u(x_m)
    \end{aligned}$};
\node at(16,9)      (26)[zbox,fill=blue!40]{Branch net};
\node at(16,4.5)    (27)[zbox,fill=blue!40]{Trunk net};
\node at(18.5,9)    (28)[rectangle, minimum width =25pt, minimum height =110pt, inner sep=5pt,draw=black]  {} ;
\node at(18.5,4.5)  (29)[rectangle, minimum width =25pt, minimum height =110pt, inner sep=5pt,draw=black]  {} ;
\node at(18.55,10.5) (30)[circle, minimum width =17pt, minimum height =17pt, inner sep=0.1pt, fill=gray!30,draw=black]{$b_1$};
\node at(18.55,9.5)  (31)[circle, minimum width =17pt, minimum height =17pt, inner sep=0.1pt, fill=gray!30,draw=black]{$b_2$};
\node at(18.55,8.5)          {$\cdots$};
\node at(18.55,7.5)  (32)[circle, minimum width =17pt, minimum height =17pt, inner sep=0.1pt, fill=gray!30,draw=black]{$b_q$};
\node at(18.55,6)    (33)[circle, minimum width =17pt, minimum height =17pt, inner sep=0.1pt, fill=gray!30,draw=black]{$t_1$};
\node at(18.55,5)   (34)[circle, minimum width =17pt, minimum height =17pt, inner sep=0.1pt, fill=gray!30,draw=black]{$t_2$};
\node at(18.55,4)          {$\cdots$};
\node at(18.55,3)   (35)[circle, minimum width =17pt, minimum height =17pt, inner sep=0.1pt, fill=gray!30,draw=black]{$t_q$};
\node at(19.8,7)    (36)[circle, minimum width =17pt, minimum height =17pt, inner sep=0.1pt, fill=gray!30,draw=black]{$\times$};
\node at(21.4,7)      (37){$\cS_{\theta_s}(u;x)$};
\node at(13.5,4.5)  (38){$x$};

\draw[->] (24) --(25);
\draw[->] (25) --(26);
\draw[->] (26) --(30);
\draw[->] (26) --(31);
\draw[->] (26) --(32);
\draw[->] (27) --(33);
\draw[->] (27) --(34);
\draw[->] (27) --(35);
\draw[->] (28) --(36);
\draw[->] (29) --(36);
\draw[->] (36) --(37);
\draw[->] (38) --(27);

\end{tikzpicture}

\caption{The network structure of DeepONet, adopted from~\cite{lu2019deeponet}.}
\label{fig:deeponet}
\end{figure}

Here, we fix $\cS_{\theta_s}$ as the unstacked DeepONet and  employ fully-connected neural networks in both the branch net and the trunk net.
We approximate the solution of \eqref{eq:burgers-equation} by~$x(x-1)\cS_{\theta_s}(u; x)$ with $x\in \overline{\Omega}$, where $x(x-1)$ is introduced to enforce the boundary conditions $y(0)=y(1)=0$ in~\eqref{eq:burgers-equation}.
To train $\cS_{\theta_s}$, we assemble a training dataset $\{(u_i(x_j), x_j, y_i(x_j)) : 1 \leq i \leq N, 1 \leq j \leq m\}$, where each $u_i \in U$ is a sampled function from a certain distribution, $y_i = S(u_i)$ is the solution to~\eqref{eq:burgers-equation} with $u = u_i$, $N$ is the number of sampled functions, and $\{x_1, \ldots, x_m\}$ is a discretization of $\overline{\Omega}$.
The DeepONet $\cS_{\theta_s}$ is then trained by minimizing the following loss function:
\begin{equation}\label{eq:deeponet-loss}
    \cL_s(\theta_s) = \frac{1}{Nm}\sum_{i=1}^N \sum_{j=1}^m \left| x_j (x_j - 1) \cS_{\theta_s}(u_i; x_j) - y_i(x_j) \right|^2.
\end{equation}
Let $\theta_s^*$ be the trained parameter.
For each $u \in L^2(\Omega)$, the solution of \eqref{eq:burgers-equation} (i.e. $S(u)$) can thus be approximated by evaluating $\cS_{\theta_s^*}(u; x)$ for each $x \in \Omega$.

Note that DeepONet admits only one function input and does not apply to approximating the solution operator $A$, where multiple inputs $y$ and $p$ are involved. To overcome this issue, we consider the MIONet~\cite{jin2022mionet}, whose architecture with two input functions $y$ and $p$ is illustrated in~\Cref{fig:mionet}.
To summarize, the MIONet first applies different branch nets to each input function $u$ and a trunk net to the input spatial variable $x$, and then performs an element-wise product followed by a sum on the outputs of different subnetworks.

\begin{figure}[!ht]
    \centering

\begin{tikzpicture}[global scale =0.5]
\centering

\node at(4.5,10.0)    (24) {$y$};
\node at(4.5,12)      (25) [zbox]{$\begin{aligned}
    y(x_1)\\y(x_2)\\ \cdots\\y(x_m)
    \end{aligned}$};
\node at(9,10.0)    (44) {$p$};
\node at(9,12)      (45) [zbox]{$\begin{aligned}
    p(x_1)\\p(x_2)\\ \cdots\\p(x_m)
    \end{aligned}$};
\node at(4.5,15)      (26)[zbox,fill=blue!40]{Branch net of $y$};
\node at(9,15)    (43)[zbox,fill=blue!40]{Branch net of $p$};
\node at(13.5,15)    (27)[zbox,fill=blue!40]{Trunk net};
\node at(4.5,17.55)    (28)[rectangle, minimum width =110pt, minimum height =25pt, inner sep=5pt,draw=black]  {} ;
\node at(13.5,17.55)  (29)[rectangle, minimum width =110pt, minimum height =25pt, inner sep=5pt,draw=black]  {} ;
\node at(9,17.55)  (42)[rectangle, minimum width =110pt, minimum height =25pt, inner sep=5pt,draw=black]  {} ;
\node at(3,17.55) (30)[circle, minimum width =17pt, minimum height =17pt, inner sep=0.1pt, fill=gray!30,draw=black]{$b^{(y)}_1$};
\node at(4,17.55)  (31)[circle, minimum width =17pt, minimum height =17pt, inner sep=0.1pt, fill=gray!30,draw=black]{$b^{(y)}_2$};
\node at(5,17.55)          {$\cdots$};
\node at(6,17.55)  (32)[circle, minimum width =17pt, minimum height =17pt, inner sep=0.1pt, fill=gray!30,draw=black]{$b^{(y)}_q$};
\node at(7.5,17.55) (39)[circle, minimum width =17pt, minimum height =17pt, inner sep=0.1pt, fill=gray!30,draw=black]{$b^{(p)}_1$};
\node at(8.5,17.55)  (40)[circle, minimum width =17pt, minimum height =17pt, inner sep=0.1pt, fill=gray!30,draw=black]{$b^{(p)}_2$};
\node at(9.5,17.55)          {$\cdots$};
\node at(10.5,17.55)  (41)[circle, minimum width =17pt, minimum height =17pt, inner sep=0.1pt, fill=gray!30,draw=black]{$b^{(p)}_q$};
\node at(12,17.55)    (33)[circle, minimum width =17pt, minimum height =17pt, inner sep=0.1pt, fill=gray!30,draw=black]{$t_1$};
\node at(13,17.55)   (34)[circle, minimum width =17pt, minimum height =17pt, inner sep=0.1pt, fill=gray!30,draw=black]{$t_2$};
\node at(14,17.55)          {$\cdots$};
\node at(15,17.55)   (35)[circle, minimum width =17pt, minimum height =17pt, inner sep=0.1pt, fill=gray!30,draw=black]{$t_q$};
\node at(9,18.8)    (36)[circle, minimum width =17pt, minimum height =17pt, inner sep=0.1pt, fill=gray!30,draw=black]{$\odot$};
\node at(9,19.8)    (46)[circle, minimum width =17pt, minimum height =17pt, inner sep=0.1pt, fill=gray!30,draw=black]{$+$};
\node at(9,21)      (37){$\cA_{\theta_a}(y, p; x)$};
\node at(13.5,12.5)  (38){$x$};

\draw[->] (24) --(25);
\draw[->] (25) --(26);
\draw[->] (26) --(30);
\draw[->] (26) --(31);
\draw[->] (26) --(32);
\draw[->] (44) --(45);
\draw[->] (45) --(43);
\draw[->] (43) --(39);
\draw[->] (43) --(40);
\draw[->] (43) --(41);
\draw[->] (42) --(36);
\draw[->] (27) --(33);
\draw[->] (27) --(34);
\draw[->] (27) --(35);
\draw[->] (28) --(36);
\draw[->] (29) --(36);
\draw[->] (36) --(46);
\draw[->] (46) --(37);
\draw[->] (38) --(27);

\end{tikzpicture}

\caption{The structure of MIONet with two input functions $y$ and $p$. The symbol $\odot$ denotes the element-wise product.}
\label{fig:mionet}
\end{figure}
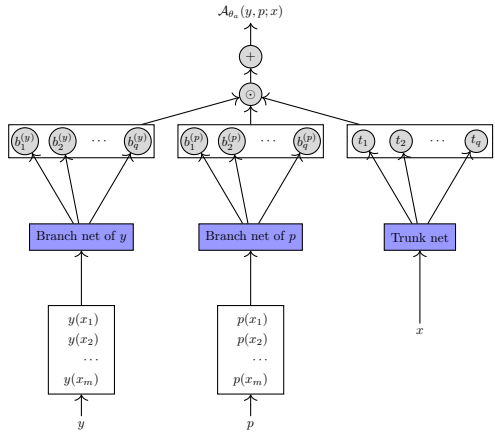

We apply an MIONet $\cA_{\theta_a}$ for the solution operator $A$ of the adjoint equation \eqref{eq:burgers-equation-adj}. The solution $z$ of \eqref{eq:burgers-equation-adj} is approximated by $x(x-1) \cA_{\theta_a}(y, p; x)$ with $x\in \overline{\Omega}$.
Here, $x(x-1)$ serves as a hard constraint to ensure the satisfaction of the boundary condition in~\eqref{eq:burgers-equation-adj}.
To train $\cA_{\theta_a}$, we first generate a training set $\{(y_i(x_j), p_i(x_j), x_j,  z_i(x_j))\}_{1 \leq i \leq N, 1 \leq j \leq m}$, where $y_i, p_i \in Y$ are sampled functions from certain distributions, and $z_i$ is the solution of~\eqref{eq:burgers-equation-adj} corresponding to $y = y_i$ and $p = p_i$.
Then, the MIONet $\cA_{\theta_a}$ is trained by minimizing the loss function
\begin{equation}\label{eq:mionet-loss}
    \cL_a(\theta_a) = \frac{1}{Nm}\sum_{i=1}^N \sum_{j=1}^m \left| x_j (x_j - 1) \cA_{\theta_a}(y_i, p_i; x_j) - z_i(x_j) \right|^2.
\end{equation}
Denote $\theta_a^*$ as the trained parameter.
The solution of \eqref{eq:burgers-equation-adj} (i.e. $A(y, p)$) can thus be approximated by evaluating $\cA_{\theta_a^*}(y, p; x)$ for each $x \in \Omega$.

\subsection{Primal-dual-based operator learning method for \texorpdfstring{\eqref{eq:burgers-optctrl}--\eqref{eq:burgers-equation}}{(3.1)-(3.2)}}

With the pre-trained neural solution operators $\cS_{\theta_s^*}$ and $\cA_{\theta_a^*}$, a primal-dual-based operator learning method for~\eqref{eq:burgers-optctrl}--\eqref{eq:burgers-equation} can be specified from \Cref{alg:pdhg-opl} and we list it compactly as \Cref{alg:burgers-pdhg-opl}.

\begin{algorithm}[!ht]

	\caption{A primal-dual-based operator learning method for \eqref{eq:burgers-optctrl}--\eqref{eq:burgers-equation}}
	\begin{algorithmic}[1]
	    \REQUIRE   Pre-trained neural solution operators  $\mathcal{S}_{\theta_s^*}$ and $\mathcal{A}_{\theta_a^*}$; initialization $u^0, p^0$; stepsizes $\tau, \sigma>0.$
	
        \FOR {$k = 0, 1, \ldots $}
            \STATE $y^k(x) = x(x-1)\cS_{\theta_s^*}(u^k; x)$.
            \STATE $u^{k+1}(x) = \min \left\{ \max \left\{ \frac{u^k(x) - \tau x(x-1) \cA_{\theta_a^*}(y^k, p^k; x)}{\alpha \tau + 1}, a \right\}, b \right\}$.
            \STATE $p^{k+1}(x) = \frac{1}{1 + \sigma} \left(p^k(x) + \sigma x(x-1) \cS_{\theta_s^*}(2u^{k+1} -u^k; x) - \sigma y_d(x)\right)$.
        \ENDFOR
	    \ENSURE Numerical solution $(\hat{u}, \hat{y}):=(u^{k+1}, x(x-1)\cS_{\theta_s^*}(u^{k+1}))$ to \eqref{eq:burgers-optctrl}--\eqref{eq:burgers-equation}.
	\end{algorithmic}
	\label{alg:burgers-pdhg-opl}
\end{algorithm}
Notably, the pre-trained neural solution operators  $\cS_{\theta_s^*}$ and $\cA_{\theta_a^*}$ are employed across different iterations in \Cref{alg:burgers-pdhg-opl}.  The evaluation of $\cS_{\theta_s^*}(u; x)$ and $\cA_{\theta_a^*}(y, p; x)$ is computationally cheaper than solving~\eqref{eq:burgers-equation} and~\eqref{eq:burgers-equation-adj} by traditional numerical solvers.
Moreover, it is clear that $\cS_{\theta_s^*}$ and $\cA_{\theta_a^*}$ can be reused for solving problem~\eqref{eq:burgers-optctrl}--\eqref{eq:burgers-equation} with different parameters $\alpha$ and $y_d$.
These verify the advantages claimed for the proposed approach in \Cref{sec:preliminaries}.

\subsection{Numerical experiments}\label{sec:burgers-numexp}
We now present some numerical results for solving the problem \eqref{eq:burgers-optctrl}--\eqref{eq:burgers-equation} to validate the effectiveness and efficiency of \Cref{alg:burgers-pdhg-opl}. We consider an example adopted from~\cite{de2004comparison}, with $\nu = 1/12$ and
$U_{ad} = \{u \in L^2(\Omega) : u \leq 0.3 \text{~~a.e.~in~} \Omega \}$.
We test on several different parameters $\alpha$ and $y_d$, which will be specified later.

We discretize $\overline{\Omega} = [0, 1]$ by $m = 101$ equi-spaced grid points $\{x_j\}_{j=1}^m$ and generate the training set for $\cS_{\theta_s}$ with functions discretized over this grid. We first sample functions $\{u_i\}_{i=1}^N$ with $N=5\times 10^3$ over a finer grid, from a Gaussian random field over $\overline{\Omega}$ subject to the homogeneous boundary condition:
\begin{equation}\label{s1}
    u_i \sim \mathcal{GP}(0, C(x_1, x_2)) \mid \{u_i(0) = u_i(1) = 0\},
\end{equation}
where $C(x_1, x_2) = r \exp \left(-\lVert x_1 - x_2 \rVert _2^2 / 2 \ell^2 \right)$, with $r =0.2$ and $l=0.2$, is the Gaussian covariance kernel.
We next down-sample each $u_i$ to the resolution $m = 101$, and solve the state equation~\eqref{eq:burgers-equation} with $u = u_i$ by the \texttt{bvp5c} function in MATLAB to get the solution $y_i$.
Then, $\{(u_i(x_j), x_j, y_i(x_j))\}_{1\leq i \leq N, 1 \leq j \leq m}$ forms a training dataset of $\cS_{\theta_s}$.
For the training set of $\cA_{\theta_a}$, we employ the functions $\{y_i := S(u_i)\}_{i=1}^N$, where $\{u_i\}_{i=1}^N$ are the functions generated for the training of $\cS_{\theta_s}$.
Additionally, we sample $N$ functions $\{p_i\}_{i=1}^N$ from the same Gaussian random field $\mathcal{GP}(0, C(x_1, x_2))$ but without imposing the homogeneous boundary condition.
We solve the adjoint equation~\eqref{eq:burgers-equation-adj} with $y = y_i$ and $p = p_i$ using the \texttt{bvp5c} solver, yielding the solutions $\{z_i\}_{i=1}^N$.
The training set of $\cA_{\theta_a}$ is set as $\{(y_i(x_j), p_i(x_j), x_j, z_i(x_j))\}_{1 \leq i \leq N, 1 \leq j \leq m}$. In total, $2N=10^4$ PDEs are required to be solved for generating the training sets for $\cS_{\theta_s}$ and $\cA_{\theta_a}$.

All the branch nets and trunk nets in $\mathcal{S}_{\theta_s}$ and $\mathcal{A}_{\theta_a}$ are set as fully-connected neural networks, each consisting of two layers with 101 neurons per layer and equipped with \texttt{ReLU} activation functions.
Then, to train $\cS_{\theta_s}$  and $\cA_{\theta_a}$, we minimize the loss functions  \eqref{eq:deeponet-loss} and \eqref{eq:mionet-loss} by  the Adam optimizer~\cite{kingma2014adam} with batch size $B:=100 \times 101$ for $2\times 10^3$ epochs.
The parameters ${\theta_s}$ and ${\theta_a}$ are initialized by the default \texttt{PyTorch} settings. The learning rate is first set to $0.001$ and then exponentially decayed by a factor of $0.8$ every $100$ iterations.
We report the training time for neural networks $\cS_{\theta_s}$ and $\cA_{\theta_a}$ in \Cref{tab:burger-model-acc}.

To validate the testing accuracy, we generate $10^4$ testing data respectively for the trained neural solution operators $\cS_{\theta_s^*}$ and $\cA_{\theta_a^*}$, following the same procedures for generating the training sets.
The testing results are presented in~\Cref{tab:burger-model-acc}. We observe that
the trained $\cS_{\theta_s^*}$ and $\cA_{\theta_a^*}$ produce satisfactory accuracy on the testing set, which validates the effectiveness of DeepONet and MIONet in solving (\ref{eq:burgers-equation}) and (\ref{eq:burgers-equation-adj}).

\begin{table}[!ht]
    \centering
    \small
    \caption{The training time and testing accuracy of the neural solution operators $\cS_{\theta_s^*}$ and $\cA_{\theta_a^*}$ measured in $L^2$-norm.
    Here, the absolute error is defined as $\|\hat{v} - v\|_{L^2(\Omega)}$, and the relative error is defined as $\|\hat{v} - v\|_{L^2(\Omega)} / \|v\|_{L^2(\Omega)}$, where $v(x)$ denotes the exact solution of (\ref{eq:burgers-equation}) or (\ref{eq:burgers-equation-adj}), and $\hat{v}(x)$ is the predicted value $\cS_{\theta_s^*}(u; x)$ or $\cA_{\theta_a^*}(y, p; x)$.}
    \begin{tabular}{ c c c c c c }
        \toprule
        ~ & Training & \multicolumn{2}{c}{Absolute $L^2$-error} & \multicolumn{2}{c}{Relative $L^2$-error} \\
        \cmidrule(lr){3-4} \cmidrule(lr){5-6}
        ~ & Time (s) &  Mean & SD & Mean & SD \\
        \midrule
        $\cS_{\theta_s^*}(u; x)$  & $367$ & $5.73 \times 10^{-4}$ & $3.79 \times 10^{-4}$ & $4.81 \times 10^{-3}$ & $9.31 \times 10^{-3}$\\
        $\cA_{\theta_a^*}(y, p; x)$ & $489$ & $6.71 \times 10^{-4}$ & $4.01 \times 10^{-4}$ & $4.68 \times 10^{-3}$ & $6.56 \times 10^{-3}$\\
        \bottomrule
    \end{tabular}
    \label{tab:burger-model-acc}
\end{table}

Next, we solve the optimal control problem~\eqref{eq:burgers-optctrl}--\eqref{eq:burgers-equation} by \Cref{alg:burgers-pdhg-opl} with the trained neural solution operators $\cS_{\theta_s^*}$ and $\cA_{\theta_a^*}$.
We initialize \Cref{alg:burgers-pdhg-opl} with $u^0 = 0$ and $p^0 = 0$.
The stepsizes are set to $\tau = 1.0$ and $\sigma = 0.2$.
The stopping criterion for \Cref{alg:burgers-pdhg-opl} is set to be
\begin{equation*}
    \max \left\{ \frac{\|u^{k+1} - u^k\|_{L^2(\Omega)}}{\max\{1, \|u^k\|_{L^2(\Omega)}\}}, \frac{\|p^{k+1} - p^k\|_{L^2(\Omega)}}{\max\{1, \|p^k\|_{L^2(\Omega)}\}} \right\} \leq 10^{-5}.
\end{equation*}
Besides, we also consider a heuristic for accelerating \Cref{alg:burgers-pdhg-opl}.
For the special case where $S$ is linear, it has been shown in \cite{he2012convergence} that \eqref{eq:pdhg-nonlinear} can be interpreted as an application of the proximal point algorithm, and this perspective inspires further updating $u$ and $p$ by relaxation steps with constant step sizes \cite{gol1979modified}.
The resulting algorithm usually numerically outperforms the original one, see e.g. \cite{biccari2023two,condat2023proximal,he2017algorithmic,he2012convergence}.
Inspired by \cite{gol1979modified,he2012convergence}, we apply the same strategy in this experiment by appending the following relaxation steps at each iteration of \Cref{alg:burgers-pdhg-opl}:
\begin{equation}\label{eq:relaxation-steps}
	u^{k+1} \gets u^k + \rho (u^{k+1} - u^k), \quad
	p^{k+1} \gets p^k + \rho (p^{k+1} - p^k),
\end{equation}
and term the resulting method as \Cref{alg:burgers-pdhg-opl}-R.

We test \Cref{alg:burgers-pdhg-opl} and \Cref{alg:burgers-pdhg-opl}-R on problem \eqref{eq:burgers-optctrl}--\eqref{eq:burgers-equation} with different sets of values for $\alpha$ and $y_d$.
To validate the accuracy, we compare the computed control $\hat{u}$ and state $\hat{y}$ with the solution $(u^*, y^*)$ obtained by the semi-smooth Newton method with finite element discretization (SSN-FEM) in \cite{de2004comparison}.
In particular, the absolute error $\epsilon_{abs}$ and relative error $\epsilon_{rel}$ of $\hat{u}, \hat{y}$ are defined as
\begin{equation*}
\begin{aligned}
	& \epsilon_{abs}(\hat{u}) := \|\hat{u} - u^*\|_{L^2(\Omega)}, \quad \epsilon_{rel}(\hat{u}) := \|\hat{u} - u^*\|_{L^2(\Omega)} / \|u^*\|_{L^2(\Omega)}, \\
	& \epsilon_{abs}(\hat{y}) := \|\hat{y} - y^*\|_{L^2(\Omega)}, \quad \epsilon_{rel}(\hat{y}) := \|\hat{y} - y^*\|_{L^2(\Omega)} / \|y^*\|_{L^2(\Omega)},
\end{aligned}
\end{equation*}
All the methods are evaluated on a uniform grid of size $m=2000$ over $\Omega$.
The computed controls and states of \Cref{alg:burgers-pdhg-opl} are illustrated in~\Cref{fig:burger-computed-ctrl-state}, and the numerical results are summarized in~\Cref{tab:burger-numerical-results-nopc}. For \Cref{alg:burgers-pdhg-opl}-R, the absolute and relative errors are the same as \Cref{alg:burgers-pdhg-opl}; hence we only report its number of iterations in \Cref{tab:burger-numerical-results-nopc}. We observe that both \Cref{alg:burgers-pdhg-opl} and \Cref{alg:burgers-pdhg-opl}-R achieve satisfactory accuracy.

\begin{figure}[!ht]
	\centering
  \subfloat[Computed control $u$ with $\alpha = 0.1$, $y_d = 0.3$]{\includegraphics[width=0.32\textwidth]{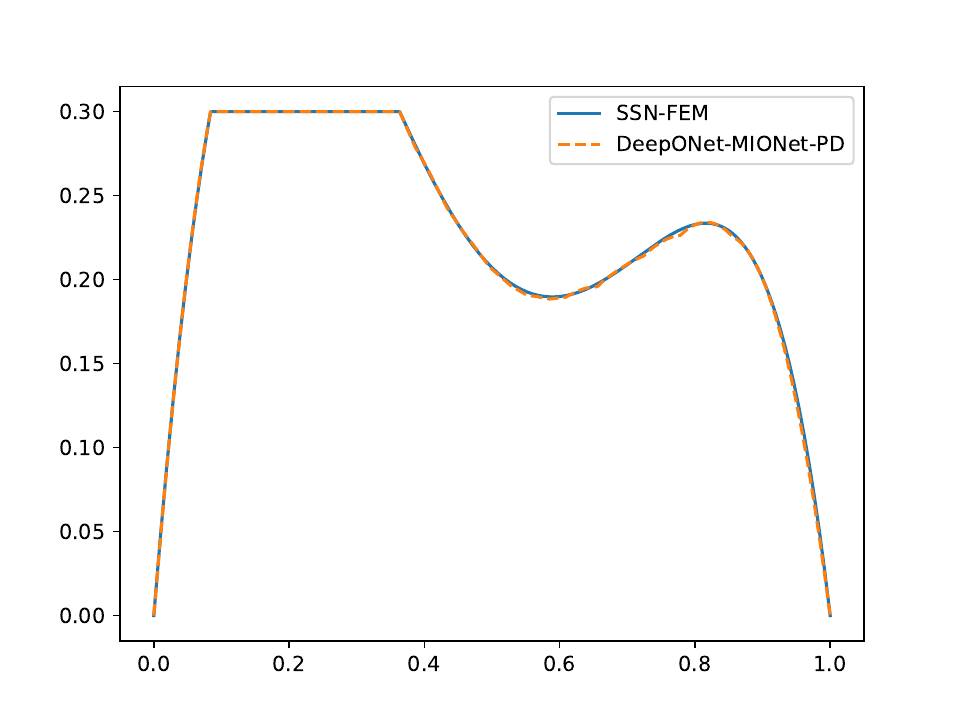}}\hfill
  \subfloat[Computed control $u$ with $\alpha = 0.01$, $y_d = 0.3$]{\includegraphics[width=0.32\textwidth]{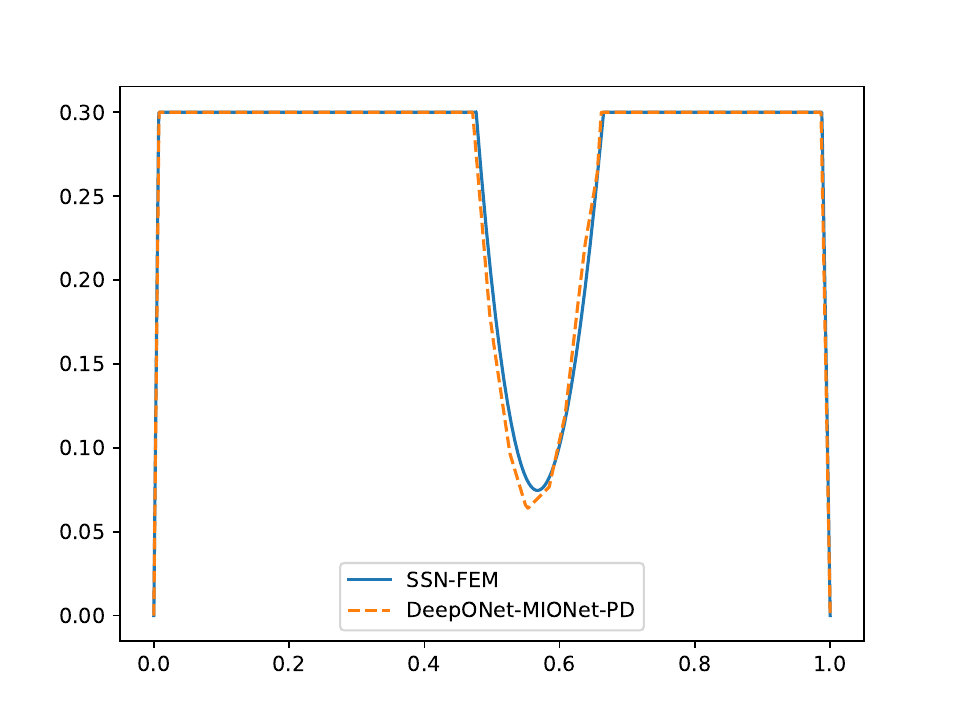}}\hfill
  \subfloat[Computed control $u$ with $\alpha = 0.1$, $y_d = 0.2$]{\includegraphics[width=0.32\textwidth]{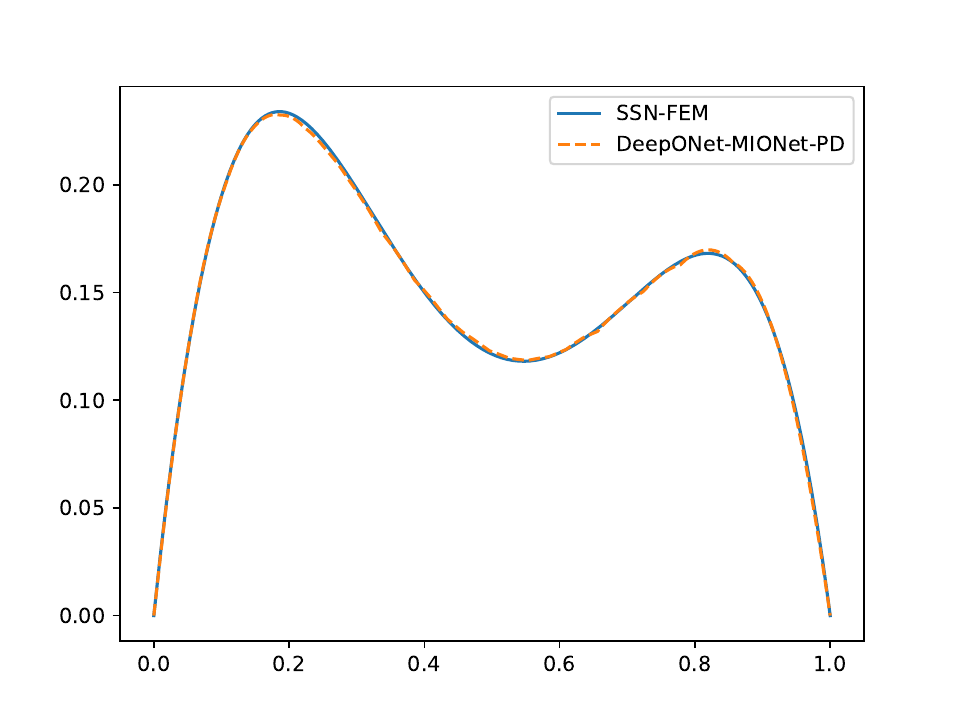}}\\
  \subfloat[Computed state $y$ with $\alpha = 0.1$, $y_d = 0.3$]{\includegraphics[width=0.32\textwidth]{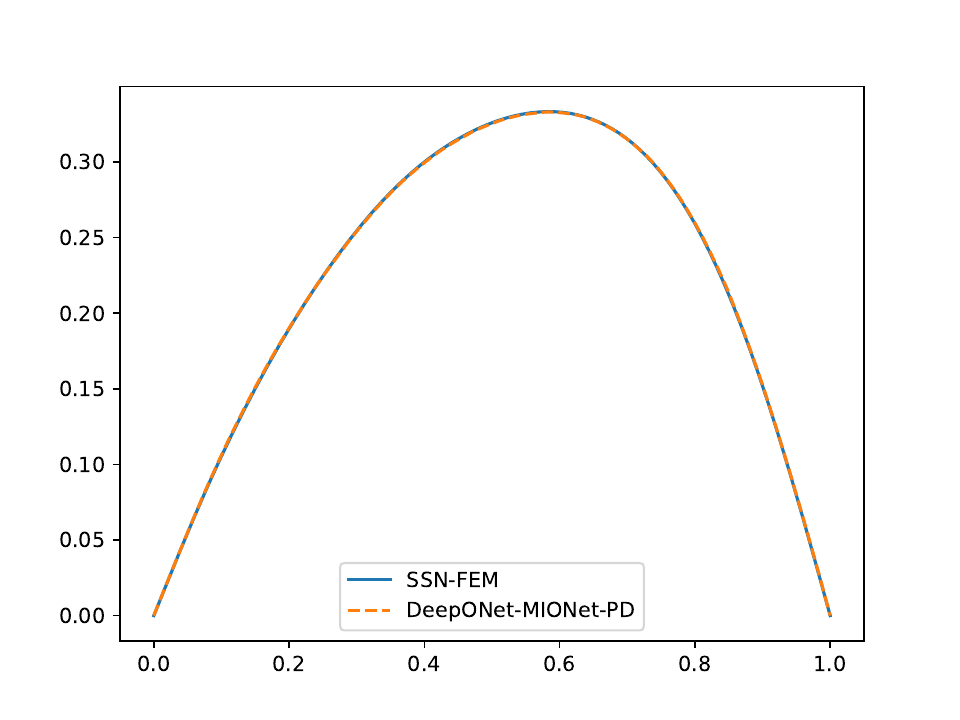}}\hfill
  \subfloat[Computed state $y$ with $\alpha = 0.01$, $y_d = 0.3$]{\includegraphics[width=0.32\textwidth]{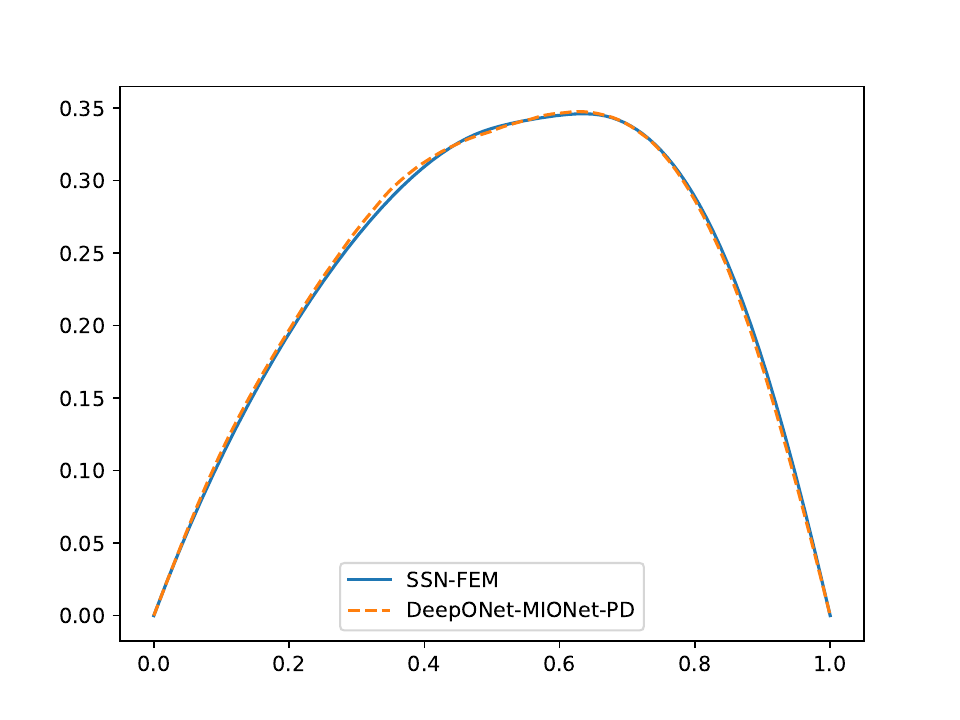}}\hfill
  \subfloat[Computed state $y$ with $\alpha = 0.1$, $y_d = 0.2$]{\includegraphics[width=0.32\textwidth]{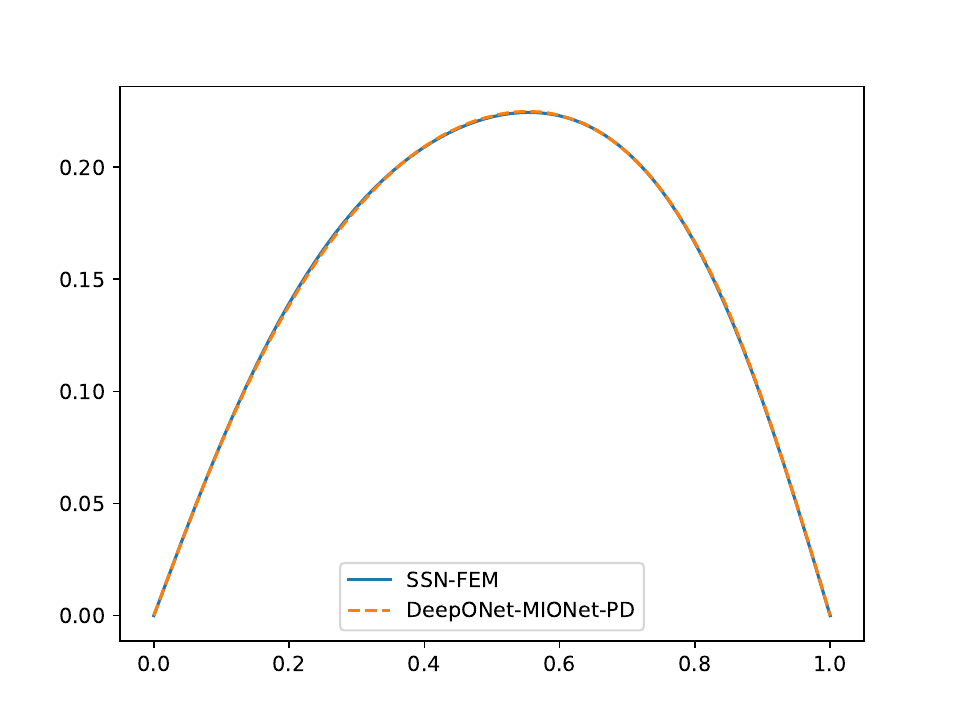}}
	\caption{Computed controls and states for~\eqref{eq:burgers-optctrl} by \Cref{alg:burgers-pdhg-opl} (DeepONet-MIONet-PD), compared with the baseline solver SSN-FEM}.
	\label{fig:burger-computed-ctrl-state}
\end{figure}

\begin{table}[!ht]
    \small
    \centering
    \caption{Numerical results of \Cref{alg:burgers-pdhg-opl} and \Cref{alg:burgers-pdhg-opl}-R for problem \eqref{eq:burgers-optctrl}--\eqref{eq:burgers-equation}. The number of iterations of \Cref{alg:burgers-pdhg-opl} and \Cref{alg:burgers-pdhg-opl}-R are reported before and after the slash (/), respectively.}
    \begin{tabular}[c]{ c c c c c c }
        \toprule
		~ & Iteration & $\epsilon_{abs}(\hat{u})$ & $\epsilon_{rel}(\hat{u})$ & $\epsilon_{abs}(\hat{y})$ & $\epsilon_{rel}(\hat{y})$ \\
        \midrule
        $\alpha = 0.1$, $y_d = 0.3$ & $85$ / $49$ & $1.33 \times 10^{-3}$ & $5.62 \times 10^{-3}$ & $4.89 \times 10^{-4}$ & $1.98 \times 10^{-3}$ \\
        $\alpha = 0.01$, $y_d = 0.3$ & \hspace{-0.5em}$509$ / $316$\hspace{-0.5em} & $7.97 \times 10^{-3}$ & $2.87 \times 10^{-2}$ & $3.02 \times 10^{-3}$ & $1.16 \times 10^{-2}$ \\
        $\alpha = 0.1$, $y_d = 0.2$ & $78$ / $45$ & $1.13 \times 10^{-3}$ & $7.07 \times 10^{-3}$ & $4.57 \times 10^{-4}$ & $2.73 \times 10^{-3}$ \\
        \bottomrule
    \end{tabular}
    \label{tab:burger-numerical-results-nopc}
    \normalsize
\end{table}

Additionally, we fix $\alpha=0.1$ and $y_d=0.3$, and test \Cref{alg:burgers-pdhg-opl}, \Cref{alg:burgers-pdhg-opl}-R, the SSN-FEM in \cite{de2004comparison}, and the ADMM-PINNs in \cite{song2024admm} (which is also reviewed in Section \ref{sec:admm}) on different grid resolutions.
Among the ADMM-PINNs proposed in \cite{song2024admm}, we only include the ADMM-OtA-PINNs in our comparison for succinctness, as it is numerically more accurate than the other methods discussed therein.
We employ 6 SSN iterations and 20 outer ADMM iterations, respectively.
We set $\gamma = 0.1$ in \eqref{eq:admm-pinn} for the ADMM-PINNs.
The subproblem \eqref{eq:admm-pinn-u} is solved in each iteration by PINNs with $5000$ training epochs and an Adam optimizer with a learning rate of $0.001$.

Some results are listed in \Cref{tab:burger-numerical-results-compare}. It is known that the computational complexity of SSN-FEM scales as $O(m^3)$, owing to the requirement of solving distinct size-$m$ linear systems introduced by FEM at each iteration. Hence, its computation time grows dramatically as the grid resolution $m$ increases.
On the other hand, thanks to the generalization ability of DeepONets, the computation time of Algorithms \ref{alg:burgers-pdhg-opl} and \ref{alg:burgers-pdhg-opl}-R is independent of the grid resolution.
Thus, both Algorithms \ref{alg:burgers-pdhg-opl} and \ref{alg:burgers-pdhg-opl}-R become much faster than SSN-FEM over high-resolution grids, specifically by two orders of magnitude on the finest grid.
The relative errors $\epsilon_{rel}(\hat{u})$ and $\epsilon_{rel}(\hat{y})$ show that Algorithms \ref{alg:burgers-pdhg-opl} and \ref{alg:burgers-pdhg-opl}-R are comparable to SSN-FEM in terms of numerical accuracy.
Meanwhile, compared with ADMM-PINNs, both Algorithms \ref{alg:burgers-pdhg-opl} and \ref{alg:burgers-pdhg-opl}-R are far more efficient for all grid resolutions, with only a marginally lower accuracy. Recall from Section \ref{sec:admm} that the ADMM-PINNs requires retraining neural networks at each iteration, which is the major reason for its slow performance.

\begin{table}[!ht]
	\centering
	\small
	\caption{Numerical comparisons for the SSN-FEM, ADMM-PINNs (the ADMM-OtA-PINNs variant), {\Cref{alg:burgers-pdhg-opl}}, and {\Cref{alg:burgers-pdhg-opl}}-R ($\rho=1.8$) evaluated on different grid resolutions $m$.}
	\begin{tabular}{c c c c c }
		\toprule
		\multirow{2}[3]{*}{$m$} & \multicolumn{4}{c}{Computation Time (s)} \\ [0.5ex]
		\cmidrule(lr){2-5}
		~ & SSN-FEM & ADMM-PINNs & {\Cref{alg:burgers-pdhg-opl}} &{\Cref{alg:burgers-pdhg-opl}-R} \\
		\midrule
		$100$ & $0.0669$ & $1641.5721$ & $0.2121$ & $0.1687$ \\
		$500$ & $0.1870$ & $1644.5827$ & $0.2304$ & $0.1771$ \\
		$1000$ & $0.4663$ & $1595.3219$ & $0.2358$ & $0.1812$ \\
		$2000$ & $2.3502$ & $1726.1786$ & $0.2418$ & $0.1855$ \\
		$5000$ & $10.0893$ & $1671.7742$ & $0.2354$ & $0.1739$ \\
		$10000$ & $57.8353$ & $1669.3511$ & $0.2492$ & $0.1726$  \\
		\midrule
		\multirow{2}[3]{*}{$m$} & \multicolumn{4}{c}{$\epsilon_{rel}(\hat{u})$}  \\ [0.5ex]
		\cmidrule(lr){2-5}
		~ & SSN-FEM & ADMM-PINNs & {\Cref{alg:burgers-pdhg-opl}} & {\Cref{alg:burgers-pdhg-opl}-R} \\
		\midrule
		$100$ & $\sim$ & $1.4018 \times 10^{-2}$ & $1.6164 \times 10^{-2}$ & $ 1.6127 \times 10^{-2}$ \\
		$500$ & $\sim$ & $2.8268 \times 10^{-3}$ & $6.7103 \times 10^{-3}$ & $ 6.7025 \times 10^{-3}$ \\
		$1000$ & $\sim$ & $1.4613 \times 10^{-3}$ & $5.9249 \times 10^{-3}$ & $ 5.9204 \times 10^{-3}$ \\
		$2000$ & $\sim$ & $7.7310 \times 10^{-4}$ & $5.6274 \times 10^{-3}$ & $ 5.6245 \times 10^{-3}$ \\
		$5000$ & $\sim$ & $3.2740 \times 10^{-4}$ & $5.4891 \times 10^{-3}$ & $5.4884 \times 10^{-3}$ \\
		$10000$ & $\sim$ & $3.0953 \times 10^{-4}$ & $5.4498 \times 10^{-3}$ & $5.4493 \times 10^{-3}$ \\
		\midrule
		\multirow{2}[3]{*}{$m$} & \multicolumn{4}{c}{$\epsilon_{rel}(\hat{y})$}  \\ [0.5ex]
		\cmidrule(lr){2-5}
		~ & SSN-FEM & ADMM-PINNs & {\Cref{alg:burgers-pdhg-opl}} & {\Cref{alg:burgers-pdhg-opl}-R} \\
		\midrule
		$100$ & $\sim$ & $1.6817 \times 10^{-3}$ & $2.8071 \times 10^{-3}$ & $ 2.8140 \times 10^{-3}$ \\
		$500$ & $\sim$ & $3.6155 \times 10^{-4}$ & $2.0543 \times 10^{-3}$ & $ 2.0569 \times 10^{-3}$ \\
		$1000$ & $\sim$ & $1.6218 \times 10^{-4}$ & $2.0024 \times 10^{-3}$ & $ 2.0045 \times 10^{-3}$ \\
		$2000$ & $\sim$ & $8.7682 \times 10^{-5}$ & $1.9810 \times 10^{-3}$ & $ 1.9826 \times 10^{-3}$ \\
		$5000$ & $\sim$ & $2.8611 \times 10^{-5}$ & $1.9688 \times 10^{-3}$ & $1.9702 \times 10^{-3}$ \\
		$10000$ & $\sim$ & $4.4054 \times 10^{-5}$ & $1.9650 \times 10^{-3}$ & $1.9664 \times 10^{-3}$ \\
		\bottomrule
	\end{tabular}
	\label{tab:burger-numerical-results-compare}
\end{table}

\section{Sparse bilinear control of parabolic equations}\label{sec:bp-opt-ctrl}

In this section, we showcase a time-dependent example of model~\eqref{eq:basic-problem} with sparsity regularization and delineate the implementation of \Cref{alg:pdhg-opl} in this context.
We consider the following sparse bilinear control problem:
\begin{equation}\label{eq:bilinparab-optctrl}
    \begin{aligned}
        \min_{u\in L^2(\Omega_T), y\in L^2(\Omega_T)} \quad &\frac{1}{2}\|y - y_d\|^2_{L^2(\Omega_T)} + \frac{\alpha}{2}\|u\|^2_{L^2(\Omega_T)} + \beta \|u\|_{L^1(\Omega_T)} + I_{U_{ad}}(u),
    \end{aligned}
\end{equation}
subject to the parabolic state equation
\begin{equation}\label{eq:bilinparab-equation}
         \partial_t y - \Delta y + uy = f \text{~in~} \Omega\times(0,T), \quad
          y = 0 \text{~on~} \partial \Omega \times (0, T), \quad
         y(0) = 0  \text{~in~} \Omega.
\end{equation}
Here, $\Omega \subset \mathbb{R}^d$ is a bounded domain ($d\geq 1$), $\partial\Omega$ is the Lipschitz continuous boundary of $\Omega$, and $\Omega_T := \Omega \times (0, T)$ with $0<T<+\infty$. The constants $\alpha > 0$, $\beta \geq 0$ are regularization parameters, $y_d \in L^2(\Omega_T)$ is a prescribed desired state, $f\in L^2(\Omega_T)$ is a given source term, and $I_{U_{ad}}(u)$ is the indicator function of the admissible set $U_{ad} := \{u \in L^\infty(\Omega_T): a \leq u(x, t) \leq b \text{~a.e.~in~} \Omega_T \}.$
The $L^1$-regularization term $\beta\|u\|_{L^1(\Omega_T)}$ is convex but not Fr\'echet differentiable.
Problem~\eqref{eq:bilinparab-optctrl}--\eqref{eq:bilinparab-equation} admits an optimal control with sparse support due to the presence of the $L^1$-regularization~\cite{casas2017review,stadler2009elliptic}.

\subsection{Primal-dual decoupling for~\texorpdfstring{\eqref{eq:bilinparab-optctrl}--\eqref{eq:bilinparab-equation}}{}}\label{se:pd_bi}

We delineate the primal-dual decoupling \eqref{eq:pdhg-nonlinear} of \eqref{eq:bilinparab-optctrl}--\eqref{eq:bilinparab-equation} by elaborating on the computations of the proximal operators $\prox_{\sigma F^*}$, $\prox_{\tau G}$, and the terms $S(2u^{k+1} - u^k)$ and $(S'(u^k))^*p^k$.
For this purpose, we first specify the functionals $F$ and $G$ defined in Section \ref{sec:pd-decoup} as
\begin{equation} \label{eq:bilinparab-FG}
    \begin{aligned}
        F(y)  = \frac{1}{2}\|y - y_d\|^2_{L^2(\Omega_T)}, ~
        G(u)  = \frac{\alpha}{2}\|u\|^2_{L^2(\Omega_T)} + \beta \|u\|_{L^1(\Omega_T)} + I_{U_{ad}}(u),
    \end{aligned}
\end{equation}
We then have the following results regarding $\prox_{\sigma F^*}$ and $\prox_{\tau G}$.
\begin{proposition}\label{prop:bilinparab-prox}
    Let $F, G$ be defined as in~\eqref{eq:bilinparab-FG}.
    Then, for any $\sigma > 0$, the proximal operator $\prox_{\sigma F^*}$ is given by
    \begin{equation}\label{eq:bilinparab-prox-F-conj}
            \prox_{\sigma F^*}(y) = \dfrac{1}{1 + \sigma} (y - \sigma y_d) \quad \forall y\in L^2(\Omega_T).
    \end{equation}
    For any $u\in L^2(\Omega_T)$, $\tau > 0$ and $(x, t)^\top \in \Omega_T$, the proximal operator $\prox_{\tau G}$ is given by
    \begin{equation}\label{eq:bilinparab-prox-G}
        \prox_{\tau G}(u)(x, t) = \begin{cases}
            \cP_{[a, b]} \left( \dfrac{u(x, t) - \tau \beta}{\alpha \tau + 1} \right)\quad & \text{~if~} u(x, t) \geq \tau \beta, \\
            \cP_{[a, b]} (0) & \text{~if~} |u(x, t)| < \tau \beta ,\\
            \cP_{[a, b]} \left( \dfrac{u(x, t) + \tau \beta}{\alpha \tau + 1} \right) \quad & \text{~if~} u(x, t) \leq -\tau \beta,
        \end{cases}
    \end{equation}
    where $\cP_{[a, b]}$ denotes the projection onto the interval $[a, b]$.
\end{proposition}

It is easy to see that \eqref{eq:bilinparab-prox-F-conj} and~\eqref{eq:bilinparab-prox-G} involve only pointwise operations, hence $\prox_{\sigma F^*}$ and $\prox_{\tau G}$ are straightforward to compute.

Note that $S: L^2(\Omega_T)\rightarrow L^2(\Omega_T)$ is now specified as the solution operator of the state equation~\eqref{eq:bilinparab-equation}.
Regarding the evaluation of $S(2u^{k+1} - u^k)$ and $(S'(u^k))^* p^k$, $S(2u^{k+1} - u^k)$ reduces to solving the state equation~\eqref{eq:bilinparab-equation} with $u = 2u^{k+1} - u^k$; and for $(S'(u^k))^* p^k$, we have the following result:
\begin{proposition}\label{prop:adjoint_bi}
   Consider problem \eqref{eq:bilinparab-optctrl}--\eqref{eq:bilinparab-equation}. For any $u \in L^2(\Omega_T)$ and $p \in L^2(\Omega_T)$, let $z\in L^2(\Omega_T)$ be the solution of the following adjoint equation:
    \begin{equation}\label{eq:bilinparab-equation-adj}
             -\partial_t z - \Delta z + uz = p  \text{~in~} \Omega\times (0,T),  \quad
             z = 0  \text{~on~} \partial \Omega \times (0, T),\quad
                         z(T) = 0 \text{~in~} \Omega.
    \end{equation}
    Then, we have $(S'(u))^* p = -(S(u)) z$.
\end{proposition}

\Cref{prop:adjoint_bi} can be proved in a similar way to \Cref{prop:burgers-adjoint} and we thus omit the proof here for succinctness. Computing $S(2u^{k+1}-u^k)$ and $(S'(u^k))^* p^k$ thus reduces to solving \eqref{eq:bilinparab-equation} and~\eqref{eq:bilinparab-equation-adj}.

\subsection{Fourier neural operator (FNO)}\label{sec:fno}

We now apply the Fourier neural operator (FNO)~\cite{li2021fourier} to approximate the solution operators $S$ of \eqref{eq:bilinparab-equation}.
As illustrated in~\Cref{fig:fno}, the input of the FNO is a function $u$ discretized over a grid $\mathcal{D} \subset \Omega_T$ and represented as a vector in $\mathbb{R}^{|\mathcal{D}|}$.
The input function is first lifted to a vector $\cP\in \mathbb{R}^{|\mathcal{D} | \times m_p}$ by a linear layer, then transformed through several Fourier layers to generate a vector $\cQ\in\mathbb{R}^{|\mathcal{D} | \times m_p}$, and finally projected back to $\mathbb{R}^{|\mathcal{D}|}$.
In each Fourier layer, the transformations are simultaneously performed in the time domain and the frequency domain.
In the time domain, the input is transformed by applying a linear layer $W$.
Meanwhile, the input is carried into the frequency domain by a Fourier transform $\cF$, truncated into $k$ Fourier modes, transformed by an elementwise multiplication with a complex matrix $\mathcal{R}$, and then projected back to the time domain by the inverse Fourier transform $\cF^{-1}$.
When the grid $\mathcal{D}$ is uniform on $\Omega_T$, the Fourier transform $\mathcal{F}$ and its inverse $\mathcal{F}^{-1}$ can be efficiently computed by the fast Fourier transform algorithm.
We refer to~\cite{kovachki2023neural,li2021fourier} for a more detailed description.

\begin{figure}[!ht]
    \centering
\begin{tikzpicture}[global scale =0.40]
\centering

\node at(-0.5,1)      (1)[rectangle, minimum width =600pt, minimum height =230pt, inner sep=3pt,draw=black]  {};

\node at(-10,0.7)   (2) [circle,draw=black,fill=blue!30,minimum width =40pt, minimum height =40pt,font=\fontsize{20}{20}\selectfont]{$v$};

\node at(-2,-2)   (3) [circle,draw=black,fill=orange!55,minimum width =40pt, minimum height =40pt,font=\fontsize{20}{20}\selectfont]{$W$};

\node at(6,0.7)   (4) [circle,draw=black,fill=orange!100,minimum width =40pt, minimum height =40pt,font=\fontsize{20}{20}\selectfont]{$+$};

\node at(9,0.7)   (5) [circle,draw=black,fill=orange!100,minimum width =40pt, minimum height =40pt,font=\fontsize{20}{20}\selectfont]{$\sigma$};

\node at(-1.9,2.8)      (6)[rectangle, fill=LimeGreen!30,minimum width =360pt, minimum height =100pt, inner sep=5pt,draw=black,font=\fontsize{20}{20}\selectfont]  {};

\node at(-10.5,8)   (7) [circle,draw=black,fill=blue!30,minimum width =35pt, minimum height =35pt,font=\fontsize{20}{20}\selectfont]{$u$};

\node at(-8,8)   (8) [circle,draw=black,fill=orange!55,minimum width =35pt, minimum height =35pt,font=\fontsize{20}{20}\selectfont]{$\mathcal{P}$};

\node at(-5,8)      (9)[rectangle,fill=LimeGreen!30, minimum width =70pt, minimum height =30pt, inner sep=5pt,draw=black,font=\fontsize{16}{16}\selectfont]  {$Layer~1$};

\node at(-1,8)      (10)[rectangle,fill=LimeGreen!30, minimum width =70pt, minimum height =30pt, inner sep=5pt,draw=black,font=\fontsize{16}{16}\selectfont]  {$Layer~2$};

\node at(5,8)      (12)[rectangle,fill=LimeGreen!30, minimum width =70pt, minimum height =30pt, inner sep=5pt,draw=black,font=\fontsize{16}{16}\selectfont]  {$Layer~L$};

\node at(8,8)   (13) [circle,draw=black,fill=orange!55,minimum width =35pt, minimum height =35pt,font=\fontsize{20}{20}\selectfont]{$\mathcal{Q}$};

\node at(10,8)   (14) [circle,draw=black,fill=blue!30,minimum width =35pt, minimum height =35pt,font=\fontsize{20}{20}\selectfont]{$y$};

\node at(-7.5,3)   (15) [circle,draw=black,fill=orange!20,minimum width =40pt, minimum height =40pt,font=\fontsize{20}{20}\selectfont]{$\mathcal{F}$};
\node at(-2,3)   (16) [circle,draw=black,fill=orange!55,minimum width =40pt, minimum height =40pt,font=\fontsize{20}{20}\selectfont]{$\mathcal{R}$};
\node at(3.6,3)   (17) [circle,draw=black,fill=orange!20,minimum width =40pt, minimum height =40pt,font=\fontsize{18}{18}\selectfont]{$\mathcal{F}^{-1}$};
\node at(7.5,4)[font=\fontsize{16}{16}\selectfont]{Fourier Layer};

\filldraw (1.7,8) circle (.1) ;
\filldraw (2,8) circle (.1) ;
\filldraw (2.3,8) circle (.1) ;

\draw[eaxis,xscale=0.5] (-13,1.6) -- (-6,1.6);
\draw[eaxis,xscale=0.5] (-13,2.4) -- (-6,2.4);
\draw[eaxis,xscale=0.5] (-13,3.2) -- (-6,3.2);
\draw[eaxis,xscale=0.5] (-13,4) -- (-6,4);

\draw[eaxis,xscale=0.5] (-2.2,4) -- (4.8,4);
\draw[eaxis,xscale=0.5] (-2.2,3.2) -- (4.8,3.2);

\draw[eaxis,yscale=0.4] (-4.75,3.2) -- (-4.75,4.8);
\draw[eaxis,yscale=0.4] (-4.75,5.2) -- (-4.75,6.8);
\draw[eaxis,yscale=0.4] (-4.75,7.2) -- (-4.75,8.8);
\draw[eaxis,yscale=0.4] (-4.75,9.2) -- (-4.75,10.8);

\draw[eaxis,yscale=0.4] (0.5,7.2) -- (0.5,8.8);
\draw[eaxis,yscale=0.4] (0.5,9.2) -- (0.5,10.8);

\draw[elegant,black,line width=0.5,domain=-\num:\num,xscale=0.5,yscale=0.3,xshift=-270pt,yshift=380pt] plot(\x,{sin(\x r)});
\draw[elegant,black,line width=0.5,domain=-\num:\num,xscale=0.5,yscale=0.3,xshift=-270pt,yshift=300pt] plot(\x,{sin(2*\x r)});
\draw[elegant,black,line width=0.5,domain=-\num:\num,xscale=0.5,yscale=0.3,xshift=-270pt,yshift=230pt] plot(\x,{sin(3*\x r)});
\draw[elegant,black,line width=0.5,domain=-\num:\num,xscale=0.5,yscale=0.3,xshift=-270pt,yshift=150pt] plot(\x,{sin(4*\x r)});

\draw[elegant,black,line width=0.5,domain=-\num:\num,xscale=0.5,yscale=0.3,xshift=28pt,yshift=380pt] plot(\x,{-sin(\x r)});
\draw[elegant,black,line width=0.5,domain=-\num:\num,xscale=0.5,yscale=0.3,xshift=28pt,yshift=300pt] plot(\x,{-sin(2*\x r)});

\draw[->][line width=0.9 pt] (2) --(-8.3,2.9);
\draw[->][line width=0.9 pt] (4.4,2.9) --(4);
\draw[->][line width=1 pt] (4) --(5);
\draw[->][line width=0.9 pt] (2) --(3);
\draw[->][line width=0.9 pt] (3) --(4);
\draw[->][line width=0.9 pt](7) --(8);
\draw[->][line width=0.9 pt] (8) --(9);
\draw[->][line width=0.9 pt] (9) --(10);
\draw[->][line width=0.9 pt] (12) --(13);
\draw[->][line width=0.9 pt] (10) --(1.55,8);
\draw[->][line width=0.9 pt] (2.5,8) --(12);
\draw[->][line width=0.9 pt] (13) --(14);
\draw[densely dashed,line width=0.9 pt](-11,5) --(-2,7.5);
\draw[densely dashed,line width=0.9 pt](10,5) --(0,7.5);
\end{tikzpicture}
\caption{The structure of the Fourier neural operator~\cite{li2021fourier}.}
\label{fig:fno}
\end{figure}
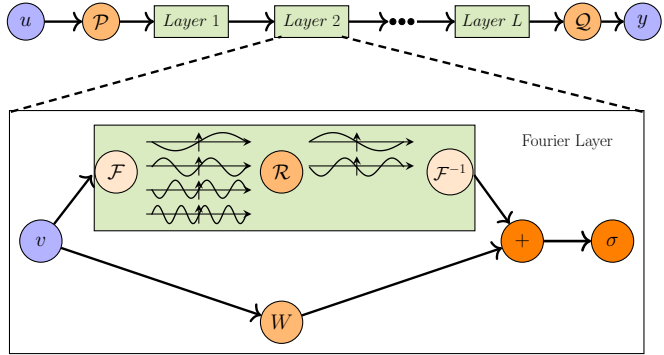

We construct an FNO $\cS_{\theta_s}$ and approximate the solution $y$ of \eqref{eq:bilinparab-equation} by $\cS_{\theta_s}(u)$.
To train $\cS_{\theta_s}$,
we generate a set of training data $\{((u_i, f_i), y_i)\}_{i=1}^N$, where $u_i, f_i \in U$ are randomly sampled functions, and each $y_i \in Y$ is the corresponding solution of \eqref{eq:bilinparab-equation}.
All the functions are discretized over a meshgrid denoted by $\overline{\mathcal{D}} \subset \overline{\Omega_T}$.
The neural network $\cS_{\theta_s}$ can be trained with the following loss function:
\begin{equation}\label{eq:mifno-loss}
	\cL(\theta_s) = \frac{1}{N}\sum_{i=1}^N \frac{\sum_{(x, t)^\top \in \overline{\mathcal{D}}} \left| \cS_{\theta_s}(u_i, f_i)(x, t) - y_i(x, t) \right|^2}{\sum_{(x, t)^\top \in \overline{\mathcal{D}}} \left|y_i(x, t)\right|^2}.
\end{equation}
Let $\theta_s^*$ be the parameter obtained by minimizing the loss function \eqref{eq:mifno-loss}.
Then, the solution $y$ of \eqref{eq:bilinparab-equation} can be evaluated by $\cS_{\theta_s^*}(u, f)$.

It is worth noting that the trained FNO $\cS_{\theta_s^*}$ can be re-employed to approximate the solution $z$ of~\eqref{eq:bilinparab-equation-adj}.
Indeed, by letting $\tilde{t} = T - t$, \eqref{eq:bilinparab-equation-adj} can be transformed into the following equation:
    \begin{equation*}
	\partial_{\tilde{t}} z - \Delta z + uz = p \text{~in~} \Omega\times (0,T), \quad
	 z = 0 \text{~on~} \partial \Omega \times (0, T),\quad
		 z(0) = 0  \text{~in~} \Omega,
\end{equation*}
which has the same structure as that of \eqref{eq:bilinparab-equation}.
Hence, with the FNO $\cS_{\theta_s^*}$, the solution $z$ of~\eqref{eq:bilinparab-equation-adj} can be evaluated by
    $\cS_{\theta_s^*}(u_{\text{rev}}, p_{\text{rev}})$ with $u_\text{rev}(x,t)=u(x,T-t)$ and $p_\text{rev}(x,t)=p(x,T-t)$.
Following from Proposition \ref{prop:adjoint_bi}, $S'(u^k)^*p^k$ can be approximated by $-\cS_{\theta_s^*}(u^k, f)\cS_{\theta_s^*}(u^k_\text{rev},p^k_\text{rev})$.
Therefore, it suffices to train a single FNO $\cS_{\theta_s^*}$ for the implementation of \Cref{alg:pdhg-opl} to problem \eqref{eq:bilinparab-optctrl}--\eqref{eq:bilinparab-equation}.

\subsection{A primal-dual-based operator learning method for~\texorpdfstring{\eqref{eq:sp-prob}--\eqref{eq:sp-equation}}{(5.1)-(5.2)}}

With the discussions in Sections \ref{se:pd_bi} and \ref{sec:fno}, we conclude in \Cref{alg:bilinparab-pdhg-opl} a primal-dual-based operator learning method for solving~\eqref{eq:bilinparab-optctrl}--\eqref{eq:bilinparab-equation}.

\begin{algorithm}[!ht]
	\caption{A primal-dual-based operator learning method for \eqref{eq:bilinparab-optctrl}--\eqref{eq:bilinparab-equation}}
	\begin{algorithmic}[1]
	    \REQUIRE Pre-trained neural solution operator $\cS_{\theta_s^*}$, initialization $u^0$, $p^0$, and stepsizes $\tau, \sigma>0$.
        \FOR {$k = 0, 1, \ldots $}
            \STATE $y^k = \cS_{\theta_s^*}(u^k, f)$.
            \STATE $u^k_{\text{rev}} (x, t) = u^k(x, T-t)$, \quad $p^k_{\text{rev}} (x, t) = p^k(x, T-t)$.
            \STATE $u^{k+1} = \prox_{\tau G} (u^k + \tau y^k \cS_{\theta_s^*}(u^k_{\text{rev}}, p^k_{\text{rev}}))$, where $\prox_{\tau G}$ is calculated by \eqref{eq:bilinparab-prox-G}.
            \STATE $p^{k+1} = \frac{1}{1 + \sigma} \left( p^k + \sigma \cS_{\theta_s^*}(2u^{k+1} -u^k, f) - \sigma y_d \right) $.
                   \ENDFOR
	    \ENSURE Numerical solution $(\hat{u}, \hat{y}):=(u^{k+1}, \cS_{\theta_s^*}(u^{k+1},f))$ to \eqref{eq:bilinparab-optctrl}--\eqref{eq:bilinparab-equation}.
	\end{algorithmic}
	\label{alg:bilinparab-pdhg-opl}
\end{algorithm}

\subsection{Numerical experiments}

In this subsection, we demonstrate the implementation of \Cref{alg:bilinparab-pdhg-opl} to problem \eqref{eq:bilinparab-optctrl}--\eqref{eq:bilinparab-equation} and validate its effectiveness and efficiency. We set $\Omega=(0,1)^2$, $T=1$,
$a = -1$, $b = 2$, and test on different parameters $\alpha$ and $\beta$ that will be specified later.
We generate the problem data $f$ and $y_d$ following the procedure introduced in~\cite[Section 6, Procedure 2]{schindele2017proximal} so that \eqref{eq:bilinparab-optctrl}--\eqref{eq:bilinparab-equation} admits a closed-form optimal solution.
Specifically, let
\begin{equation*}
    \begin{aligned}
        y^*(x, t)  = 5 \sqrt{\beta} t \sin(3\pi x_1) \sin(\pi x_2),\quad
        p^*(x, t)   = 5\sqrt{\beta} (t-1) \sin(\pi x_1) \sin(\pi x_2),
    \end{aligned}
\end{equation*}
and
\begin{equation*}
    \begin{aligned}
        & u^*(x, t) = \begin{cases}
            \max \left\{\frac{-p^*(x, t) y^*(x, t) + \beta}{\alpha}, a\right\} \quad & \text{~if~} p^*(x, t) y^*(x, t) > \beta, \\
            \min \left\{\frac{-p^*(x, t) y^*(x, t) - \beta}{\alpha}, b\right\} \quad & \text{~if~} p^*(x, t) y^*(x, t) < -\beta, \\
            0 \quad & \text{~otherwise}.
        \end{cases} \\
        & f = \partial_t  y^* - \Delta y^* + u^* y^*, \quad y_d  = y^* + (-\partial_t p^* - \Delta p^* + u^* p^*).
    \end{aligned}
\end{equation*}
Then, $(u^*, y^*)^\top$ is an optimal solution of problem \eqref{eq:bilinparab-optctrl}--\eqref{eq:bilinparab-equation}.

To generate a training set for $\cS_{\theta_s}$, we discretize $\overline{\Omega_T}$ by an $m \times m \times m_T$ equi-spaced grid $\overline{\mathcal{D}} = \{ ((x_1)_i, (x_2)_j, t_k) \in \overline{\Omega_T} : 1 \leq i, j \leq m, \ 1 \leq k \leq m_T \}$ with $m = 64$ and $m_T = 64$.
Using the Julia package \texttt{GuassianRandomFields}, we sample $N=2048$ functions $\{f_i\}_{i=1}^N$, discretized over $\overline{\Omega_T}$, from the Gaussian random field
    $f_i \sim \mathcal{GR}\left(0, C((x, t)^\top, (\tilde{x}, \tilde{t})^\top)\right)$
with the Gaussian covariance kernel $C((x, t)^\top, (\tilde{x}, \tilde{t})^\top) = 5 \exp (-\lVert (x, t)^\top - (\tilde{x}, \tilde{t})^\top \rVert _2^2 / 0.09 )$.
To obtain $\{u_i\}_{i=1}^N$, we first sample $N=2048$ functions $\{v_i\}_{i=1}^N$ according to $v_i \sim \mathcal{GR}(0, 2024(-\Delta + 9I)^{-1.5})$ with zero Dirichlet boundary conditions on the Laplacian.
Then, for any $(x, t)^\top \in \Omega_T$, we take
\begin{equation}\label{eq:u_con_bi}
    u_i(x, t) = \begin{cases}
        \max \left\{ \frac{v_i(x, t) - 3}{4}, 2 \right\} \quad & \text{~if~} v_i(x, t) > 3, \\
        \min \left\{ \frac{v_i(x, t) + 3}{4}, -1  \right\} \quad & \text{~if~} v_i(x, t) < 3, \\
        0 \quad & \text{~otherwise}.
    \end{cases}
\end{equation}
The constructed $u_i$ in \eqref{eq:u_con_bi} satisfies $u_i \in U_{ad}$ and is sparse, which is consistent with the property of $u^k$ in each iteration of \Cref{alg:bilinparab-pdhg-opl}.
For each generated $(u_i, f_i)$, we solve the state equation~\eqref{eq:bilinparab-equation} by the \texttt{solvepde} function in MATLAB's \texttt{Partial Differential Equation Toolbox} to get a solution $y_i$.
Then, $\{(u_i, f_i, y_i)\}_{1\leq i \leq N}$ forms a training set for $\cS_{\theta_s}$, for which $N=2048$ PDEs are required to be solved.

We construct the FNO $\cS_{\theta_s}$ with three hidden Fourier layers and take $m_p = 20$ as the lifting dimension.
In each Fourier layer, we set the number of truncated frequency modes $k=8$ and adopt the \texttt{GeLU} activation function.
As suggested in~\cite{li2021fourier}, we pad each non-periodic input with $8$ zero elements in each spatial-temporal dimension. We initialize all the neural networks following the default setting of \texttt{PyTorch}. To train $\cS_{\theta_s}$, we minimize the loss function~\eqref{eq:mifno-loss} by the Adam optimizer with a batch size of $8$ for $300$ epochs. We set the initial stepsize as $0.001$, and decay it exponentially by a factor of $0.5$ for every $50$ epochs.

To evaluate the trained neural solution operator $\cS_{\theta_s^*}$, we generate $256$ testing data $\{(u_i, f_i, y_i)\}_{i=1}^256$ following the same procedure as that for generating the training set.
The absolute and relative errors are reported in~\Cref{tab:bilinparab-model-acc}, which indicate that the trained neural solution operator $\cS_{\theta_s^*}$ achieves satisfactory accuracy in solving~\eqref{eq:bilinparab-equation} and~\eqref{eq:bilinparab-equation-adj}.

\begin{table}[!ht]
    \centering
    \small
    \caption{Testing accuracy of the trained neural solution operator $\cS_{\theta_s^*}$ (Training time: $9.65 \times 10^{3}$ seconds).}
    \begin{tabular}{ c c c }
        \toprule
      ~&$\|\cS_{\theta_s^*}(u)-y\|_{L^2(\Omega_T)}$&$\|\cS_{\theta_s^*}(u)-y\|_{L^2(\Omega_T)}$/$\|y\|_{L^2(\Omega_T)}$ \\
        \midrule
        Mean & $3.93 \times 10^{-4}$ & $3.53 \times 10^{-3}$ \\
        SD  & $1.73 \times 10^{-4}$ & $8.91 \times 10^{-4}$ \\
        \bottomrule
    \end{tabular}
    \label{tab:bilinparab-model-acc}
\end{table}

We then solve the optimal control problem \eqref{eq:bilinparab-optctrl}--\eqref{eq:bilinparab-equation} by \Cref{alg:bilinparab-pdhg-opl} with the trained neural solution operator $\cS_{\theta_s^*}$.
We initialize \Cref{alg:bilinparab-pdhg-opl} with $u^0 = 0$ and $p^0 = 0$ and set the stepsizes $\tau = 500.0$ and $\sigma = 0.4$.
The following termination condition is applied:
\begin{equation} \label{eq:term-cond1}
    \max \left\{ \frac{\|u^{k+1} - u^k\|_{L^2(\Omega_T)}}{\max\{1, \|u^k\|_{L^2(\Omega_T)}\}}, \frac{\|p^{k+1} - p^k\|_{L^2(\Omega_T)}}{\max\{1, \|p^k\|_{L^2(\Omega_T)}\}} \right\} \leq 10^{-5}.
\end{equation}
The absolute and relative errors of the computed state and control, together with the number of iterations, are reported in~\Cref{tab:bilinparab-computed-ctrl-state-acc}, where the absolute error $\epsilon_{abs}$ and the relative error $\epsilon_{rel}$ of $\hat{u}$ and $\hat{y}$ are defined as
\begin{equation*}
	\begin{aligned}
		& \epsilon_{abs}(\hat{u}) := \|\hat{u} - u^*\|_{L^2(\Omega_T)}, \quad \epsilon_{rel}(\hat{u}) := \|\hat{u} - u^*\|_{L^2(\Omega_T)} / \|u^*\|_{L^2(\Omega_T)}, \\
		& \epsilon_{abs}(\hat{y}) := \|\hat{y} - y^*\|_{L^2(\Omega_T)}, \quad \epsilon_{rel}(\hat{y}) := \|\hat{y} - y^*\|_{L^2(\Omega_T)} / \|y^*\|_{L^2(\Omega_T)}.
	\end{aligned}
\end{equation*}
We plot the computed control and the exact control in~\Cref{fig:bilinparab-computed-ctrl-state} for $\alpha = \beta = 0.01$ at $t = 0.25$, $t = 0.5$, and $t = 0.75$.
Similar to the numerical experiment in Section~\ref{sec:burgers}, we further test the combination of \Cref{alg:bilinparab-pdhg-opl} with the relaxation steps \eqref{eq:relaxation-steps}, which is termed \Cref{alg:bilinparab-pdhg-opl}-R.
We choose $\rho = 1.618$ and report the number of iterations required in~\Cref{tab:bilinparab-computed-ctrl-state-acc}.
The absolute and relative errors of \Cref{alg:bilinparab-pdhg-opl}-R are identical to those of \Cref{alg:bilinparab-pdhg-opl}.
From these results, we observe that \Cref{alg:bilinparab-pdhg-opl} and \Cref{alg:bilinparab-pdhg-opl}-R achieve satisfactory accuracy for solving the optimal control problem \eqref{eq:bilinparab-optctrl}--\eqref{eq:bilinparab-equation} with different testing parameters.

\begin{table}[!ht]
    \small
    \centering
    \caption{Numerical results of \Cref{alg:bilinparab-pdhg-opl} and \Cref{alg:bilinparab-pdhg-opl}-R for problem \eqref{eq:bilinparab-optctrl}--\eqref{eq:bilinparab-equation}. The number of iterations of \Cref{alg:bilinparab-pdhg-opl} and \Cref{alg:bilinparab-pdhg-opl}-R are reported before and after the slash (/), respectively.}
    \begin{tabular}[c]{ c c c c c c }
        \toprule
        ~ & \hspace{-0.5em}Iteration\hspace{-0.5em} & $\epsilon_{abs}(\hat{u})$ & $\epsilon_{rel}(\hat{u})$ & $\epsilon_{abs}(\hat{y})$ & $\epsilon_{rel}(\hat{y})$ \\
        \midrule
        $\alpha = 0.1$, $\beta = 0.01$ & $32$ / $19$ & $5.32 \times 10^{-4}$ & $6.08 \times 10^{-3}$ & $9.40 \times 10^{-4}$ & $6.59 \times 10^{-3}$ \\
        $\alpha = 0.01$, $\beta = 0.01$ & $33$ / $23$ & $4.87 \times 10^{-3}$ & $9.52 \times 10^{-3}$ & $1.01 \times 10^{-3}$ & $7.08 \times 10^{-3}$ \\
        \hspace{-0.3em}$\alpha = 0.001$, $\beta = 0.001$\hspace{-0.3em} & $55$ / $34$ & $1.55 \times 10^{-2}$ & $3.04 \times 10^{-2}$ & $6.84 \times 10^{-4}$ & $1.52 \times 10^{-2}$ \\
        \bottomrule
    \end{tabular}
    \label{tab:bilinparab-computed-ctrl-state-acc}
    \normalsize
\end{table}

\begin{figure}[!ht]
    \centering
    \subfloat[Computed control at $t = 0.25$]{%
        \includegraphics[width=0.33\textwidth]{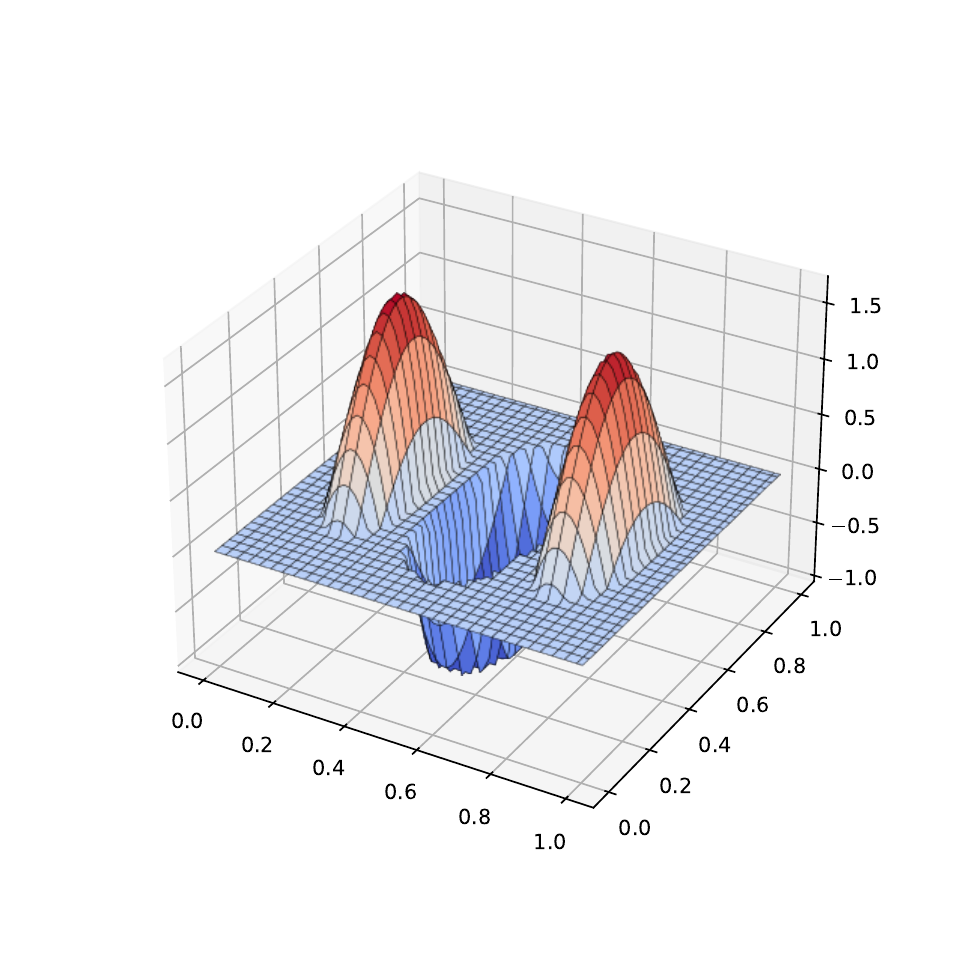}%
    }\hfill
    \subfloat[Computed control at $t = 0.5$]{%
        \includegraphics[width=0.33\textwidth]{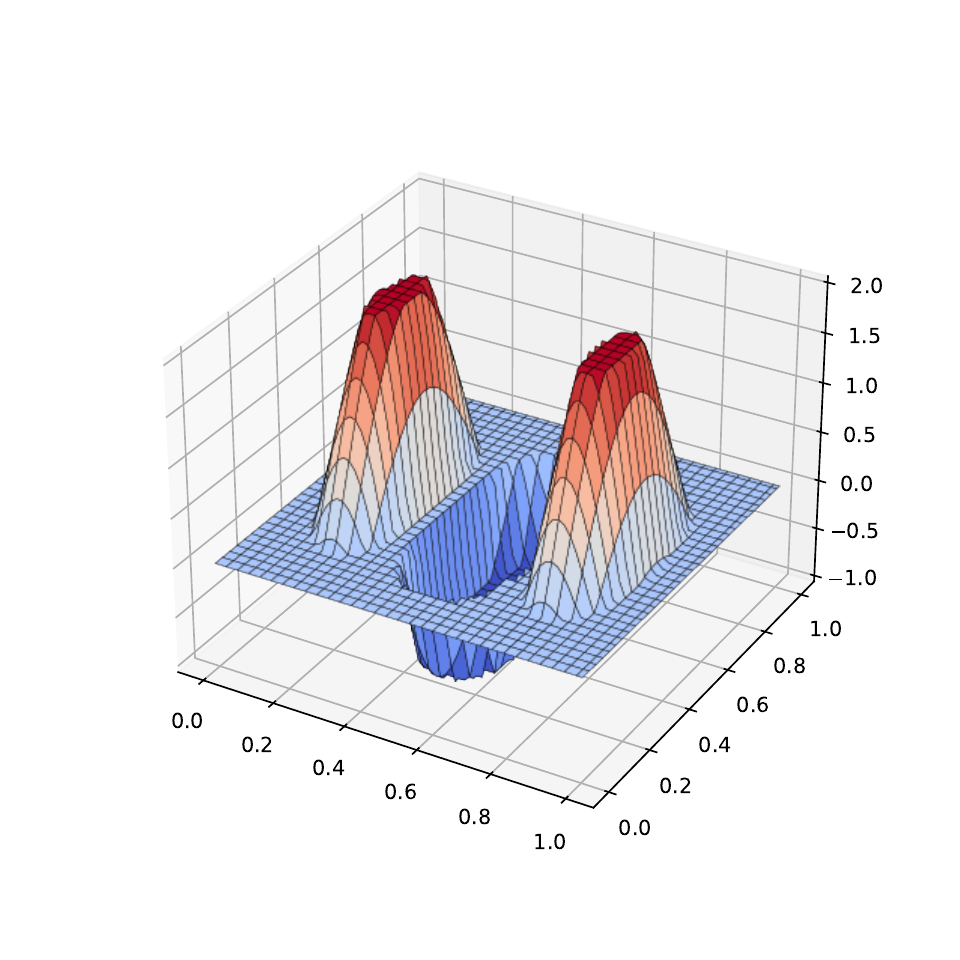}%
    }\hfill
    \subfloat[Computed control at $t = 0.75$]{%
        \includegraphics[width=0.33\textwidth]{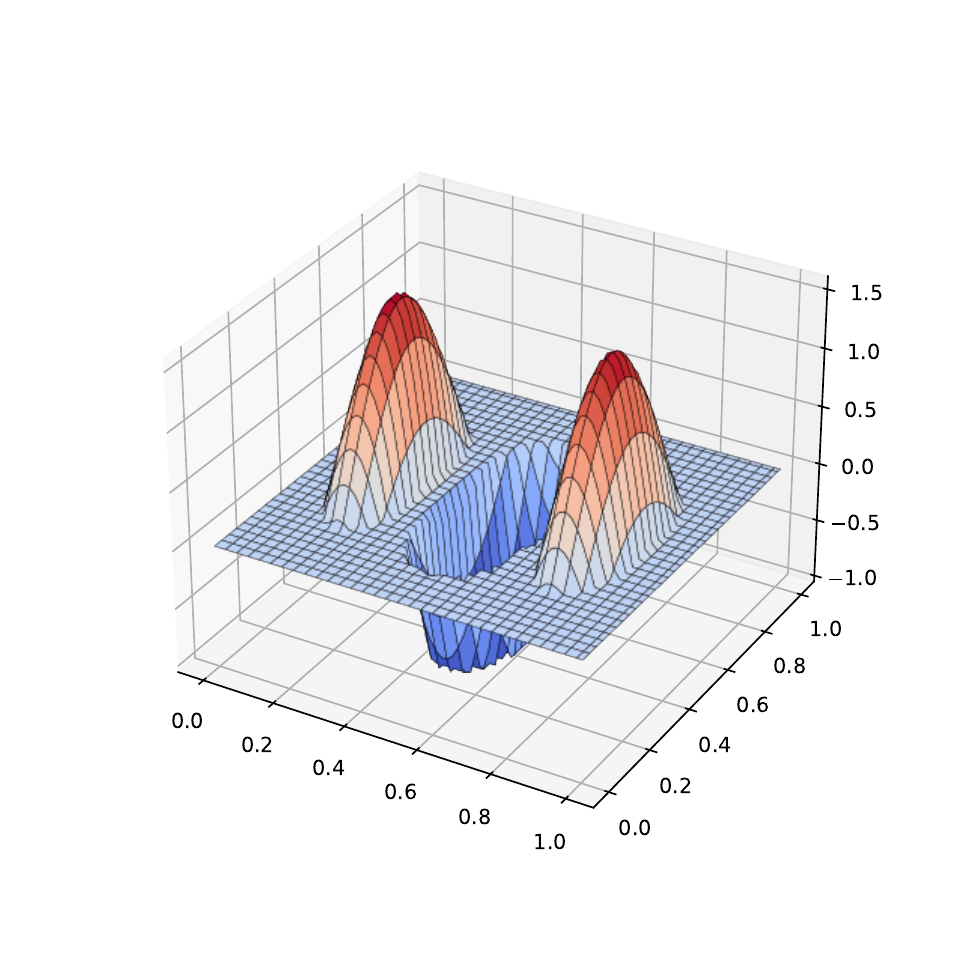}%
    }\hfill
    \subfloat[Error of control at $t = 0.25$]{%
        \includegraphics[width=0.32\textwidth]{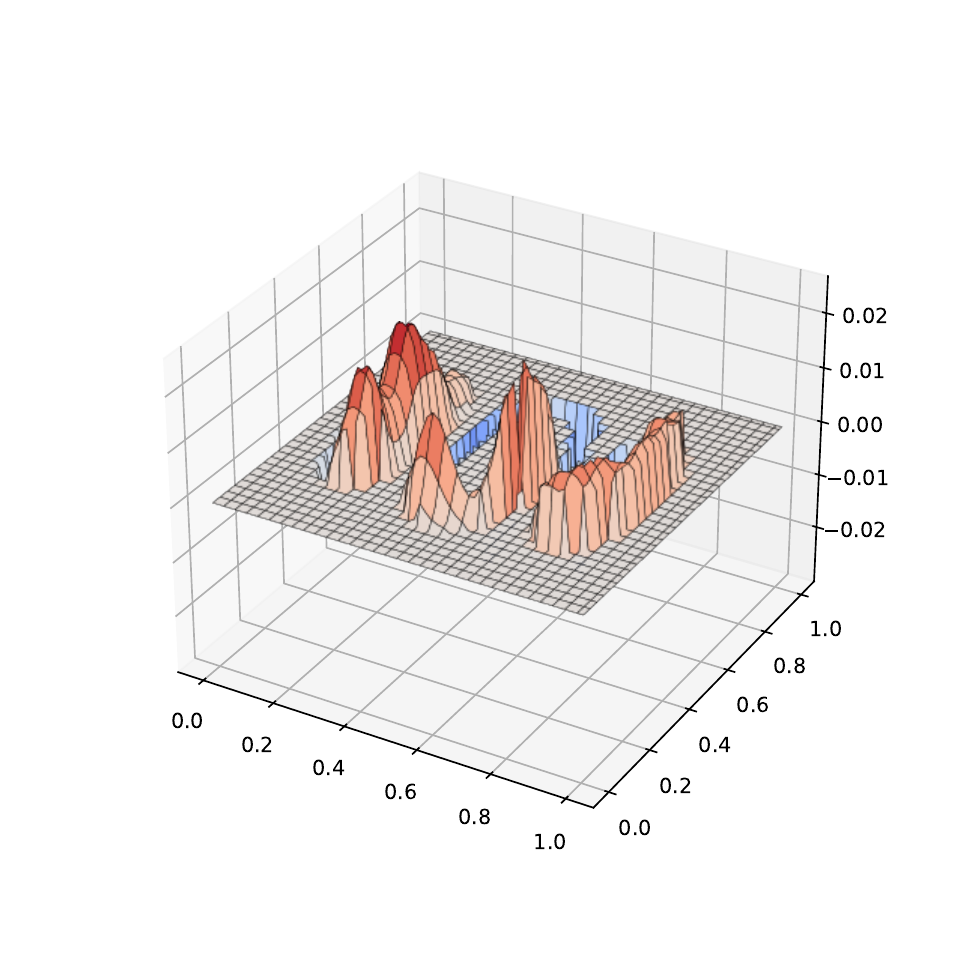}
    }\hfill
    \subfloat[Error of control at $t = 0.5$]{%
        \includegraphics[width=0.32\textwidth]{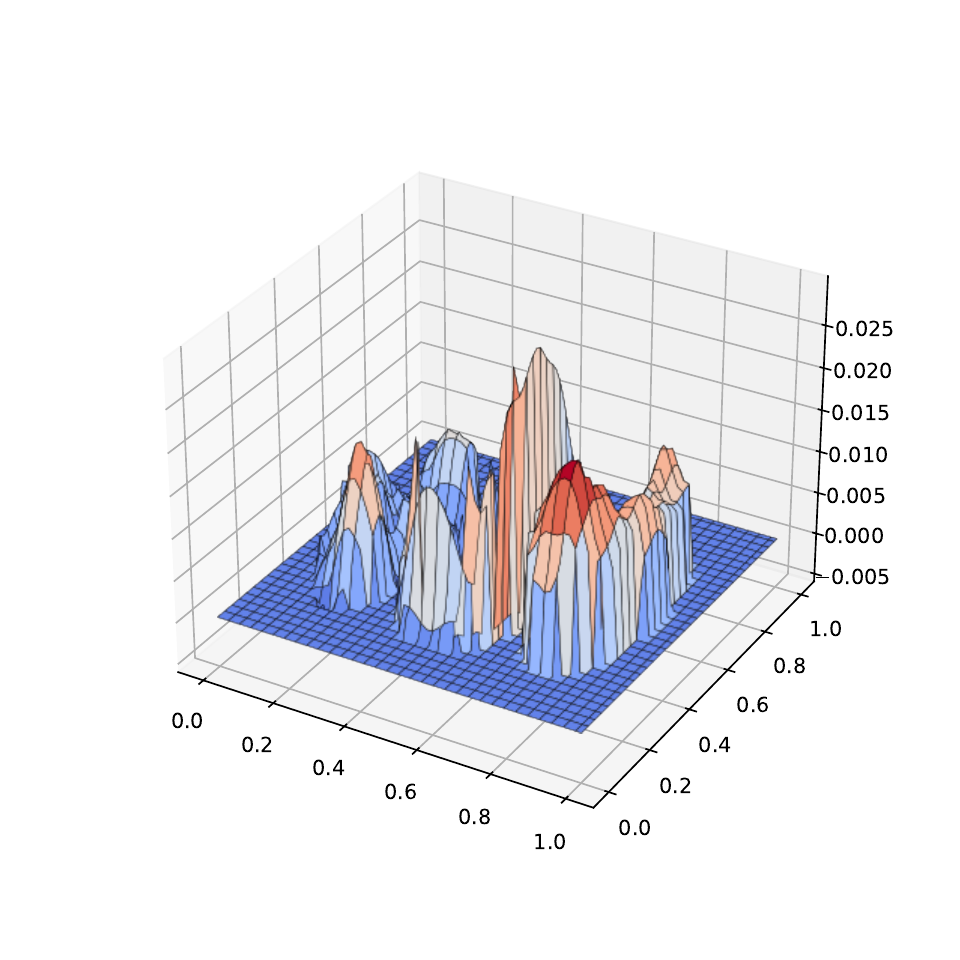}
    }\hfill
    \subfloat[Error of control at $t = 0.75$]{%
        \includegraphics[width=0.32\textwidth]{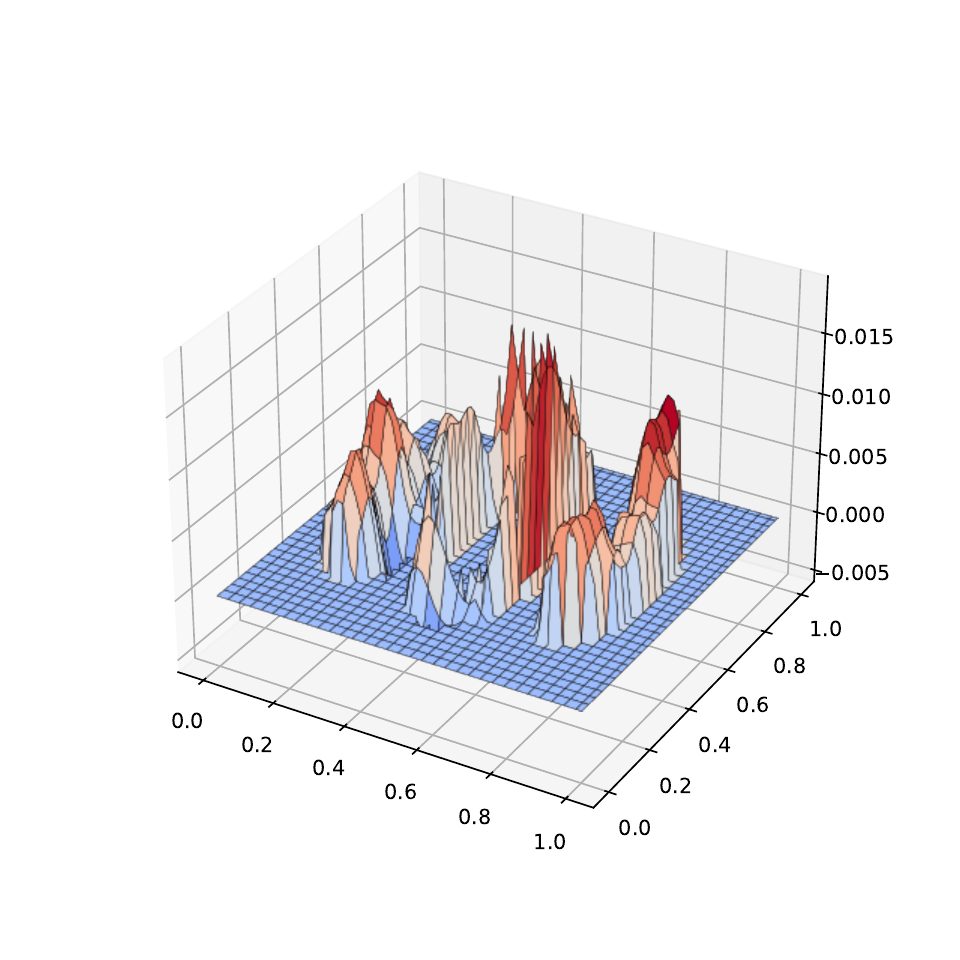}
    }\hfill
    \caption{Computed control and exact control for problem \eqref{eq:bilinparab-optctrl}--\eqref{eq:bilinparab-equation} with $\alpha=\beta=0.01$.}
    \label{fig:bilinparab-computed-ctrl-state}
\end{figure}

Finally, we fix $\alpha = \beta = 0.01$ and compare Algorithms \ref{alg:bilinparab-pdhg-opl}, \ref{alg:bilinparab-pdhg-opl}-R, the SSN-FEM, and the ADMM-FEM introduced in Section \ref{sec:admm} on different spatial grid resolutions $m$.
We employ $6$ SSN iterations and $10$ ADMM iterations \eqref{eq:admm-pinn}, respectively.
In each SSN iteration, the semismooth Newton equation is solved by the conjugate gradient (CG) method until a residual of $10^{-5}$ is reached.
For the ADMM-FEM, we set $\gamma = 0.01$ and apply the Barzilai-Borwein (BB) method \cite{barzilai1988two} with long stepsizes for the resulting subproblem \eqref{eq:admm-pinn-u} until a gradient norm of $10^{-5}$ is obtained. 
All the PDEs encountered in the SSN-FEM and ADMM-FEM are discretized by FEM with $m \times m$ grids, and the resulting linear systems are solved by the CG method.

The results are listed in \Cref{tab:bilinparab-computation-time}.
Both Algorithms \ref{alg:bilinparab-pdhg-opl} and \ref{alg:bilinparab-pdhg-opl}-R significantly outperform the SSN-FEM and ADMM-FEM in all tested resolutions.
Compared with the results in Section \ref{sec:burgers-numexp}, the time cost of the FEM-based method is further increased, since the current example is time-dependent and defined on a two-dimensional spatial domain, which increases the number and scale of the linear systems to be solved.
All the above results further validate that the proposed operator learning methods are promising for solving nonsmooth optimal control problems, especially when the underlying PDE is multi-dimensional and time-dependent.

\begin{table}[!ht]
    \centering
	\small
    \caption{Computation time (seconds) for the SSN-FEM, ADMM-FEM, \Cref{alg:bilinparab-pdhg-opl}, and \Cref{alg:bilinparab-pdhg-opl}-R ($\rho=1.618$), with respect to different spatial grid resolutions $m$.}
    \begin{tabular}{c c c c c }
		\toprule
		$m$ & SSN-FEM & ADMM-FEM & {\Cref{alg:bilinparab-pdhg-opl}} &{\Cref{alg:bilinparab-pdhg-opl}-R} \\
		\midrule
        $16$ & $5.0788$ & $6.4495$ & $0.5102$ & $0.4152$ \\
		$32$ & $13.5294$ & $15.9366$ & $0.6834$ & $0.5221$  \\
		$64$ & $38.6675$ & $46.7103$ & $0.8310$ & $0.6519$ \\
		$128$ & $235.3233$ & $360.1374$ & $2.5893$ & $1.8828$ \\
		$256$ & $2597.8904$ & $10138.0012$ & $9.4701$ & $6.8634$ \\
		\bottomrule
	\end{tabular}
    \label{tab:bilinparab-computation-time}
    \normalsize
\end{table}

\section{Optimal control of semilinear parabolic equations}\label{sec:sp-opt-ctrl}
In this section, we further apply \Cref{alg:pdhg-opl} to optimal control of semilinear parabolic equations.  Let $\Omega$ be a bounded domain of $\mathbb{R}^d$ $(d \geq 1)$ with $\partial \Omega$ as its
boundary, and let $\Omega_T:=\Omega\times(0,T)$ with $0<T<+\infty$.  We consider the following optimal control problem:
\begin{equation}\label{eq:sp-prob}
    \begin{aligned}
        \min_{u \in L^2(\Omega_T), y \in L^2(\Omega_T)} \quad &\frac{1}{2}\|y - y_d\|^2_{L^2(\Omega_T)} + \frac{\alpha}{2}\|u\|^2_{L^2(\Omega_T)} + \beta \|u\|_{L^1(\Omega_T)} + I_{U_{ad}}(u),
    \end{aligned}
\end{equation}
subject to the semilinear parabolic equation
\begin{equation}\label{eq:sp-equation}
         \partial_t y - \Delta y + R(y) = u \text{~in~} \Omega\times(0,T), \quad
         y = 0 \text{~on~} \partial \Omega \times (0, T),\quad
                 y(0) = 0  \text{~in~} \Omega.
\end{equation}
Here, the target $y_d \in L^2(\Omega_T)$ is given, $\alpha>0, \beta \geq 0$ are regularization parameters, $R$ is a Fr\'echet differentiable nonlinear operator given by $R(y) := (y - c_1)(y - c_2)(y - c_3),$
where $c_1, c_2, c_3 \in \mathbb{R}$.
The admissible set is defined as $U_{ad} := \{u \in L^\infty(\Omega_T): a \leq u(x,t) \leq b \text{~a.e.~in~} \Omega_T \}$.

\subsection{A primal-dual-based operator learning method for~\texorpdfstring{\eqref{eq:sp-prob}--\eqref{eq:sp-equation}}{(5.1)-(5.2)}}

We first note that the  primal-dual decoupling \eqref{eq:pdhg-nonlinear} can be applied to problem~\eqref{eq:sp-prob}--\eqref{eq:sp-equation} with
\begin{equation*}
	\begin{aligned}
		F(y)  = \frac{1}{2}\|y - y_d\|^2_{L^2(\Omega_T)}, ~
				G(u)  = \frac{\alpha}{2}\|u\|^2_{L^2(\Omega_T)} + \beta \|u\|_{L^1(\Omega_T)} + I_{U_{ad}}(u),
	\end{aligned}
\end{equation*}
Note that $F$ and $G$ here are the same as those defined in \Cref{sec:bp-opt-ctrl}, thus prox$_{\sigma F^*}$ and prox$_{\tau G}$ are precisely given by~\Cref{prop:bilinparab-prox}.
Moreover, by definition, $S(2u^{k+1}-u^k)$ is the solution of the state equation~\eqref{eq:sp-equation}. For the computation of $(S'(u^k))^* p^k$, we have the following result.
The proof of \Cref{prop:semilinear-adjoint} is similar to that of \Cref{prop:burgers-adjoint}, and we omit it for succinctness.
\begin{proposition}\label{prop:semilinear-adjoint}
    Consider problem \eqref{eq:sp-prob}--\eqref{eq:sp-equation}.
    For any $u, p \in L^2(\Omega_T)$, let $z$ be the solution of the following adjoint equation:
    \begin{equation}\label{eq:sp-equation-adj}
            -\partial_t z - \Delta z + R'(y) z = p  \text{~in~} \Omega_T \quad
            z = 0  \text{~on~} \partial \Omega \times (0, T), \quad
             z(T) = 0  \text{~in~} \Omega,
    \end{equation}
    where $y = S(u)$.
    Then, we have $z = (S'(u))^* p$.
\end{proposition}

Thus, the computation of $S(2u^{k+1}-u^k)$ and $(S'(u^k))^* p^k$ amounts to solving \eqref{eq:sp-equation} and its adjoint equation \eqref{eq:sp-equation-adj}, respectively.
We now construct two FNOs $\cS_{\theta_s}$ and $\cA_{\theta_a}$, parameterized by $\theta_s$ and $\theta_a$, and approximate the solutions
$y$ and $z$ of  \eqref{eq:sp-equation}~and~\eqref{eq:sp-equation-adj}  via
$h_s\cS_{\theta_s}(u)$ and $h_a \cA_{\theta_a}(y, p)$, respectively.
Here, the auxiliary functions $h_s, h_a: \overline{\Omega} \times [0,T] \rightarrow \mathbb{R}$ with $h_s \vert_{\partial \Omega}=h_a \vert_{\partial \Omega}=0$ are introduced to enforce the homogeneous initial and Dirichlet boundary conditions in~\eqref{eq:sp-equation} and~\eqref{eq:sp-equation-adj}.
We train $\cS_{\theta_s}$ and $\cA_{\theta_a}$ in a similar manner as described in Section \ref{sec:fno}.
Let $\cS_{\theta_s^*}$ and $\cA_{\theta_a^*}$ be the trained neural solution operators with  $\theta_s^*$ and $\theta_a^*$ being the optimal parameters. Then, a primal-dual-based operator learning method for \eqref{eq:sp-prob}--\eqref{eq:sp-equation} can be specified from \Cref{alg:pdhg-opl}, which is summarized in \Cref{alg:sp-pdhg-opl}.
\begin{algorithm}[!ht]
	\caption{A primal-dual-based operator learning method for \eqref{eq:sp-prob}--\eqref{eq:sp-equation}}
	\begin{algorithmic}[1]
	    \REQUIRE Pre-trained neural solution operators $\cS_{\theta_s^*}$ and $\cA_{\theta_a^*}$; initialization $u^0$, $p^0$; stepsizes $\tau, \sigma>0$; auxiliary functions $h_s$ and $h_a$.
        \FOR {$k = 0, 1, \ldots $}
            \STATE $y^k = h_s \cS_{\theta_s^*}(u^k)$.
            \STATE $u^{k+1} = \prox_{\tau G} (u^k - \tau h_a \cA_{\theta_a^*}(y^k, p^k))$, where $\prox_{\tau_G}$ is calculated by \eqref{eq:bilinparab-prox-G}.
            \STATE $p^{k+1} = \frac{1}{1 + \sigma} \left(p^k + \sigma h_s \cS_{\theta_s^*}(2u^{k+1} -u^k) - \sigma y_d \right)$.
        \ENDFOR
	    \ENSURE Numerical solution $(\hat{u}, \hat{y}):=(u^{k+1}, h_s\cS_{\theta_s^*}(u^{k+1}))$ to \eqref{eq:sp-prob}--\eqref{eq:sp-equation}.
	\end{algorithmic}
	\label{alg:sp-pdhg-opl}
\end{algorithm}

\subsection{Numerical experiments}
We test \Cref{alg:sp-pdhg-opl} on the optimal control problem  \eqref{eq:sp-prob}--\eqref{eq:sp-equation} and verify its efficiency. To this end, we consider a challenging example adopted from~\cite{langer2020unstructured} with a highly nonlinear operator $R$. In particular, setting $\Omega=(0,1)^2$ and $T=1$, we test problem~\eqref{eq:sp-prob}--\eqref{eq:sp-equation} with $R(y) = y (y - 0.25) (y + 1)$
and the desired state
\begin{equation*}
    \begin{aligned}
        y_d(x, t) = & \exp \left( -20\left((x_1 - 0.2)^2 + (x_2 - 0.2)^2 + (t - 0.2) ^2\right) \right) \\
        & + \exp \left( -20\left((x_1 - 0.7)^2 + (x_2 - 0.7)^2 + (t - 0.9)^2\right) \right).
    \end{aligned}
\end{equation*}
We take the parameters $\alpha = 10^{-4}$, $a = -10$, and $b = 20$.
For the regularization parameter $\beta$, we consider two cases: $\beta = 0.004$ and $\beta = 0$, which correspond to sparse and non-sparse optimal controls, respectively.

For the training set of $\cS_{\theta_s}$, we first discretize $\overline{\Omega_T}$ by an $m \times m \times m_T$ equi-spaced grid $\overline{\mathcal{D}} = \{ ((x_1)_i, (x_2)_j, t_k) \in \overline{\Omega_T} : 1 \leq i, j \leq m, 1 \leq k \leq m_T \}$ with $m=m_T=64$, and sample $N_s = 2048$ functions $\{ v_i \}_{i=1}^{N_s}$ from the Gaussian random field
$v_i \sim \mathcal{GR} (0, C((x, t)^\top, (\tilde{x}, \tilde{t})^\top))$,
where the Gaussian covariance kernel is specified as $C((x, t)^\top, (\tilde{x}, \tilde{t})^\top) = 150 \exp( - 4 \| (x, t)^\top - (\tilde{x}, \tilde{t})^\top \|_2^2 )$.
For each $(x, t)^\top \in \Omega_T$, the generated $v_i$ is further projected to a sparse function $u_i$ via
\begin{equation*}
    u_i(x, t) = \begin{cases}
        x_1 x_2 (x_1-1) (x_2-1) \cdot \max \{ 25 v_i(x, t) - 125, ~ 20 \} \quad & \text{~if~} v_i(x, t) > 5, \\
        x_1 x_2 (x_1-1) (x_2-1) \cdot \min \{ 25 v_i(x, t) + 125, ~ -10  \} \quad & \text{~if~} v_i(x, t) < 5, \\
        0 \quad & \text{~otherwise}.
    \end{cases}
\end{equation*}
We solve the state equation \eqref{eq:sp-equation} for each $u_i$ to get a solution $y_i$ by the \texttt{MATLAB PDE Toolbox}, and use $\{(u_i, y_i)\}_{i=1}^N$ as the training set for $\cS_{\theta_s}$.
To generate a training set for $\cA_{\theta_a}$, we employ the same grids $\overline{\mathcal{D}}$ and sample $N_a = 2048$ functions $\{ \bar{y}_i := \psi_i f_i \}_{i=1}^{N_a}$, where each $f_i$ is sampled from the Gaussian random field $\mathcal{GR}(0, C((x, t)^\top, (\tilde{x}, \tilde{t})^\top))$ with $C((x, t)^\top, (\tilde{x}, \tilde{t})^\top) = 5 \exp(-4 \|(x, t)^\top - (\tilde{x}, \tilde{t})^\top\|_2^2)$, and $\psi_i(x, t) := t^{2/5} x(x-1) y(y-1)$ imposes homogeneous boundary and initial conditions on $\bar{y}_i$.
Next, we sample $N_a$ functions $\{g_i\}_{i=1}^{N_a}$ from the same Gaussian random field $\mathcal{GR}(0, C(x_1, x_2))$ and set $p_i = y_d - \psi_i g_i$.
We then solve for each $(\bar{y}_i, p_i)$ the solution $z_i$ of~\eqref{eq:sp-equation-adj}, and take $\{(\bar{y}_i, p_i, z_i)\}_{i=1}^{N_a}$ as the training set for $\cA_{\theta_a}$.

For the two FNOs $\cS_{\theta_s}$ and $\cA_{\theta_a}$, we set their lifting dimensions to $m_p = 32$ and $m_p = 16$, respectively.
Each of $\cS_{\theta_s}$ and $\cA_{\theta_a}$ consists of three hidden Fourier layers and \texttt{GeLU} activation function.
The number of truncated Fourier modes is $k = 16$ for $\cS_{\theta_s}$ and $k = 8$ for $\cA_{\theta_a}$.
All the neural network parameters $\theta_s$ and $\theta_a$ are initialized by the default \texttt{PyTorch} settings.
	
We take $h_s(x, t) := t^{2/5} xy(1-x) (1-y)$ and $h_a(x, t) := h_s(x, T-t)$. Then, the two FNOs $\cS_{\theta_s}$ and $\cA_{\theta_a}$ are trained by respectively minimizing the following loss functions:
\begin{equation}\label{eq:fno-loss1}
	\cL_s(\theta_s) = \frac{1}{N_s}\sum_{i=1}^{N_s} \frac{\sum_{(x, t)^\top \in \overline{\mathcal{D}}}|h_s(x, t)  \cS_{\theta_s}(u_i)(x, t) - y_i(x, t)|^2}{\sum_{(x, t)^\top \in \overline{\mathcal{D}}} \left|y_i(x, t)\right|^2},
\end{equation}
\begin{equation}\label{eq:fno-loss2}
	\cL_a(\theta_a) = \frac{1}{N_a}\sum_{i=1}^{N_a} \frac{ \sum_{(x, t)^\top \in \overline{\mathcal{D}}}\left|h_a(x, t)  \cA_{\theta_a}(y_i, p_i)(x, t) - z_i(x, t)\right|^2}{\sum_{(x, t)^\top \in \overline{\mathcal{D}}}\left|z_i(x, t)\right|^2}.
\end{equation}
For training $\cS_{\theta_s}$, we minimize the loss function \eqref{eq:fno-loss1} by the Adam optimizer with batch size $8$ for 500 epochs, where
the learning rate is initialized as $0.001$ and then multiplied by $0.6$ for every 50 epochs.
For $\cA_{\theta_a}$, the loss function \eqref{eq:fno-loss1} is minimized by the Adam optimizer with a batch size of $8$ for 300 epochs.
We initialize the learning rate as $0.001$ and decay it by a factor of $0.6$ every $30$ epochs.

We solve the optimal control problem~\eqref{eq:sp-prob}--\eqref{eq:sp-equation} for two parameters $\beta = 0.004$ and $\beta = 0$ by \Cref{alg:sp-pdhg-opl} with the pre-trained $\cS_{\theta_s^*}$ and $\cA_{\theta_a^*}$.
We initialize \Cref{alg:sp-pdhg-opl} with $u^0 = p^0 = 0$, and set the stepsizes to be $\tau = 500$ and $\sigma = 0.4$.
The algorithm is terminated when~\eqref{eq:term-cond1} is satisfied.
The numerical solutions for $\beta = 0.004$ and $\beta = 0$ are shown in~\Cref{fig:sp-computed-ctrl-state-beta-0.004} and \Cref{fig:sp-computed-ctrl-state-beta-0}, respectively, where the computed control and state are depicted at $t = 0.25$, $t = 0.5$, and $t = 0.75$.
The results match the figures obtained by FEM-based methods in~\cite{langer2020unstructured}.

\begin{figure}[!ht]
    \centering
    \subfloat[Computed $u$ at $t = 0.25$]{%
        \includegraphics[width=0.33\textwidth]{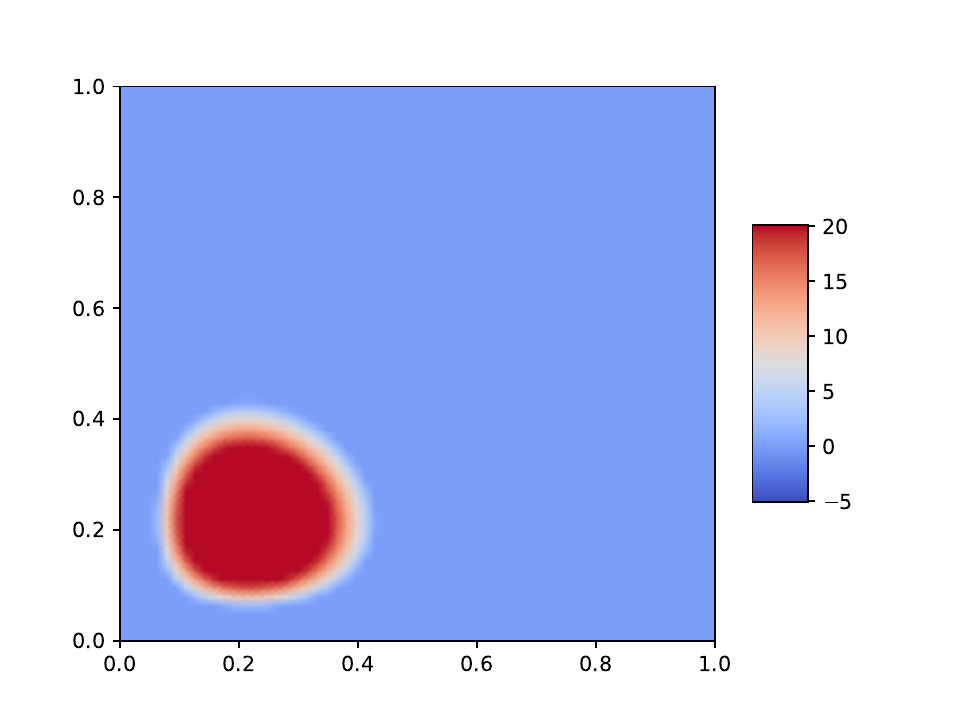}%
    }\hfill
    \subfloat[Computed $u$ at $t = 0.5$]{%
        \includegraphics[width=0.33\textwidth]{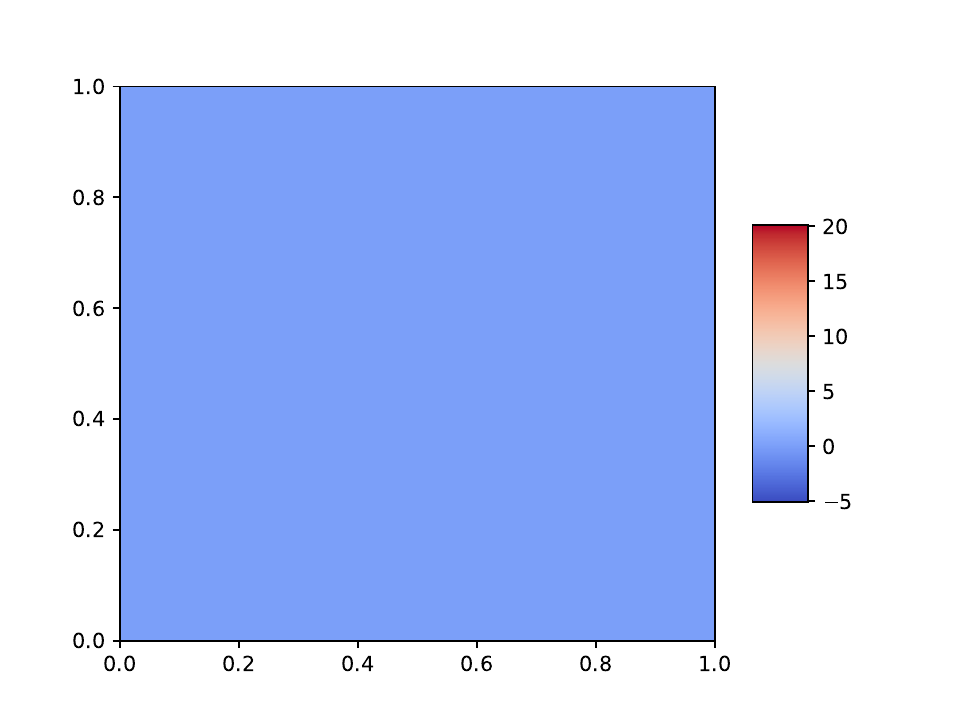}%
    }\hfill
    \subfloat[Computed $u$ at $t = 0.75$]{%
        \includegraphics[width=0.33\textwidth]{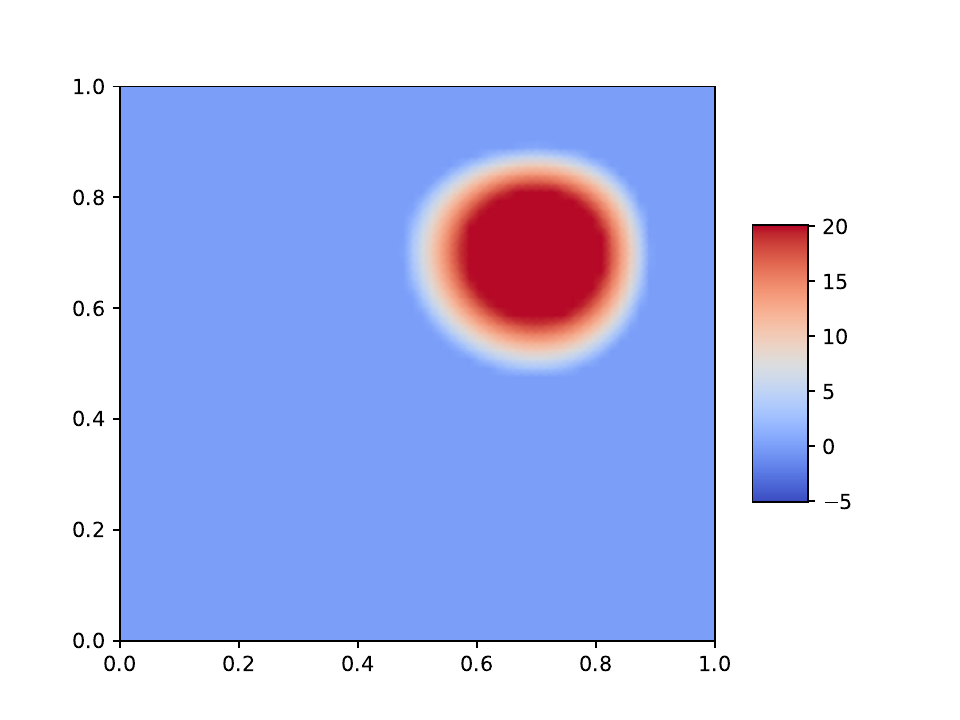}%
    }\\
    \subfloat[Computed $y$ at $t = 0.25$]{%
        \includegraphics[width=0.33\textwidth]{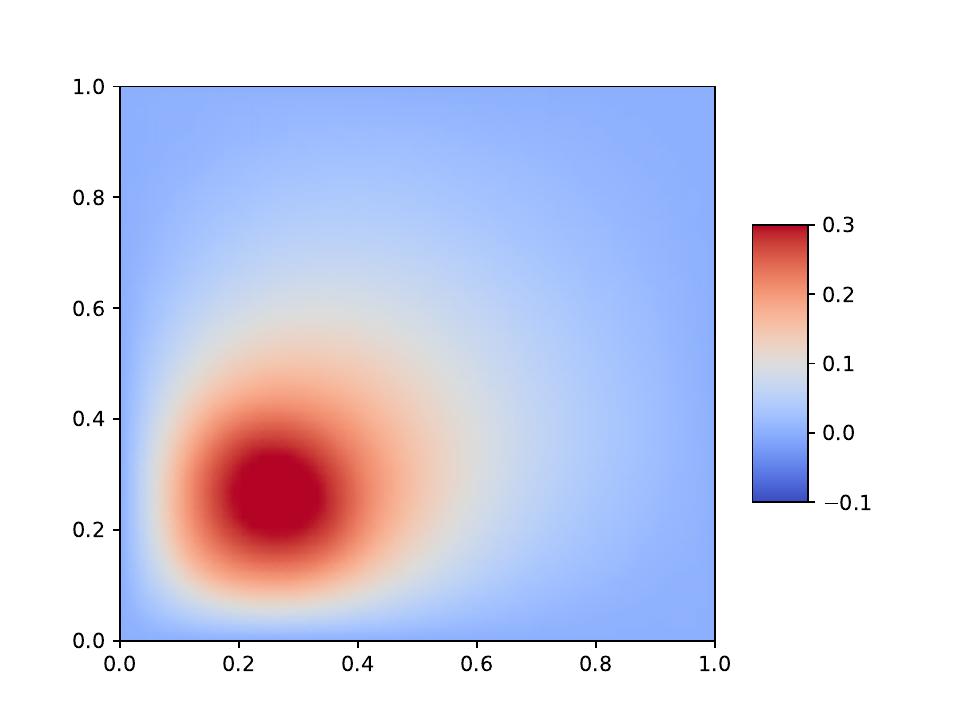}%
    }\hfill
    \subfloat[Computed $y$ at $t = 0.5$]{%
        \includegraphics[width=0.33\textwidth]{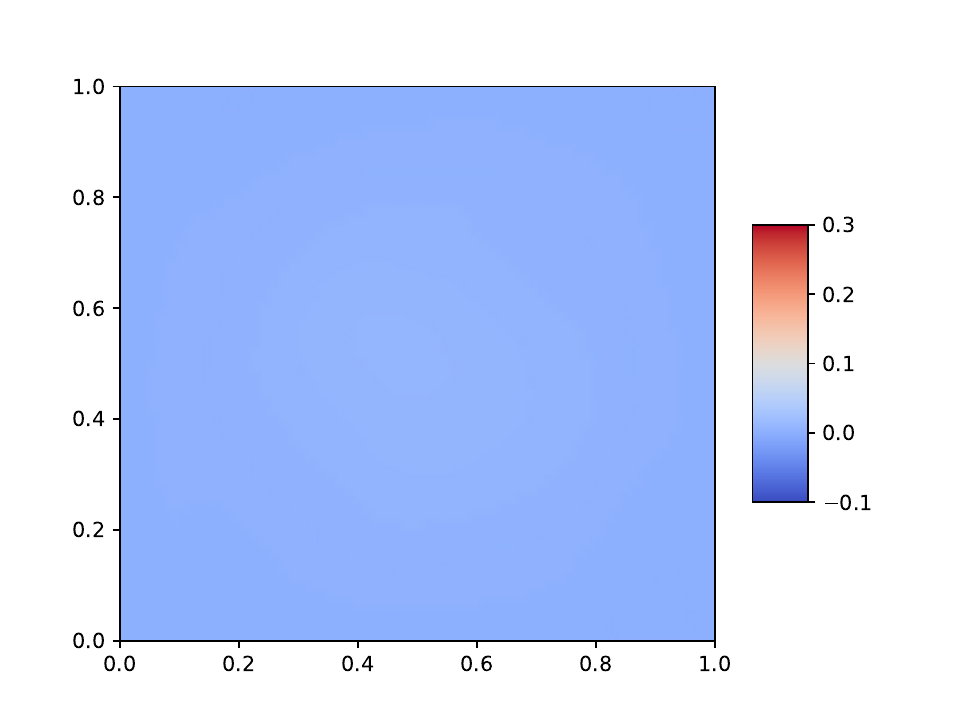}%
    }\hfill
    \subfloat[Computed $y$ at $t = 0.75$]{%
        \includegraphics[width=0.33\textwidth]{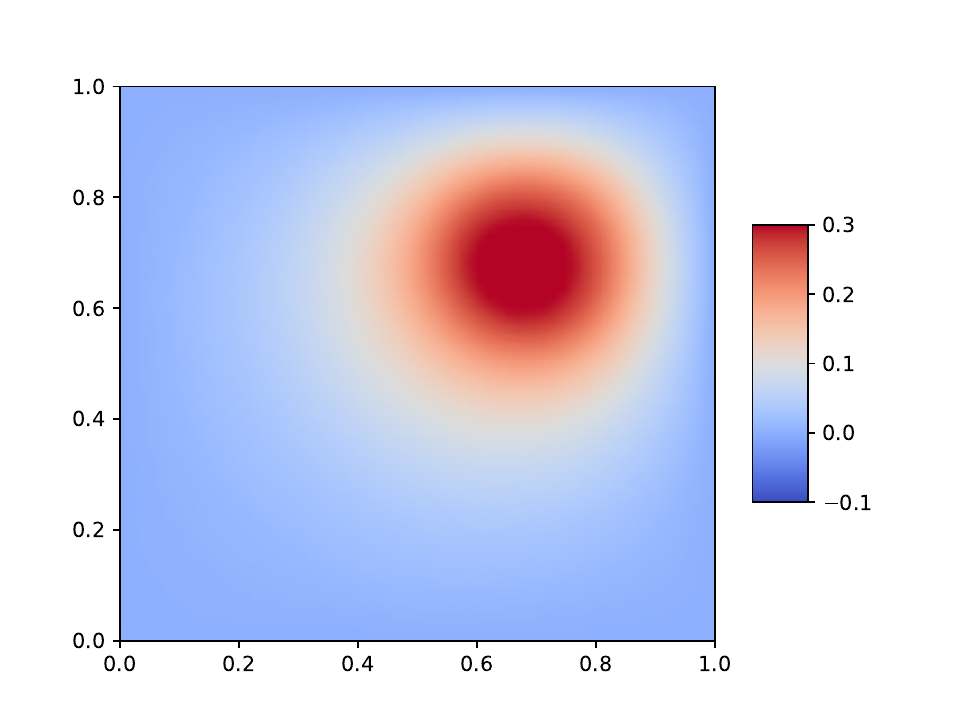}%
    }
    \caption{Computed optimal control and state for~\eqref{eq:sp-prob}--\eqref{eq:sp-equation} with $\beta = 0.004$.}
    \label{fig:sp-computed-ctrl-state-beta-0.004}
\end{figure}

\begin{figure}[!ht]
    \centering
    \subfloat[Computed $u$ at $t = 0.25$]{%
        \includegraphics[width=0.32\textwidth]{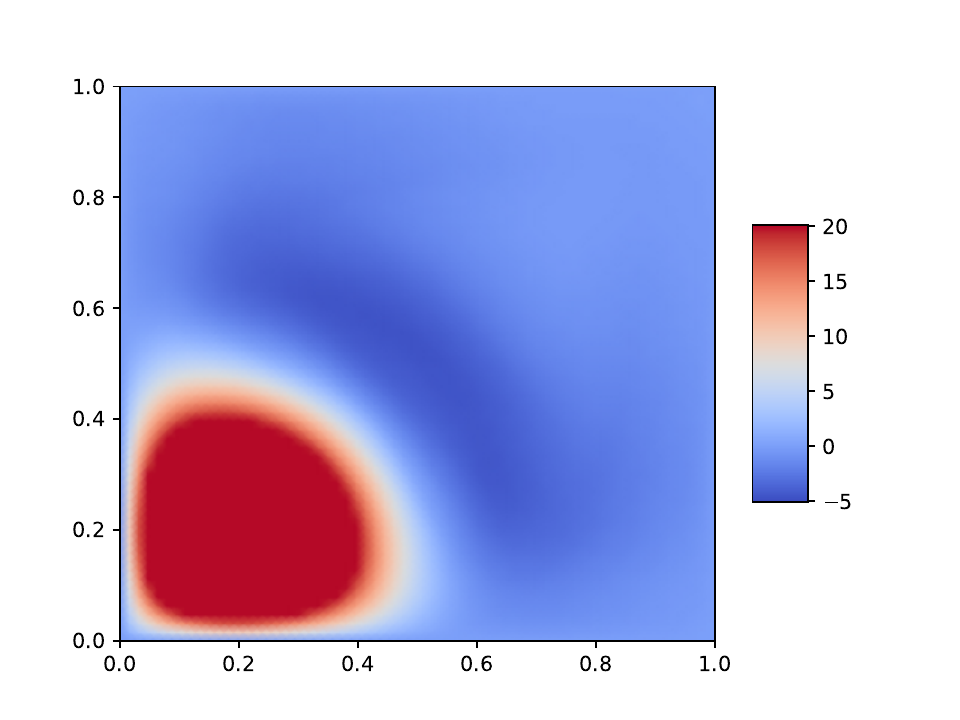}%
    }\hfill
    \subfloat[Computed $u$ at $t = 0.5$]{%
        \includegraphics[width=0.32\textwidth]{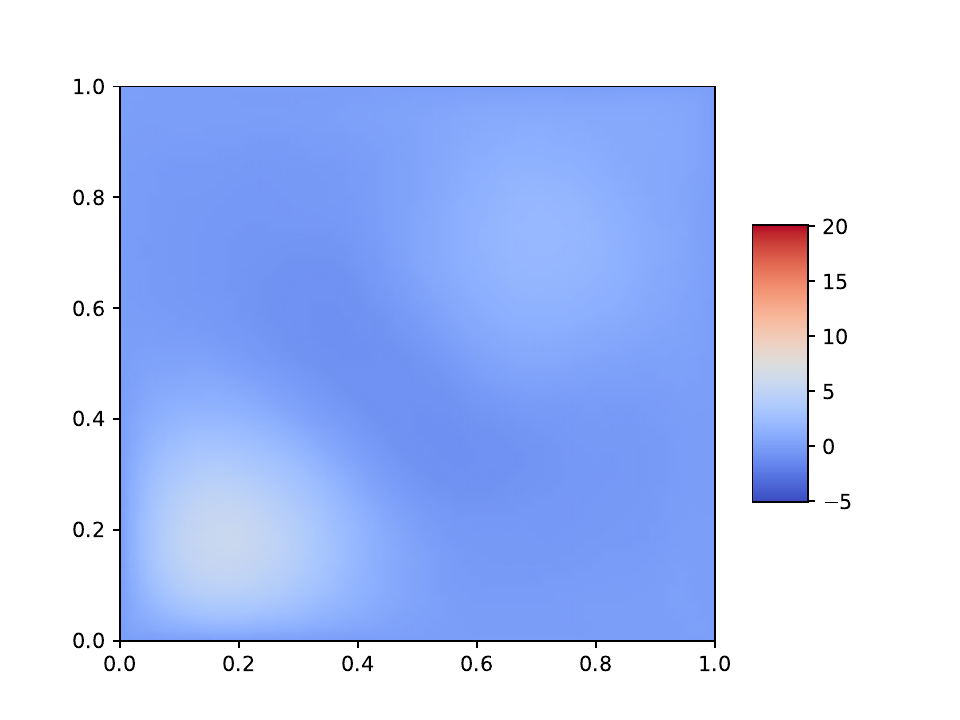}%
    }\hfill
    \subfloat[Computed $u$ at $t = 0.75$]{%
        \includegraphics[width=0.32\textwidth]{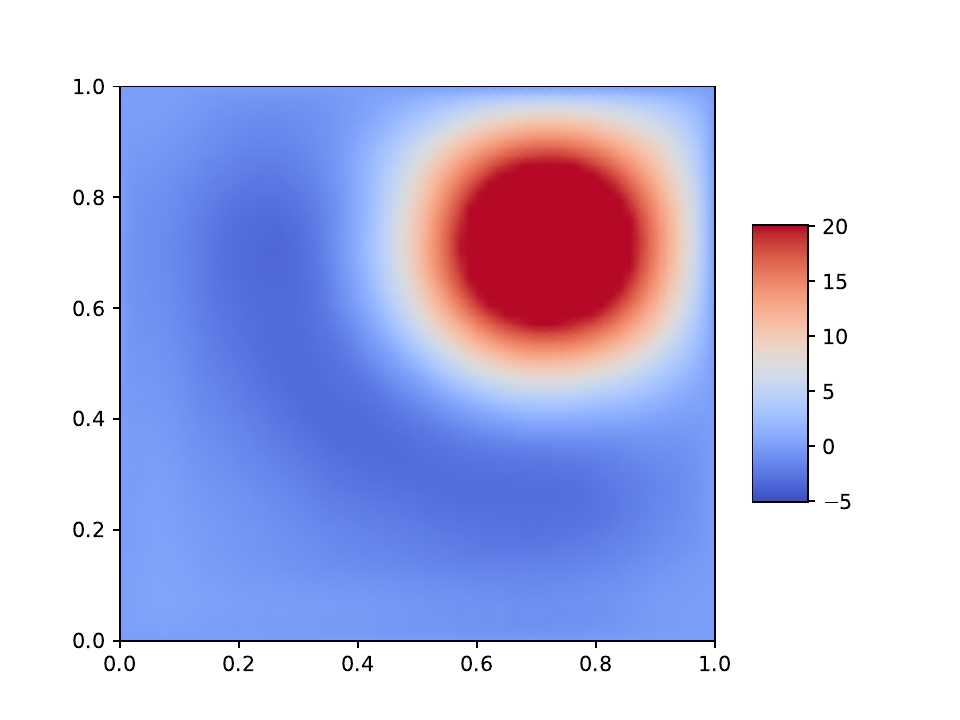}%
    }\\
    \subfloat[Computed $y$ at $t = 0.25$]{%
        \includegraphics[width=0.32\textwidth]{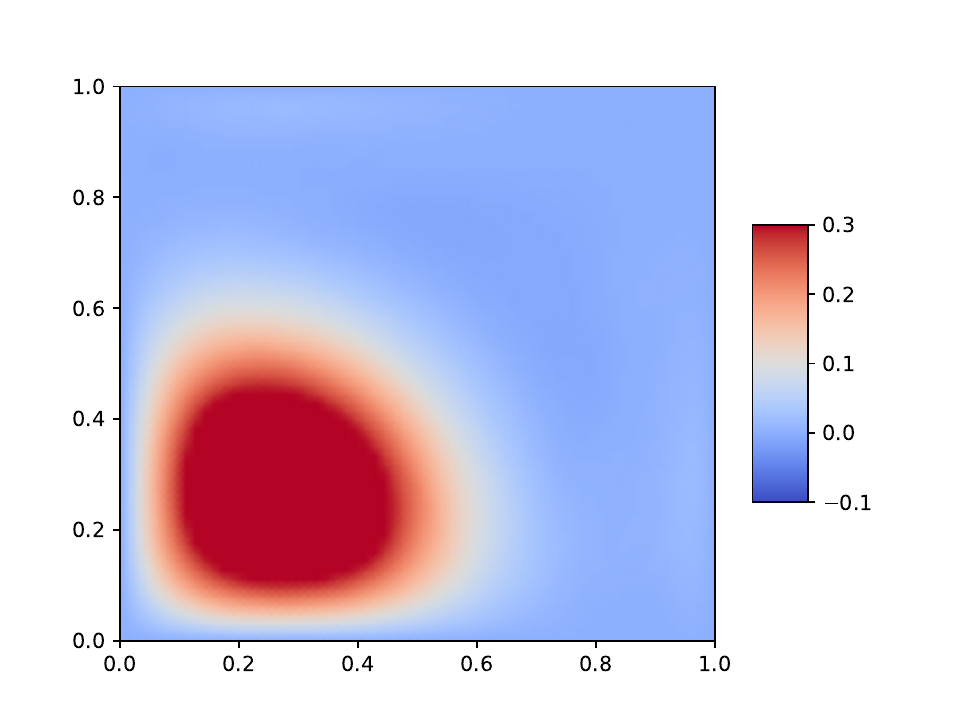}%
    }\hfill
    \subfloat[Computed $y$ at $t = 0.5$]{%
        \includegraphics[width=0.32\textwidth]{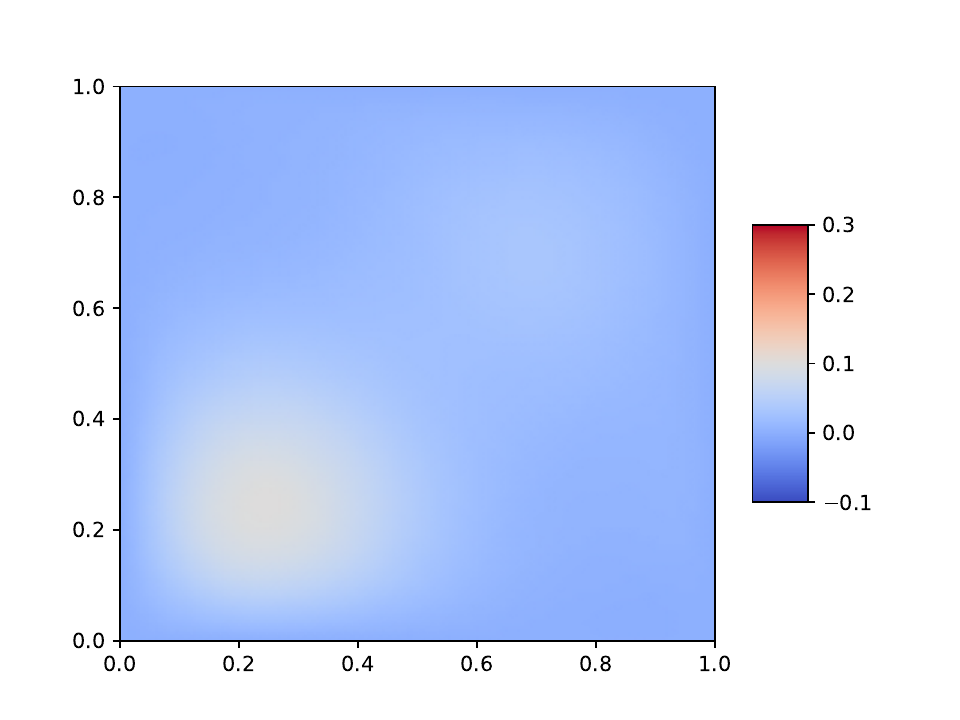}%
    }\hfill
    \subfloat[Computed $y$ at $t = 0.75$]{%
        \includegraphics[width=0.32\textwidth]{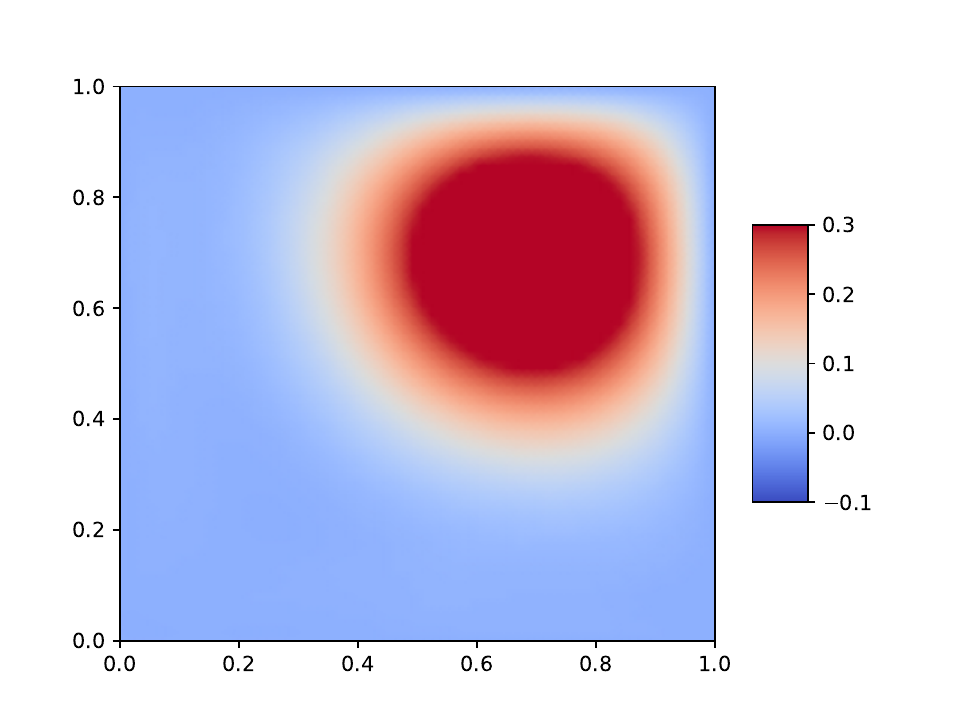}%
    }
    \caption{Computed optimal control and state for~\eqref{eq:sp-prob}--\eqref{eq:sp-equation} with $\beta = 0$.}
    \label{fig:sp-computed-ctrl-state-beta-0}
\end{figure}

We validate the efficiency of \Cref{alg:sp-pdhg-opl} by comparing it with the SSN-FEM and PD-FEM introduced in Section \ref{sec:pd-decoup}. We run 10 iterations of SSN-FEN and 20 iterations of PD-FEN in total.
For SSN-FEN, the Newton's equation is solved by CG with a tolerance of $10^{-5}$.
We set the stepsizes of the PD-FEM as those applied in \Cref{alg:sp-pdhg-opl}. The results is summarized in \Cref{tab:sp-computation-time}. It is observed that \Cref{alg:sp-pdhg-opl} significantly outperforms the SSN-FEM and PD-FEM in terms of computational time, thus validating the numerical efficiency of \Cref{alg:sp-pdhg-opl} over traditional algorithms for solving the sparse optimal control problem of semilinear parabolic equations.

\begin{table}[!ht]
    \footnotesize
    \centering
    \caption{Computation time (seconds) for the SSN-FEM, PD-FEM, and \Cref{alg:sp-pdhg-opl}, with respect to different spatial grid resolutions $m$.}
    \begin{tabular}{c c c c}
		\toprule
		$m$ & SSN-FEM & PD-FEM & {\Cref{alg:sp-pdhg-opl}} \\
		\midrule
        $16$ & $7.0862$ & $9.0693$ & $1.7186$ \\
		$32$ & $22.2930$ & $21.1309$ & $2.0926$ \\
		$64$ & $75.2745$ & $69.1137$ & $3.7704$ \\
		$128$ & $537.6598$ & $735.1338$ & $14.8336$ \\
		$256$ & $3410.1498$ & $9343.6343$ & $56.3242$ \\
		\bottomrule
	\end{tabular}
	\label{tab:sp-computation-time}
    \normalsize
\end{table}

\section{Conclusions and perspectives}\label{sec:conclusion}
We propose a primal-dual-based operator learning framework for solving a general class of nonsmooth optimal control problems with nonlinear partial differential equation (PDE) constraints.
The resulting primal-dual-based operator learning methods inherit the efficiency and generalization of operator learning while leveraging the model-based property of the primal-dual methods.
Specifically, the pre-trained deep neural networks, serving as surrogate models for the involved PDE subproblems, can be reused across primal-dual iterations and different problem parameter settings. Consequently, the proposed methods differ from existing algorithms in that they do not require repeatedly solving discretized algebraic systems or retraining neural networks.
Extensive numerical experiments on various nonsmooth optimal control problems with nonlinear PDE constraints validate the efficiency and accuracy of the proposed primal-dual-based operator learning methods.

Our work leaves some important questions for future research. The numerical efficiency of the proposed methods underscores the theoretical understanding of their convergence properties in general cases.
It is also interesting to investigate how to extend our discussions to the case where the PDEs are defined in high-dimensional spaces such as the nonlinear Black-Scholes equations \cite{e2021algorithms} and the Boltzmann kinetic equations \cite{naldi2010mathematical}.

\bibliographystyle{siamplain}
\bibliography{references}

\begin{thebibliography}{10}

\bibitem{antil2024mathematical}
{\sc H.~Antil}, {\em Mathematical opportunities in digital twins ({MATH-DT})}, preprint, arXiv:2402.10326, 2024.

\bibitem{barry2022physics}
{\sc J.~Barry-Straume, A.~Sarshar, A.~A. Popov, and A.~Sandu}, {\em Physics-informed neural networks for {PDE}-constrained optimization and control}, preprint, arXiv:2205.03377, 2022.

\bibitem{barzilai1988two}
{\sc J.~Barzilai and J.~M. Borwein}, {\em Two-point step size gradient methods}, IMA Journal of Numerical Analysis, 8 (1988), pp.~141--148.

\bibitem{behrens2014real}
{\sc M.~Behrens, H.~G. Bock, S.~Engell, P.~Khobkhun, and A.~Potschka}, {\em Real-time {PDE} constrained optimal control of a periodic multicomponent separation process}, in Trends in PDE Constrained Optimization, Springer, 2014, pp.~521--537.

\bibitem{berggren1996computational}
{\sc M.~Berggren, R.~Glowinski, and J.~L. Lions}, {\em A computational approach to controllability issues for flow-related models. ({II}): {Control} of two-dimensional, linear advection-diffusion and {Stokes} models}, International Journal of Computational Fluid Dynamics, 6 (1996), pp.~253--274.

\bibitem{bhattacharya2021model}
{\sc K.~Bhattacharya, B.~Hosseini, N.~B. Kovachki, and A.~M. Stuart}, {\em Model reduction and neural networks for parametric {PDEs}}, The SMAI Journal of Computational Mathematics, 7 (2021), pp.~121--157.

\bibitem{biccari2023two}
{\sc U.~Biccari, Y.~Song, X.~Yuan, and E.~Zuazua}, {\em A two-stage numerical approach for the sparse initial source identification of a diffusion--advection equation}, Inverse Problems, 39 (2023), p.~095003.

\bibitem{biegler2007real}
{\sc L.~T. Biegler, O.~Ghattas, M.~Heinkenschloss, D.~Keyes, and B.~van Bloemen~Waanders}, {\em Real-time PDE-constrained Optimization}, SIAM, 2007.

\bibitem{cao2023lno}
{\sc Q.~Cao, S.~Goswami, and G.~E. Karniadakis}, {\em {Laplace} neural operator for solving differential equations}, Nature Machine Intelligence, 6 (2024), pp.~631--640.

\bibitem{casas2017review}
{\sc E.~Casas}, {\em A review on sparse solutions in optimal control of partial differential equations}, SeMA Journal, 74 (2017), pp.~319--344.

\bibitem{casas2024convergence}
{\sc E.~Casas and M.~Mateos}, {\em Convergence analysis of the semismooth {Newton} method for sparse control problems governed by semilinear elliptic equations}, SIAM Journal on Control and Optimization, 62 (2024), pp.~3076--3090.

\bibitem{chambolle2011first}
{\sc A.~Chambolle and T.~Pock}, {\em A first-order primal-dual algorithm for convex problems with applications to imaging}, Journal of Mathematical Imaging and Vision, 40 (2011), pp.~120--145.

\bibitem{chen1995universal}
{\sc T.~Chen and H.~Chen}, {\em Universal approximation to nonlinear operators by neural networks with arbitrary activation functions and its application to dynamical systems}, IEEE Transactions on Neural Networks, 6 (1995), pp.~911--917.

\bibitem{clason2019acceleration}
{\sc C.~Clason, S.~Mazurenko, and T.~Valkonen}, {\em Acceleration and global convergence of a first-order primal-dual method for nonconvex problems}, SIAM Journal on Optimization, 29 (2019), pp.~933--963.

\bibitem{clason2017primal}
{\sc C.~Clason and T.~Valkonen}, {\em Primal-dual extragradient methods for nonlinear nonsmooth {PDE}-constrained optimization}, SIAM Journal on Optimization, 27 (2017), pp.~1314--1339.

\bibitem{colli2015optimal}
{\sc P.~Colli and J.~Sprekels}, {\em Optimal control of an {Allen}--{Cahn} equation with singular potentials and dynamic boundary condition}, SIAM Journal on Control and Optimization, 53 (2015), pp.~213--234.

\bibitem{condat2023proximal}
{\sc L.~Condat, D.~Kitahara, A.~Contreras, and A.~Hirabayashi}, {\em Proximal splitting algorithms for convex optimization: A tour of recent advances, with new twists}, SIAM Review, 65 (2023), pp.~375--435.

\bibitem{cybenko1989approximation}
{\sc G.~Cybenko}, {\em Approximation by superpositions of a sigmoidal function}, Mathematics of Control, Signals and Systems, 2 (1989), pp.~303--314.

\bibitem{dai2023solving}
{\sc Y.~Dai, B.~Jin, R.~Sau, and Z.~Zhou}, {\em Solving elliptic optimal control problems via neural networks and optimality system}, preprint, arXiv: 2308.11925, 2023.

\bibitem{de2015numerical}
{\sc J.~C. De~los Reyes}, {\em Numerical PDE-Constrained Optimization}, Springer, 2015.

\bibitem{de2004comparison}
{\sc J.~C. de~los Reyes and K.~Kunisch}, {\em A comparison of algorithms for control constrained optimal control of the {Burgers} equation}, Calcolo, 41 (2004), pp.~203--225.

\bibitem{de2005semismooth}
{\sc J.~C. De~Los~Reyes and K.~Kunisch}, {\em A semi-smooth {N}ewton method for control constrained boundary optimal control of the {Navier}--{Stokes} equations}, Nonlinear Analysis: Theory, Methods \& Applications, 62 (2005), pp.~1289--1316.

\bibitem{e2021algorithms}
{\sc W.~E, J.~Han, and A.~Jentzen}, {\em Algorithms for solving high dimensional {PDEs}: from nonlinear {Monte} {Carlo} to machine learning}, Nonlinearity, 35 (2021), p.~278.

\bibitem{e2018deep}
{\sc W.~E and B.~Yu}, {\em The deep {Ritz} method: a deep learning-based numerical algorithm for solving variational problems}, Communications in Mathematics and Statistics, 6 (2018), pp.~1--12.

\bibitem{elvetun2016split}
{\sc O.~L. Elvetun and B.~F. Nielsen}, {\em The split {Bregman} algorithm applied to {PDE}-constrained optimization problems with total variation regularization}, Computational Optimization and Applications, 64 (2016), pp.~699--724.

\bibitem{glowinski1994exact}
{\sc R.~Glowinski and J.-L. Lions}, {\em Exact and approximate controllability for distributed parameter systems}, Acta Numerica, 3 (1994), pp.~269--378.

\bibitem{glowinski1995exact}
{\sc R.~Glowinski and J.-L. Lions}, {\em Exact and approximate controllability for distributed parameter systems}, Acta Numerica, 4 (1995), pp.~159--328.

\bibitem{glowinski1975sur}
{\sc R.~Glowinski and A.~Marroco}, {\em Sur l'approximation, par {\'e}l{\'e}ments finis d'ordre un, et la r{\'e}solution, par p{\'e}nalisation-dualit{\'e} d'une classe de probl{\`e}mes de dirichlet non lin{\'e}aires}, Revue Fran{\c{c}}aise d'Automatique, Informatique, Recherche Op{\'e}rationnelle. Analyse Num{\'e}rique, 9 (1975), pp.~41--76.

\bibitem{glowinski2022application}
{\sc R.~Glowinski, Y.~Song, X.~Yuan, and H.~Yue}, {\em Application of the alternating direction method of multipliers to control constrained parabolic optimal control problems and beyond}, Annals of Applied Mathematics, 38 (2022), pp.~115--158.

\bibitem{gol1979modified}
{\sc E.~G. Gol’shtein and N.~Tret’yakov}, {\em Modified {Lagrangians} in convex programming and their generalizations}, Point-to-Set Maps and Mathematical Programming,  (1979), pp.~86--97.

\bibitem{griesse2010local}
{\sc R.~Griesse, N.~Metla, and A.~R{\"o}sch}, {\em Local quadratic convergence of {SQP} for elliptic optimal control problems with mixed control-state constraints}, Control and Cybernetics, 39 (2010), pp.~717--738.

\bibitem{grote2014inexact}
{\sc M.~J. Grote, J.~Huber, D.~Kourounis, and O.~Schenk}, {\em Inexact interior-point method for {PDE}-constrained nonlinear optimization}, SIAM Journal on Scientific Computing, 36 (2014), pp.~A1251--A1276.

\bibitem{he2017algorithmic}
{\sc B.~He, F.~Ma, and X.~Yuan}, {\em An algorithmic framework of generalized primal--dual hybrid gradient methods for saddle point problems}, Journal of Mathematical Imaging and Vision, 58 (2017), pp.~279--293.

\bibitem{he2012convergence}
{\sc B.~He and X.~Yuan}, {\em Convergence analysis of primal-dual algorithms for a saddle-point problem: from contraction perspective}, SIAM Journal on Imaging Sciences, 5 (2012), pp.~119--149.

\bibitem{heinkenschloss1999analysis}
{\sc M.~Heinkenschloss and F.~Tr{\"o}ltzsch}, {\em Analysis of the {Lagrange}--{SQP}--{Newton} method for the control of a phase field equation}, Control and Cybernetics, 28 (1999), pp.~177--211.

\bibitem{hintermuller2006sqp}
{\sc M.~Hinterm{\"u}ller and M.~Hinze}, {\em A {SQP}-semismooth {Newton}-type algorithm applied to control of the instationary {Navier}--{Stokes} system subject to control constraints}, SIAM Journal on Optimization, 16 (2006), pp.~1177--1200.

\bibitem{hinze2008optimization}
{\sc M.~Hinze, R.~Pinnau, M.~Ulbrich, and S.~Ulbrich}, {\em Optimization with PDE Constraints}, vol.~23, Springer Science \& Business Media, 2008.

\bibitem{hoppe2021convergence}
{\sc F.~Hoppe and I.~Neitzel}, {\em Convergence of the {SQP} method for quasilinear parabolic optimal control problems}, Optimization and Engineering, 22 (2021), pp.~2039--2085.

\bibitem{hornik1989multilayer}
{\sc K.~Hornik, M.~Stinchcombe, and H.~White}, {\em Multilayer feedforward networks are universal approximators}, Neural Networks, 2 (1989), pp.~359--366.

\bibitem{hwang2022solving}
{\sc R.~Hwang, J.~Y. Lee, J.~Y. Shin, and H.~J. Hwang}, {\em Solving {PDE}-constrained control problems using operator learning}, Proceedings of the AAAI Conference on Artificial Intelligence, 36 (2022), pp.~4504--4512.

\bibitem{jin2022mionet}
{\sc P.~Jin, S.~Meng, and L.~Lu}, {\em {MIONet}: Learning multiple-input operators via tensor product}, SIAM Journal on Scientific Computing, 44 (2022), pp.~A3490--A3514.

\bibitem{kawaguchi2017generalization}
{\sc K.~Kawaguchi, Y.~Bengio, and L.~P. Kaelbling}, {\em Generalization in deep learning}, in Mathematical Aspects of Deep Learning, Cambridge University Press, 2022, pp.~112--148.

\bibitem{kidger2020universal}
{\sc P.~Kidger and T.~Lyons}, {\em Universal approximation with deep narrow networks}, in Proceedings of Thirty Third Conference on Learning Theory, PMLR, 2020, pp.~2306--2327.

\bibitem{kingma2014adam}
{\sc D.~P. Kingma and J.~Ba}, {\em {Adam}: A method for stochastic optimization}, preprint, arXiv:1412.6980, 2014.

\bibitem{kovachki2023neural}
{\sc N.~Kovachki, Z.~Li, B.~Liu, K.~Azizzadenesheli, K.~Bhattacharya, A.~Stuart, and A.~Anandkumar}, {\em Neural operator: Learning maps between function spaces with applications to {PDEs}}, Journal of Machine Learning Research, 24 (2023), pp.~1--97.

\bibitem{kroner2011semismooth}
{\sc A.~Kr{\"o}ner, K.~Kunisch, and B.~Vexler}, {\em Semismooth {Newton} methods for optimal control of the wave equation with control constraints}, SIAM Journal on Control and Optimization, 49 (2011), pp.~830--858.

\bibitem{lai2023hard}
{\sc M.-C. Lai, Y.~Song, X.~Yuan, H.~Yue, and T.~Zeng}, {\em The hard-constraint {PINNs} for interface optimal control problems}, SIAM Journal on Scientific Computing,  (to appear).

\bibitem{langer2020unstructured}
{\sc U.~Langer, O.~Steinbach, F.~Tr{\"o}ltzsch, and H.~Yang}, {\em Unstructured space-time finite element methods for optimal sparse control of parabolic equations}, in Optimization and Control for Partial Differential Equations: Uncertainty Quantification, Open and Closed-Loop Control, and Shape Optimization, De Gruyter, 2022, pp.~167--188.

\bibitem{li2020neural}
{\sc Z.~Li, N.~Kovachki, K.~Azizzadenesheli, B.~Liu, K.~Bhattacharya, A.~Stuart, and A.~Anandkumar}, {\em Neural operator: {Graph} kernel network for partial differential equations}, preprint, arXiv:2003.03485, 2020.

\bibitem{li2021fourier}
{\sc Z.~Li, N.~B. Kovachki, K.~Azizzadenesheli, B.~Liu, K.~Bhattacharya, A.~Stuart, and A.~Anandkumar}, {\em Fourier neural operator for parametric partial differential equations}, in International Conference on Learning Representations, 2021.

\bibitem{lions1971optimal}
{\sc J.~L. Lions}, {\em Optimal Control of Systems Governed by Partial Differential Equations}, vol.~170, Springer, 1971.

\bibitem{lu2019deeponet}
{\sc L.~Lu, P.~Jin, G.~Pang, Z.~Zhang, and G.~E. Karniadakis}, {\em Learning nonlinear operators via {DeepONet} based on the universal approximation theorem of operators}, Nature Machine Intelligence, 3 (2021), pp.~218--229.

\bibitem{lu2021physics}
{\sc L.~Lu, R.~Pestourie, W.~Yao, Z.~Wang, F.~Verdugo, and S.~G. Johnson}, {\em Physics-informed neural networks with hard constraints for inverse design}, SIAM Journal on Scientific Computing, 43 (2021), pp.~B1105--B1132.

\bibitem{mittelmann2000solving}
{\sc H.~D. Mittelmann and H.~Maurer}, {\em Solving elliptic control problems with interior point and {SQP} methods: control and state constraints}, Journal of Computational and Applied Mathematics, 120 (2000), pp.~175--195.

\bibitem{mowlavi2023optimal}
{\sc S.~Mowlavi and S.~Nabi}, {\em Optimal control of {PDEs} using physics-informed neural networks}, Journal of Computational Physics, 473 (2023), p.~111731.

\bibitem{naldi2010mathematical}
{\sc G.~Naldi, L.~Pareschi, and G.~Toscani}, {\em Mathematical Modeling of Collective Behavior in Socio-Economic and Life Sciences}, Springer Science \& Business Media, 2010.

\bibitem{nelsen2021random}
{\sc N.~H. Nelsen and A.~M. Stuart}, {\em The random feature model for input-output maps between {Banach} spaces}, SIAM Journal on Scientific Computing, 43 (2021), pp.~A3212--A3243.

\bibitem{neyshabur2017exploring}
{\sc B.~Neyshabur, S.~Bhojanapalli, D.~McAllester, and N.~Srebro}, {\em Exploring generalization in deep learning}, in Advances in Neural Information Processing Systems, vol.~30, 2017.

\bibitem{paszke2019pytorch}
{\sc A.~Paszke, S.~Gross, F.~Massa, A.~Lerer, J.~Bradbury, G.~Chanan, T.~Killeen, Z.~Lin, N.~Gimelshein, L.~Antiga, A.~Desmaison, A.~Kopf, E.~Yang, Z.~DeVito, M.~Raison, A.~Tejani, S.~Chilamkurthy, B.~Steiner, L.~Fang, J.~Bai, and S.~Chintala}, {\em {PyTorch}: An imperative style, high-performance deep learning library}, in Advances in Neural Information Processing Systems, vol.~32, 2019.

\bibitem{raissi2019physics}
{\sc M.~Raissi, P.~Perdikaris, and G.~E. Karniadakis}, {\em Physics-informed neural networks: A deep learning framework for solving forward and inverse problems involving nonlinear partial differential equations}, Journal of Computational Physics, 378 (2019), pp.~686--707.

\bibitem{rockafellar1970convex}
{\sc R.~T. Rockafellar}, {\em Convex Analysis}, vol.~28, Princeton University Press, 1970.

\bibitem{schindele2017proximal}
{\sc A.~Schindele and A.~Borz{\`\i}}, {\em Proximal schemes for parabolic optimal control problems with sparsity promoting cost functionals}, International Journal of Control, 90 (2017), pp.~2349--2367.

\bibitem{sirignano2018dgm}
{\sc J.~Sirignano and K.~Spiliopoulos}, {\em {DGM}: A deep learning algorithm for solving partial differential equations}, Journal of Computational Physics, 375 (2018), pp.~1339--1364.

\bibitem{song2023accelerated}
{\sc Y.~Song, X.~Yuan, and H.~Yue}, {\em Accelerated primal-dual methods with enlarged step sizes and operator learning for nonsmooth optimal control problems}, preprint, arXiv:2307.00296, 2023.

\bibitem{song2024admm}
{\sc Y.~Song, X.~Yuan, and H.~Yue}, {\em The {ADMM}--{PINNs} algorithmic framework for nonsmooth {PDE}-constrained optimization: {A} deep learning approach}, SIAM Journal on Scientific Computing, 46 (2024), pp.~C659--C687.

\bibitem{stadler2009elliptic}
{\sc G.~Stadler}, {\em Elliptic optimal control problems with {$L^1$}-control cost and applications for the placement of control devices}, Computational Optimization and Applications, 44 (2009), pp.~159--181.

\bibitem{troltzsch1999optimal}
{\sc F.~Tr{\"o}ltzsch}, {\em On the {Lagrange}--{Newton}--{SQP} method for the optimal control of semilinear parabolic equations}, SIAM Journal on Control and Optimization, 38 (1999), pp.~294--312.

\bibitem{troltzsch2010optimal}
{\sc F.~Tr{\"o}ltzsch}, {\em Optimal Control of Partial Differential Equations: Theory, Methods, and Applications}, vol.~112, American Mathematical Society, 2010.

\bibitem{valkonen2014primal}
{\sc T.~Valkonen}, {\em A primal--dual hybrid gradient method for nonlinear operators with applications to {MRI}}, Inverse Problems, 30 (2014), p.~055012.

\bibitem{volkwein2000application}
{\sc S.~Volkwein}, {\em Application of the augmented {Lagrangian}-{SQP} method to optimal control problems for the stationary {Burgers} equation}, Computational Optimization and Applications, 16 (2000), pp.~57--81.

\bibitem{wang2021fast}
{\sc S.~Wang, M.~A. Bhouri, and P.~Perdikaris}, {\em Fast {PDE}-constrained optimization via self-supervised operator learning}, preprint, arXiv:2110.13297, 2021.

\bibitem{wang2021learning}
{\sc S.~Wang, H.~Wang, and P.~Perdikaris}, {\em Learning the solution operator of parametric partial differential equations with physics-informed {DeepONets}}, Science Advances, 7 (2021), p.~eabi8605.

\bibitem{weiser2004function}
{\sc M.~Weiser and A.~Schiela}, {\em Function space interior point methods for {PDE} constrained optimization}, Proceedings in Applied Mathematics and Mechanics, 4 (2004), pp.~43--46.

\end{thebibliography}

\end{document}